\documentclass[opre,nonblindrev]{informs3} %
\renewcommand{\theARTICLETOP}{}
\renewcommand{\theRRHSecondLine}{}
\renewcommand{\theLRHSecondLine}{}

\OneAndAHalfSpacedXI %

\usepackage{natbib}
 \bibpunct[, ]{(}{)}{,}{a}{}{,}%
 \def\bibfont{\small}%

\usepackage{booktabs}
\usepackage{graphicx}
\usepackage{multirow}
\usepackage{anyfontsize}
\usepackage{dsfont}
\usepackage{hyperref}
\usepackage{cleveref}
\usepackage{subcaption}
\usepackage{algorithm, algorithmicx, algpseudocode}
\usepackage{multibib}
\newcites{A}{Appendix References}
\usepackage{pdflscape}

\usepackage[showdeletions]{color-edits}
\addauthor[Chiwei]{cy}{cyan}
\addauthor[Junlin]{jc}{magenta}
\TheoremsNumberedThrough     %

\EquationsNumberedThrough    %

\begin{document}

\RUNAUTHOR{Chen, Yan, and Jiang} %

\RUNTITLE{On the LP for Dynamic Stochastic Matching and Its Application to Pricing}

\TITLE{On the Linear Programming Model for Dynamic Stochastic Matching and Its Application to Pricing}

\ARTICLEAUTHORS{%
\AUTHOR{Junlin Chen}
\AFF{Department of Industrial Engineering, Tsinghua University, \EMAIL{cjl21@mails.tsinghua.edu.cn}}
\AUTHOR{Chiwei Yan}
\AFF{Department of Industrial Engineering and Operations Research, University of California, Berkeley, \EMAIL{chiwei@berkeley.edu}}
\AUTHOR{Hai Jiang}
\AFF{Department of Industrial Engineering, Tsinghua University, \EMAIL{haijiang@tsinghua.edu.cn}}
} %

\ABSTRACT{
Important pricing problems in centralized matching markets---such as carpooling, food delivery and freight shipping platforms---often exhibit a bi-level structure. At the upper level, the platform sets prices for heterogeneous demand types (e.g., rides across origin-destination pairs, food delivery orders across restaurant-customer pairs, or less-than-truckload shipments). The lower level subsequently matches converted demands to minimize operational costs; for example, by pooling riders with similar itineraries into shared vehicles or consolidating multiple food-delivery and less-than-truckload freight orders into single courier or trailer routes. Motivated by these applications, we study the optimal value (cost) function of a linear programming model with respect to demand arrival rates, originally proposed by \citet{aouad2022dynamic} for cost-minimizing dynamic stochastic matching under limited time. In particular, we study the concavity properties of this cost function. We show that it suffices for every optimal basic feasible solution of the linear program to be nondegenerate in order to guarantee weak concavity. Leveraging this insight, we further establish that weak concavity holds when all demand types have strictly positive unmatched rates---a natural condition in stochastic environments when demands have limited patience---and characterize conditions under which this property is satisfied in the fluid linear program. Building on these theoretical insights, we develop a Minorization-Maximization (MM) algorithm that exploits the resulting difference-of-concave structure of the pricing problem. The algorithm requires little stepsize tuning and delivers substantial performance improvements over projected gradient methods on a large-scale, real-world ridesharing dataset with thousands of rider types (origin-destination pairs). This makes it a compelling algorithmic choice for solving such pricing problems in practice.

}


\maketitle

\section{Introduction}

Platforms that facilitate centralized dynamic matching are becoming increasingly prevalent in modern marketplaces. Ridesharing, food delivery, and freight logistics platforms all dynamically match stochastically arriving demand streams---such as riders, food orders, or less-than-truckload (LTL) shipments with compatible origins and destinations---into consolidated trips to minimize operational costs and maximize efficiency (see, e.g., \citealt{chen2024courier,yan2023pricing,ma2025potential}). Across these platforms, the matching process is constrained by ``patience'', such as passenger wait times, delivery guarantees, or shipping deadlines. If a compatible match is not secured before these limits expire, the platform must dispatch unmatched orders individually, resulting in higher marginal costs.

Building on the dynamic matching algorithms employed by these platforms, this research is motivated by the related \emph{pricing} challenges---specifically, the problem of determining upfront fares for demands of different types (e.g., riders with varying origins and destinations) to maximize total profit. \Cref{fig:screenshots} illustrates two real-world instances of this pricing challenge. \Cref{fig:screenshots_1} depicts a shared ride service on Uber, which offers a discounted fare relative to the UberX solo option. Similarly, \Cref{fig:screenshots_2} displays the DoorDash interface, where batched delivery is set as the default; customers must pay an additional premium for a ``Direct to you'' option to receive priority, non-batched service.

\begin{figure}[htbp]
    \centering
    \begin{subfigure}[b]{0.3\textwidth}
        \centering
        \includegraphics[width=\textwidth]{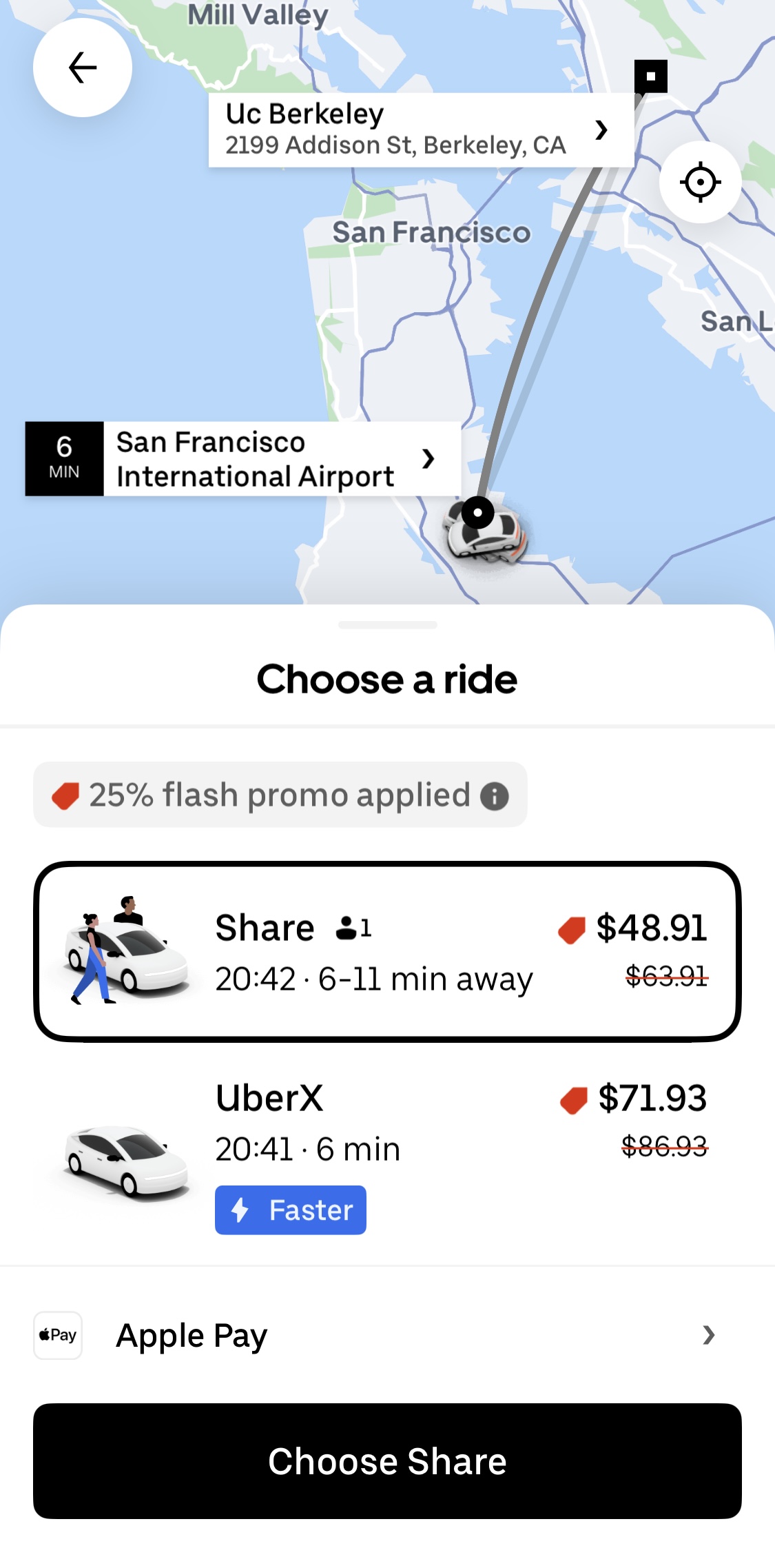}
        \caption{Shared ride}
        \label{fig:screenshots_1}
    \end{subfigure}
    \hspace{2cm}
    \begin{subfigure}[b]{0.3\textwidth}
        \centering
        \includegraphics[width=\textwidth]{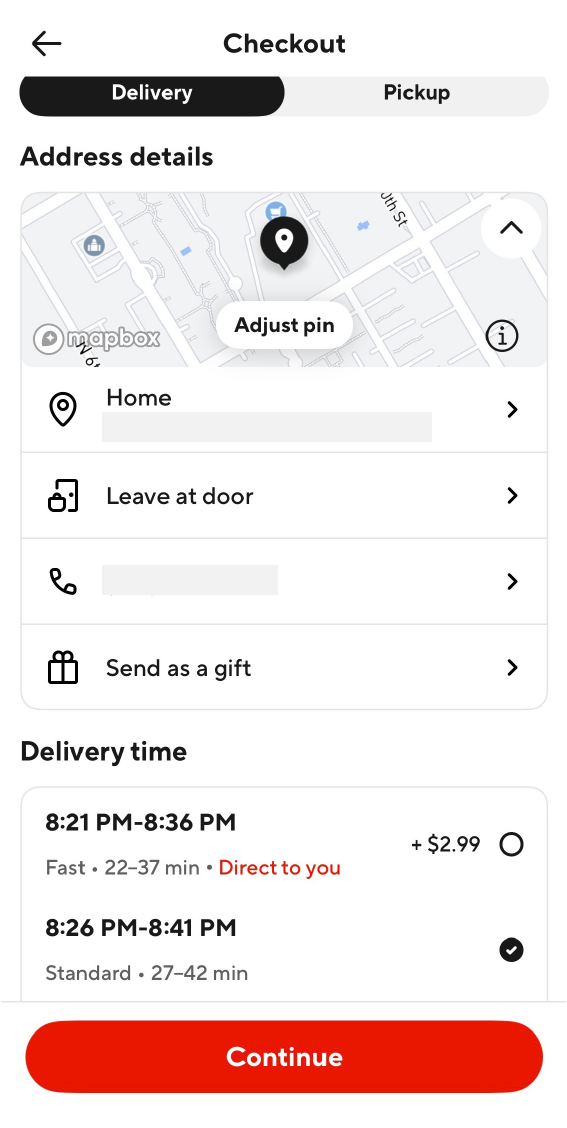}
        \caption{Online food delivery}
        \label{fig:screenshots_2}
    \end{subfigure}
    \caption{Examples of upfront pricing in dynamic matching platforms.}
\label{fig:screenshots}
\end{figure}

Pricing and matching are inherently interdependent: lowering prices can reduce revenue per match, but it also boosts arrival (conversion) rates and increases market thickness. Greater density, in turn, enables more efficient matching and lowers the cost per match. This naturally gives rise to a bi-level optimization problem. At the upper level, the platform determines optimal prices to maximize profit---often formulated as a concave revenue function minus total operational costs. Given the arrival (conversion) rates induced by these prices, the lower level executes dynamic matching to minimize costs. Under Markovian assumptions, the lower-level matching problem can often be formulated as a Markov Decision Process (MDP). In the MDP framework, the state of the system is defined as the number of demands of each type waiting in the system to be matched. When a demand arrives, the platform can (i) immediately match it with another existing demand, or (ii) keep it waiting in the system. However, the optimal policies in the MDP framework are, in general, intractable to obtain due to the curse of dimensionality. Therefore, many studies use fluid relaxations based on linear programs (LPs) to give bounds for the optimal value function \citep{collina2020dynamic, aouad2022dynamic, kessel2022,arnosti2025greedy}. 

In this paper, we utilize the cost-minimizing LP proposed by \citet{aouad2022dynamic} to model the lower-level matching cost, which provides a lower bound on the optimal cost of the MDP. While existing studies primarily focus on developing near-optimal matching policies using these LPs, our contribution lies in characterizing how the optimal cost function varies with the arrival rates. This is a crucial aspect of understanding and solving the upper-level pricing problem. To the best of our knowledge, this has not been addressed in prior work, where arrival rates are typically treated as fixed inputs. We now summarize our main results as follows.

\subsection{Overview of Our Results}

We begin with the extreme case of infinite patience, under which the pricing problem reduces to a well-behaved concave maximization problem (\Cref{prop:theta_zero}). With limited patience, however, characterizing the cost function becomes highly nontrivial because the arrival rates enter both the coefficients and the right-hand sides of the LP constraints, departing from classical parametric LP settings \citep{gal1972multiparametric, adler1992geometric}. Nevertheless, \Cref{prop:eos} reveals economies of scale in the cost function, which motivates our subsequent concavity analysis. Two key results, established in \Cref{lemma:n-weak-concave} and \Cref{prop:n-weak-concave-nonzero}, link weak concavity to nondegeneracy and strictly positive unmatched rates in the optimal LP solution. This connection allows us to certify weak concavity from partial optimal-solution structure, rather than via a full curvature analysis. The condition is also practically grounded, since strictly positive unmatched rates arise naturally in underlying dynamic stochastic matching systems with limited patience. 
Building on this result, we show that weak concavity is guaranteed when there are at most three demand types with a common patience level (\Cref{corl:n-weak-concave-same-theta}). More generally, for an arbitrary number of types, we find that (weak) concavity is more likely to hold when demands are either highly patient (\Cref{thm:n-weak-concave-same-theta}) or highly impatient (\Cref{prop:n-weak-concave-same-theta-upper-bound}), when patience levels are similar across types (\Cref{thm:n-weak-concave}), or when matching efficiencies are sufficiently heterogeneous (\Cref{thm:n-weak-concave-same-theta} and \Cref{thm:n-weak-concave}). We also discuss conditions under which the concavity of the cost function induced by the LP coincides with the concavity of the optimal value function of the corresponding MDP (\Cref{prop:n-concave-same-theta} and \Cref{prop:n-concave}). 

Although concavity of the cost function renders the pricing problem \eqref{equ:gn} nonconvex, it naturally leads to a \emph{difference-of-concave} formulation and motivates a Minorization-Maximization (MM) algorithm \citep{hunter2004tutorial}. A case study on the Chicago ridesharing dataset, involving thousands of rider types (origin-destination pairs), shows that MM substantially outperforms projected gradient methods, delivering better objective values with lower computational time. Moreover, whereas projected gradient methods are highly sensitive to stepsize choice, MM is robust across a wide range of settings and requires no stepsize tuning. This numerical stability, together with its fast convergence, makes MM a compelling algorithmic choice for solving such large-scale pricing problems in practice. We further show that our framework readily extends to several practical settings, including limited supply (e.g., a limited pool of drivers or couriers), multi-product variants (e.g., solo versus shared service), and the presence of matching disutility (e.g., detours and additional waiting time).

\subsection{Related Literature}

Our research is primarily related to two literature streams: (1) dynamic matching with abandonment, and (2) dynamic pricing and matching applications in online platforms.

\paragraph{Dynamic matching with abandonment.} 
The dynamic matching problem has attracted growing attention in recent literature. It concerns identifying optimal or near-optimal matchings among agents who arrive over time. Prior work studies this problem in either discrete-time or continuous-time settings, and LP-based fluid relaxations are often used to bound the optimal value function or the omniscient offline optimum, as well as to guide the design of online matching algorithms. Existing studies can also be categorized by whether matching is restricted to be bipartite. In what follows, we focus on the literature that studies the more general non-bipartite setting. For recent work in the bipartite setting, see \citet{hu2022dynamic}, \citet{kessel2022}, \citet{patel2024combinatorial}, and \citet{amanihamedani2025adaptive}.

In the discrete-time setting, \citet{gupta2024greedy}, \citet{kerimov2024dynamic}, and \citet{kerimov2025optimality} show that greedy-type algorithms achieve bounded regret under a \emph{general position gap} condition for the underlying static planning LP. This condition requires that the LP admit a nondegenerate optimal solution. Our weak-concavity conditions are related in spirit but differ in an essential way: our underlying LP is substantially more complex because arrival rates enter both the constraint coefficients and the right-hand sides, whereas in their formulations arrival rates appear only on the right-hand sides, reflecting the assumption that agents never abandon. See the discussion following \Cref{lemma:n-weak-concave} for details. A recent exception is \citet{eom2024batching}, who incorporate finite patience by assuming that agents depart after a fixed number of time steps. However, while their MDP model captures abandonment, their benchmark LP continues to assume infinite patience. 

Our work is related to the continuous-time literature initiated by \citet{akbarpour2020thickness}. In particular, \citet{collina2020dynamic, aouad2022dynamic, arnosti2025greedy} are aligned with our modeling approach, as they assume Poisson arrivals and exponential sojourn times. \citet{collina2020dynamic} and \citet{arnosti2025greedy} develop online algorithms with constant competitive ratios relative to the optimal offline (hindsight) policy, whereas \citet{aouad2022dynamic} provides guarantees relative to the optimal online policy. All three works derive fluid LP relaxations, but the formulation in \citet{aouad2022dynamic} bounds the optimal online policy and is therefore presumably tighter; moreover, it is the only one that explores a cost-minimization objective, which is most relevant for our pricing application. Accordingly, we adopt their formulation in this paper. It is important to note that none of the aforementioned works considers the upper-level pricing decision, which is the focus of our paper.

\paragraph{Applications of dynamic matching and pricing.}
Dynamic matching and its associated pricing problems have been studied extensively across a range of applications, particularly in ridesharing and on-demand delivery platforms. In ridesharing, \citet{ozkan2020dynamic} study bipartite matching between arriving passengers and available drivers and propose asymptotically optimal policies via a continuous linear program. \citet{feng2024two} further investigate joint pricing and bipartite matching, establishing a bounded competitive ratio. The recent work of \citet{yan2023pricing} is most closely related to ours, as it also considers non-bipartite matching and pricing for shared ride services. A key distinction is that \citet{yan2023pricing} focus on state-dependent dynamic pricing via an affine value function approximation, which is substantially more computationally intensive, whereas we study static upfront pricing through a fluid LP formulation. In on-demand delivery, \citet{ma2025potential} develop a potential-based greedy algorithm for efficient order pooling without considering pricing decisions. %

\smallskip
\paragraph{Organization}
The remainder of the paper is organized as follows. In \Cref{sec:prelim}, we introduce the model and preliminaries. In \Cref{sec:concavity}, we analyze concavity under homogeneous patience levels, and in \Cref{sec:diff_theta} we extend the analysis to heterogeneous patience. Applying these theoretical insights, \Cref{sec:pricing} develops an algorithmic approach to the price optimization problem, followed by several extensions in \Cref{sec:pricing_extension}. Finally, we conclude in \Cref{sec:conclusion}. All proofs, auxiliary results, and additional examples and computational experiments are deferred to the appendix.

\section{Preliminaries}
\label{sec:prelim}

Consider a system with $N$ distinct demand types, where $[N] = \{1, 2, \ldots, N\}$. Potential demands of type $i \in [N]$ are presented with a quoted price $p_i$, upon which they decide whether to join the platform. We assume that type-$i$ demand arrives according to an independent Poisson process with rate $\lambda_i > 0$, which is a function of the quoted price $p_i$. We assume a one-to-one correspondence between the arrival rate $\lambda_i$ and the quoted price $p_i$ for each demand type $i \in [N]$. This allows us to express the price as a function of the arrival rate, denoted by $p_i(\lambda_i)$. We also assume that the revenue function, $\lambda_i p_i(\lambda_i)$, is concave with respect to $\lambda_i$. Both assumptions are standard in the revenue management literature. 
\begin{assumption}
    There exists a one-to-one correspondence between $\lambda_i$ and $p_i$ for all $i \in [N]$. The function $\lambda_i p_i(\lambda_i)$ is concave with respect to $\lambda_i$.
\end{assumption}

Once admitted, a type-$i$ demand’s sojourn time in the system is exponentially distributed with rate $\theta_i$. Here, $\theta_i$ characterizes the \textit{patience} of type-$i$ demands---a smaller $\theta_i$ indicates a greater willingness to wait for a match.\footnote{The limited sojourn time can be interpreted either as agent patience or as the platform's willingness to wait for another matchable agent, which can enable it to extract higher profit with lower matching cost. A prominent example in ridesharing is Uber Express Pool, launched in 2018 \citep{uber_express}, where riders were asked to wait a few minutes (typically between $0$ and $5$ minutes) before their trips began, allowing the platform to form better matches with fewer detours and offer a more affordable shared-rides product.} The platform attempts to match demands before their sojourn times expire. Following \citet{huang2018match} and \citet{aouad2022dynamic}, when two demands are matched, we refer to the demand that arrives earlier as the \textit{active} demand and the one that arrives later as the \textit{passive} demand. When an active demand of type $i \in [N]$ is matched with a passive demand of type $j \in [N]$, a cost $c_{(i,j)} > 0$ is incurred. If a demand of type $i \in [N]$ remains unmatched when its patience runs out, a cost $c_{(i)} > 0$ is incurred. Consistent with \citet{yan2023pricing}, we make the following assumptions about these costs.
\begin{assumption} \label{assum:cost}
$c_{(i)} = c_{(i,i)} \leq c_{(i,j)} = c_{(j,i)},\ \forall i,j\in[N]$.
\end{assumption}

\Cref{assum:cost} states that when two demands of the same type $i \in [N]$ are matched, the resulting cost equals that of serving a single unmatched demand of type $i$. In ridesharing, food delivery, and freight logistics platforms, this means that matching two identical ODs incurs the same cost as serving either OD individually. The condition $c_{(i,j)} = c_{(j,i)}$ imposes symmetry and requires that the order of arrival---that is, which demand is active or passive---does not affect the matched cost. Additionally, $c_{(i,j)} \geq c_{(i)}$ requires that the cost of serving two matched demands is at least as great as serving a single demand alone. This is also a natural assumption in these applications, where a pooled trip must cover at least the distance required by one of the solo trips.

We also define the \textit{matching efficiency} between two types of demands as follows.

\begin{definition}
    The matching efficiency between demand types $i\in[N]$ and $j\in[N]$ is defined as $e_{i,j}=1-c_{(i,j)}/(c_{(i)}+c_{(j)})$.
\end{definition}

Matching efficiency quantifies the proportion of cost saved when two types of demands are matched. A low matching efficiency indicates greater heterogeneity between the demands---for example, when their OD pairs differ significantly in ridesharing. By \Cref{assum:cost}, $c_{(i,j)} \geq c_{(i)}$ and $c_{(i,j)} \geq c_{(j)}$, so $e_{i,j}$ is bounded above by $0.5$. Additionally, since $c_{(i,i)} = c_{(i)}$, we have $e_{i,i} = 0.5$ for all $i \in [N]$, implying that self-matching always yields the highest efficiency. The lower bound of $e_{i,j}$ can be negative when $c_{(i,j)} > c_{(i)} + c_{(j)}$, meaning that matching these two types of demands would result in higher costs than serving them individually---rendering such matches impractical in any scenario.

\subsection{Price Optimization Problem}

The platform’s objective is to maximize total profit by optimally setting prices and conducting matching for all demands. Given the one-to-one correspondence between $\lambda_i$ and $p_i$, the pricing problem can be formulated as finding the optimal vector $\boldsymbol{\lambda} = (\lambda_i)_{i \in [N]}$ that maximizes total profit:
\begin{align}
\max_{\boldsymbol{\lambda} \in \Lambda}~g(\boldsymbol{\lambda}) = \max_{\boldsymbol{\lambda} \in \Lambda}{\sum_{i\in[N]}~\lambda_ip_i(\lambda_i) - c(\boldsymbol{\lambda})} \label{equ:gn},
\end{align}
where $\Lambda$ denotes the feasible domain of $\boldsymbol{\lambda}$. In this study, we assume $\boldsymbol{\lambda}$ is defined as a compact interval $[\underline{\boldsymbol{\lambda}}, \overline{\boldsymbol{\lambda}}]$. The lower and upper bounds $\underline{\boldsymbol{\lambda}}, \overline{\boldsymbol{\lambda}}$ can be equivalently interpreted as minimum and maximum prices imposed on different demand types. For convenience, we assume $\underline{\boldsymbol{\lambda}} > \boldsymbol{0}$, where all vector inequalities in this paper are interpreted component-wise.

\begin{assumption} 
$\boldsymbol{\lambda}\in\Lambda=[\underline{\boldsymbol{\lambda}},\overline{\boldsymbol{\lambda}}]$, where $\underline{\boldsymbol{\lambda}} > \boldsymbol{0}$.
\end{assumption}

The term $c(\boldsymbol{\lambda})$ in \eqref{equ:gn} represents the optimal value (cost) function of the lower-level matching problem. The dynamic matching problem can be formulated as an MDP; we refer readers to Appendix~\ref{appendix:mdp} for the full formulation. In our study, instead of directly solving the MDP, we use the following LP to approximate $c(\boldsymbol{\lambda})$:
\begin{subequations} \label{equ:cn}
\begin{align} 
c(\boldsymbol{\lambda}) = \min_{\boldsymbol{x},\boldsymbol{y}} \quad& \sum_{i\in[N]} \sum_{j\in[N]} c_{(i,j)} x_{i,j} + \sum_{i\in[N]} c_{(i)} y_i \label{equ:cn-obj} \\
\text{s.t.} \quad & \sum_{j\in[N]} x_{j,i} + \sum_{j\in[N]} x_{i,j} + y_i = \lambda_i, & \forall i \in [N], \label{equ:cn-flow}\\
& \theta_i x_{i,j} \leq \lambda_j y_i, & \forall i,j \in [N], \label{equ:cn-bound}\\
& x_{i,j} \geq 0, & \forall i,j \in [N], \label{equ:cn-nonneg-x} \\
& y_i \geq 0, & \forall i \in [N], \label{equ:cn-nonneg-y}
\end{align}
\end{subequations}
where $\boldsymbol{x}:= (x_{i,j})_{i,j\in[N]}$ and $\boldsymbol{y}:= (y_{i})_{i\in[N]}$. 

Variable $x_{i,j}$ represents the match rate between type $i\in[N]$ active demand and type $j\in[N]$ passive demand, and $y_{i}$ represents the rate of type $i\in[N]$ demand unmatched. The objective function \eqref{equ:cn-obj} is the average total cost per unit of time. Constraints \eqref{equ:cn-flow} can be viewed as the flow balance equation, that is, in the steady state, the arrival rate of type-$i$ demand should be the summation of its active match rate, passive match rate, and unmatched rate. Constraints \eqref{equ:cn-bound} are the ratio constraints that give the lower bound of the unmatched rate $y_{i}$ in terms of $\lambda_j$, $x_{i,j}$, and $\theta_i$. Intuitively, this reflects that a type-$i$ demand may leave the system before the next type-$j$  demand arrives, with a probability of $\theta_i/(\theta_i + \lambda_j)$, thus $y_{i}$ should be a factor of at least $\theta_i/\lambda_j$ of the match rate $x_{i,j}$.
Lastly, \eqref{equ:cn-nonneg-x} and \eqref{equ:cn-nonneg-y} require that the variables are nonnegative.

This LP was originally introduced in \citet{aouad2022dynamic}, where it was shown that its optimal value serves as a lower bound on the expected cost achieved by the optimal matching policy. Correspondingly, using this LP, the optimal profit given in \eqref{equ:gn} provides an upper bound for the profit achievable by any dynamic matching policy. While \citet{aouad2022dynamic} treat $\boldsymbol{\lambda}$ as exogenously given, we instead view this LP as a parametric programming model in $\boldsymbol{\lambda} \in \Lambda$, and characterize how $c(\boldsymbol{\lambda})$ varies with $\boldsymbol{\lambda}$. This is a crucial aspect of understanding and solving the upper-level pricing problem \eqref{equ:gn}.

\subsection{(Weak) Concavity of the Cost Function}

We first observe that, as an extreme case, when demands have infinite patience (i.e., $\theta_i = 0$ for all $i \in [N]$), constraints \eqref{equ:cn-bound} become redundant. In this scenario, $\boldsymbol{\lambda}$ appears only on the right-hand side of constraint \eqref{equ:cn-flow}, and it is well known that $c(\boldsymbol{\lambda})$ becomes a convex, piecewise linear function \citep{adler1992geometric}. Moreover, under \Cref{assum:cost}, we can show that $c(\boldsymbol{\lambda})$ simplifies further to a linear form.

\begin{proposition} \label{prop:theta_zero}
    When $\theta_i=0$, $\forall i \in [N]$, $c(\boldsymbol{\lambda})$ is convex and linear on $\Lambda$.
\end{proposition}
Therefore, in this extreme case, the pricing problem \eqref{equ:gn} reduces to a convex optimization problem and can be efficiently solved.

However, when constraint \eqref{equ:cn-bound} is imposed, analyzing $c(\boldsymbol{\lambda})$ becomes significantly more challenging. In this case, $\boldsymbol{\lambda}$ also appears in the coefficients of constraint \eqref{equ:cn-bound}, making the LP fundamentally different from the classical parametric LP literature, where the parameter appears only on the right-hand side of the constraints \citep{gal1972multiparametric,adler1992geometric}. In fact, as shown in the following proposition, the cost function exhibits economies of scale, implying that $c(\boldsymbol{\lambda})$ tends to be concave rather than convex.

\begin{proposition}[Economies of scale]\label{prop:eos}
 For all $\alpha > 1$,  $\boldsymbol{\lambda}\in\Lambda$, if $\alpha\boldsymbol{\lambda} \in \Lambda$, then $c(\alpha\boldsymbol{\lambda)}/\| \alpha\boldsymbol{\lambda} \| \leq  c(\boldsymbol{\lambda})/\| \boldsymbol{\lambda} \|$,
where $\| \cdot \|$ is the Euclidean norm of the vector.
\end{proposition}

\Cref{prop:eos} states that the average cost per demand decreases when the arrival rates are scaled proportionally. This implies that $c(\boldsymbol{\lambda})$ tends to bend downward (or remain linear) as the arrival rate scales—a behavior typically associated with concavity rather than convexity. This property motivates us to conduct a more rigorous analysis of the concavity of $c(\boldsymbol{\lambda})$. Although concavity implies that the profit maximization problem \eqref{equ:gn} is not convex, it yields a concave-minus-concave structure that is useful to specific algorithmic techniques discussed later in \Cref{sec:pricing}.

Besides the notion of standard concavity, we also consider the notion of \textit{weak concavity} in this paper, a relaxed alternative defined below. See also Proposition 4.3 in \citet{vial1983strong} for an equivalent definition.

\smallskip

\begin{definition}[Weak concavity]
$c(\boldsymbol{\lambda})$ is said to be weakly concave on $\Lambda$ if there exists a finite value $\rho\geq0$, such that $c(\boldsymbol{\lambda}) - \rho \|\boldsymbol{\lambda}\|^2/2$ is concave on $\Lambda$.\footnote{Note that, even for bounded functions, the weakly concave property is nontrivial; for example, one can verify that $|x|$ (for $-1 \leq x \leq 1$) and $-\sqrt{x}$ (for $0 \leq x \leq 1$) are not weakly concave.}
\end{definition}

\smallskip

If $c(\boldsymbol{\lambda})$ is weakly concave, it can be written as the sum of a concave function, $c(\boldsymbol{\lambda}) - \rho \|\boldsymbol{\lambda}\|^2/2$, and a convex function, $\rho \|\boldsymbol{\lambda}\|^2/2$. This decomposition allows us to express the total profit as $g(\boldsymbol{\lambda}) = (\sum_{i\in[N]} \lambda_i p_i(\lambda_i) - \rho \|\boldsymbol{\lambda}\|^2/2) - (c(\boldsymbol{\lambda}) - \rho \|\boldsymbol{\lambda}\|^2/2)$, thereby restoring the concave-minus-concave structure (analogous to the structure obtained when $c(\boldsymbol{\lambda})$ is concave in the standard sense). When defining weak concavity, we restrict attention to the quadratic term
$\rho\|\boldsymbol{\lambda}\|^2/2 = (\rho/2)\sum_{i \in [N]} \lambda_i^2$,
rather than higher-order terms such as
$(\rho/2)\sum_{i \in [N]} \lambda_i^\alpha$ with $\alpha > 2$.
This restriction is without loss of generality since the feasible set $\Lambda$ is bounded.\footnote{To see this, consider the case when $c(\boldsymbol{\lambda})-\rho (\sum_{i\in[N]} \lambda_i^\alpha)/2$ is concave on $\Lambda$ for some $\alpha > 2$. One can verify that $c(\boldsymbol{\lambda})-
\hat{\rho} (\sum_{i\in[N]} \lambda_i^2)/2 = [c(\boldsymbol{\lambda})-\rho (\sum_{i\in[N]} \lambda_i^\alpha)/2] + [\rho (\sum_{i\in[N]} \lambda_i^\alpha)/2 - \hat{\rho} (\sum_{i\in[N]} \lambda_i^2)/2]$ is also concave on $\Lambda$ for $\hat{\rho} = \rho \alpha (\alpha-1) \max_{i\in[N]} \overline{\lambda}_i^{\alpha-2}$.}

\smallskip

Interestingly, our first result uncovers a structural connection between weak concavity and the nondegeneracy of optimal basic feasible solutions (BFSs).

\begin{lemma} \label{lemma:n-weak-concave}
    $c(\boldsymbol{\lambda})$ is weakly concave on $\Lambda$, if $\forall \boldsymbol{\lambda}\in\Lambda$, all optimal BFSs are nondegenerate.
\end{lemma}

The proof of \Cref{lemma:n-weak-concave} begins by establishing local weak concavity, showing that
$c(\boldsymbol{\lambda})$ is weakly concave in a sufficiently small neighborhood of any point
$\hat{\boldsymbol{\lambda}} \in \Lambda$.
Specifically, we examine the optimal basic feasible solutions (BFSs) of program~\eqref{equ:cn} at
$\hat{\boldsymbol{\lambda}}$.
Under the assumption that all optimal BFSs are nondegenerate, these solutions remain BFSs (with one remaining optimal), and no new basic solution can become optimal within a sufficiently small neighborhood of $\hat{\boldsymbol{\lambda}}$. Consequently, $c(\boldsymbol{\lambda})$ can be expressed as the pointwise minimum of the value functions induced by these BFSs. Moreover, we show that the value function associated with each such BFS is bounded and weakly concave. It follows that $c(\boldsymbol{\lambda})$ is locally weakly concave. Finally, we invoke \Cref{lemma:n-weak-concave-1} in Appendix~\ref{proof:n-weak-concave} to extend this local property to the global weak concavity of $c(\boldsymbol{\lambda})$.

We note that the condition imposed in \Cref{lemma:n-weak-concave} is closely related to several variants of the \emph{general position gap} (GPG) condition studied in \citet{kerimov2024dynamic,kerimov2025optimality,gupta2024greedy} and \citet{wei2023constant}. In \citet{kerimov2024dynamic,kerimov2025optimality}, the GPG condition requires the existence of a unique nondegenerate optimal solution, which is stronger than our assumption, as we allow multiple optimal solutions at a given $\boldsymbol{\lambda}$. In contrast, \citet{gupta2024greedy} permits multiple optimal basic feasible solutions and requires only one of them to be nondegenerate, which is weaker than our condition, since we require all optimal BFSs to be nondegenerate; these assumptions are further refined in \citet{wei2023constant}, which requires the dual problem to admit the same unique optimal solution throughout a neighborhood of any given $\boldsymbol{\lambda}$. A key distinction is that the linear programs studied in the aforementioned works assume infinite agent patience, so $\boldsymbol{\lambda}$ appears only on the right-hand side of the constraints, whereas in our setting $\boldsymbol{\lambda}$ enters the coefficients of constraint~\eqref{equ:cn-bound}, causing the dual feasible region itself to vary with $\boldsymbol{\lambda}$ and making the problem substantially more challenging.

Nevertheless, the following proposition shows that it is sufficient for a BFS to be nondegenerate if its solution has non-zero unmatched rates.
\begin{proposition} \label{prop:n-weak-concave-nonzero}
    A basic feasible solution to \eqref{equ:cn} is nondegenerate if $\boldsymbol{y}>\boldsymbol{0}$.
\end{proposition}

Given \Cref{lemma:n-weak-concave} and \Cref{prop:n-weak-concave-nonzero}, to establish the weak concavity of $c(\boldsymbol{\lambda})$, it suffices to identify conditions under which every optimal solution satisfies $\boldsymbol{y}^* > \boldsymbol{0}$.%
\footnote{The optimal solutions to \eqref{equ:cn} depend on $\boldsymbol{\lambda}$ and can be written as $\boldsymbol{x}^*(\boldsymbol{\lambda})=(x_{i,j}^*(\boldsymbol{\lambda}))_{i,j\in[N]}$ and $\boldsymbol{y}^*(\boldsymbol{\lambda})=(y_i^*(\boldsymbol{\lambda}))_{i\in[N]}$. For notational simplicity, we suppress the dependence on $\boldsymbol{\lambda}$ and denote them by $\boldsymbol{x}^*=(x_{i,j}^*)_{i,j\in[N]}$ and $\boldsymbol{y}^*=(y_i^*)_{i\in[N]}$ throughout the analysis.}
This observation allows us to bypass the complexity and potential nonsmoothness of $c(\boldsymbol{\lambda})$. Moreover, focusing on the case $\boldsymbol{y}^* > \boldsymbol{0}$ is also practically justified: the scenario $y_i^* = 0$ is an artifact of the fluid LP and cannot arise in the original MDP framework, since with limited demand patience (i.e., $\theta_i > 0$ for all $i \in [N]$), the unmatched rate is never exactly zero.

\section{Homogeneous Patience Levels} \label{sec:concavity}

In this section, we begin to analyze the (weak) concavity of the value function $c(\boldsymbol{\lambda})$. We first restrict our attention to the setting where all demands share a homogeneous patience level $\theta$ (i.e., $\theta_i=\theta$ for all $i\in[N]$). The case with heterogeneous patience levels is addressed in \Cref{sec:diff_theta}.

\subsection{Analysis of One or Two Types of Demands} \label{sec:n2}

In the simplest case where there is only one type of demand (i.e., $N=1$), the inequality in \eqref{equ:cn-bound} is always binding at optimality. Intuitively, to minimize costs, the platform must maximize the matching rate to prevent the accumulation of unmatched demands. This structural property yields closed-form expressions for $x_{1,1}^*$, $y_1^*$, as stated in the following lemma.
\begin{lemma} \label{lemma:n1-solution}
When $N=1$, the optimal solution to \eqref{equ:cn} is given by $y_1^*=\lambda_1\theta/(\theta+2\lambda_1)$ and $x_{1,1}^* = \lambda_1^2/(\theta+2\lambda_1)$.
\end{lemma}

Given that $c(\lambda_1) = c_{(1)} (x_{1,1}^* + y_1^*) = c_{(1)}\lambda_1(\theta+\lambda_1)/(\theta+2\lambda_1)$ is smooth in this setting, we can directly establish concavity via the second-order condition.

\begin{proposition} \label{prop:n1-concave}
 When $N=1$, $c(\lambda_1)$ is concave on $\Lambda$.
\end{proposition}
As implied by \Cref{lemma:n1-solution}, when the arrival rate $\lambda_1$ increases, the ratio $x_{1,1}^*/\lambda_1$ increases and asymptotically approaches $1/2$. This behavior reflects a market-thickening effect, whereby matching opportunities improve with demand density. Consequently, the marginal cost $\mathrm{d} c(\lambda_1) / \mathrm{d} \lambda_1$ decreases as $\lambda_1$ grows, establishing that $c(\lambda_1)$ is concave on $\Lambda$.

Next, we consider a system with two demand types ($N=2$) sharing a homogeneous patience level $\theta$. Because the two types can be cross-matched, the total cost becomes a sum of heterogeneous pairwise terms involving $c_{(1)}, c_{(2)}$, and $c_{(1,2)}$. This introduces complexity, rendering $c(\boldsymbol{\lambda})$ potentially non-smooth. To address this, we analyze the structure of the optimal BFS by identifying which constraints among \eqref{equ:cn-bound}--\eqref{equ:cn-nonneg-y} are binding as $\lambda_1$ and $\lambda_2$ vary. For ease of exposition, we define the following function: %
\begin{align*}
    \Delta_1(\lambda_1, \lambda_2) & := c_{(1)}\frac{\theta+\lambda_1}{\theta+2\lambda_1} + c_{(2)}\frac{\theta+\lambda_2}{\theta+2\lambda_2} - c_{(1,2)}.
\end{align*}
The following lemma shows that there exist two types of optimal BFSs, determined by the sign of $\Delta_1(\lambda_1, \lambda_2)$.
\begin{lemma} \label{lemma:n2-solution-same-theta} When $N=2$ and $\theta_1=\theta_2=\theta$, \Cref{tab:n2-solution-same-theta} gives an optimal solution to program \eqref{equ:cn}. In addition, all optimal solutions to \eqref{equ:cn} satisfy $y_1^*, y_2^*>0$.
\begin{table}[ht]
\scriptsize
\centering
\caption{Characterization of optimal solutions to program~\eqref{equ:cn} when $N=2$ and $\theta_1=\theta_2=\theta$.}
\label{tab:n2-solution-same-theta}
\resizebox{\textwidth}{!}{%
\begin{tabular}{llllllll}
\toprule
\multirow{2}{*}{Case} & \multirow{2}{*}{$\Delta_1(\lambda_1, \lambda_2)$} & \multicolumn{6}{c}{Optimal Solutions} \\ \cline{3-8} 
 & & $y_1^*$ & $y_2^*$ & $x_{1,1}^*$ & $x_{1,2}^*$ & $x_{2,1}^*$ & $x_{2,2}^*$ \\ \hline
(i) & $\leq 0$ & $ =\lambda_1 \theta/(\theta+2\lambda_1)$  & $ =\lambda_2 \theta/(\theta+2\lambda_2)$  &$=\lambda_1y_1^*/\theta$& $= 0$       & $= 0$       &$=\lambda_2y_2^*/\theta$\\
(ii) & $> 0$ & $ = \lambda_1 \theta/(\theta+2\lambda_1+2\lambda_2)$  & $= \lambda_2 \theta/(\theta+2\lambda_1+2\lambda_2)$  & $=\lambda_1y_1^*/\theta$       &$=\lambda_2y_1^*/\theta$&$=\lambda_1y_2^*/\theta$&$=\lambda_2 y_2^*/\theta$\\
\bottomrule
\end{tabular}
}
\end{table}
\end{lemma}
As detailed in \Cref{tab:n2-solution-same-theta}, when $\Delta_1(\lambda_1, \lambda_2) \leq 0$, the optimal solution suppresses cross-matching ($x_{1,2}^*=x_{2,1}^*=0$), relying entirely on self-matching. Conversely, when $\Delta_1(\lambda_1, \lambda_2) > 0$, all ratio constraints in \eqref{equ:cn-bound} bind, and the system utilizes all matching pairs.

To further interpret this lemma, note that $\Delta_1(\lambda_1, \lambda_2)$ can be equivalently written as
$$
\Delta_1(\lambda_1, \lambda_2)=c_{(1)}\left( e_{1,2} - \frac{\lambda_1}{\theta+2\lambda_1}\right) + c_{(2)} \left( e_{1,2} - \frac{\lambda_2}{\theta+2\lambda_2} \right),
$$
where $e_{1,2}$ is the efficiency of cross-matching. Since $\Delta_1(\lambda_1, \lambda_2)$ is increasing in $e_{1,2}$, the condition $\Delta_1(\lambda_1, \lambda_2) \leq 0$ is more likely to hold when $e_{1,2}$ is small. In such cases, cross-matching does not yield sufficient cost savings, making $x_{1,2}^*=x_{2,1}^*=0$ optimal. When $e_{1,2}$ is sufficiently large, both cross- and self-matching are efficient, and it becomes optimal to utilize both types of matching. Likewise, the condition $\Delta_1(\lambda_1, \lambda_2) \leq 0$ is more likely to hold for large arrival rates (large $\lambda_1, \lambda_2$) or high patience (small $\theta$). Intuitively, when demands have high arrival rates or are very patient, the platform can already achieve a significant extent of matching through self-matching alone, and thus cross-matching becomes unnecessary.

\Cref{lemma:n2-solution-same-theta} directly implies, via \Cref{prop:n-weak-concave-nonzero}, that
$c(\boldsymbol{\lambda})$ is weakly concave as all unmatched rates are nonzero.
Moreover, exploiting the explicit form of the optimal solution in
\Cref{lemma:n2-solution-same-theta} and the fact that pointwise minimization preserves concavity,
we can further show that $c(\boldsymbol{\lambda})$ remains concave when $N=2$, even though it is no longer smooth.

\begin{proposition} \label{prop:n2-concave-same-theta}
    When $N=2$ and $\theta_1=\theta_2=\theta$, $c(\boldsymbol{\lambda})$ is concave on $\Lambda$.
\end{proposition}

Moreover, we show that the optimal match rates characterized in
\Cref{lemma:n1-solution} and \Cref{lemma:n2-solution-same-theta}
can be implemented by feasible dynamic matching policies.
Specifically, when $N=1$, or when $N=2$ with $\Delta_1(\lambda_1,\lambda_2) > 0$,
a greedy policy that immediately matches any available pair is optimal. When $\Delta_1(\lambda_1,\lambda_2) \le 0$, the optimal policy still greedily performs
self-matching but prohibits cross-matching. 

\begin{proposition}\label{prop:n2-tight}
When $N \le 2$ and $\theta_i = \theta$ for all $i \in [N]$,
$c(\boldsymbol{\lambda})$ coincides
with the optimal value function of the corresponding MDP.
\end{proposition}

\subsection{Analysis of Multiple Types of Demands} \label{sec:nn}

While \Cref{prop:n1-concave} and \Cref{prop:n2-concave-same-theta} establish concavity in the two simple cases $N=1$ and $N=2$, these results are difficult to generalize because fully characterizing the optimal solutions and obtaining an explicit expression for $c(\boldsymbol{\lambda})$ become intractable as $N$ grows. Indeed, we can construct counterexamples demonstrating non-concavity when $N=3$ with homogeneous patience; see \Cref{example:non-concave-same-theta} for details. Consequently, for general $N$, we do not limit our analysis to standard concavity; we also examine the notion of \emph{weak} concavity. Specifically, our approach is to identify sufficient conditions under which the optimal unmatched rates satisfy $\boldsymbol{y}^* > \boldsymbol{0}$, thereby enabling the application of \Cref{prop:n-weak-concave-nonzero}.

We first introduce the concept of \textit{critical matching efficiency}. Given any type $i$, we sort its corresponding matching efficiencies $e_{i,1}, e_{i,2}, \cdots, e_{i,N}$ in descending order and denote the sorted result by $e_{i,(1)} \geq e_{i,(2)} \geq \cdots \geq e_{i,(N)}$, where by definition $e_{i,(1)}=e_{i,i}=0.5$. For $k\in \mathbb{Z}_{>0}$, we then define the $k$-th critical matching efficiency as
$$
e_{(k)} := 
\begin{cases}
\max\limits_{i \in [N]} \left\{ e_{i,(k)} \right\}, & \text{if } k \leq N, \\
0, & \text{if } k > N,
\end{cases}
$$
that is, when $k \leq N$, $e_{(k)}$ is the maximum $k$-th largest matching efficiency among all demand types. Since each demand has a self-matching efficiency of $0.5$, we always have $e_{(1)} = 0.5$. When $k \geq 2$, for any demand of type $i \in [N]$, there are at most $k-2$ other demand types (excluding self-match type $i$) with a matching efficiency greater than $e_{(k)}$.\footnote{In the case of homogeneous patience levels, we assume without loss of generality that $e_{(k)} < 0.5$ for $k\geq2$. This assumption is justified because any two types sharing a perfect matching efficiency (e.g., identical OD pairs) can be aggregated into a single type.} For notational convenience, we also define $e_{(k)} = 0$ for all $k > N$. {It can be easily verified that $e_{(k)}$ is non-increasing in $k\in \mathbb{Z}_{>0}$.} 

\smallskip

Based on the above definition, we now state the result concerning weak concavity.
\begin{theorem} \label{thm:n-weak-concave-same-theta} 
 When $\theta_i = \theta$, $\forall i\in[N]$, $c(\boldsymbol{\lambda})$ is weakly concave on $\Lambda$ if
$
 \underline{\lambda}_i > \theta {e_{(4)}}/({1-2e_{(4)}}),\ \forall i \in [N].
$
\end{theorem}
By imposing a lower bound on $\underline{\lambda}_i$, \Cref{thm:n-weak-concave-same-theta} ensures that weak concavity holds when demand arrival rates are sufficiently high (or equivalently, when price upper bounds are sufficiently low). This lower bound depends on $\theta$, the homogeneous patience level of demands, as well as $e_{(4)}$, the 4th critical matching efficiency.  Relating $\underline{\lambda}_i$ to $e_{(4)}$ allows us to prove that each demand type is matched with at most two other types in the optimal solutions, thereby guaranteeing the dominance of self-matching and strictly positive unmatched rates (i.e., $\boldsymbol{y}^* > \boldsymbol{0}$). 

From \Cref{thm:n-weak-concave-same-theta}, we immediately obtain the following corollary for the case $N \le 3$. When $N \le 3$, we have $e_{(4)} = 0$ by definition, and hence the right-hand side of the bound in \Cref{thm:n-weak-concave-same-theta} vanishes. As a result, weak concavity always holds when there are at most three demand types.

\begin{corollary} \label{corl:n-weak-concave-same-theta}
    When $N\leq3$ and $\theta_i = \theta$, $\forall i\in[N]$, $c(\boldsymbol{\lambda})$ is weakly concave on $\Lambda$.
\end{corollary}

While \Cref{thm:n-weak-concave-same-theta} establishes weak concavity, we can also identify a stronger condition that guarantees standard concavity, as well as the tightness of the LP.

\begin{proposition} \label{prop:n-concave-same-theta}
When $\theta_i = \theta$ for all $i \in [N]$, the function $c(\boldsymbol{\lambda})$ is concave on $\Lambda$ and coincides with the optimal value function of the corresponding MDP if
$
\underline{\lambda}_i > {\theta\, e_{(3)}}/({1 - 2 e_{(3)}}),\ \forall i \in [N].
$
\end{proposition}

\Cref{prop:n-concave-same-theta} provides a lower bound on $\underline{\lambda}_i$ under which standard concavity holds. This bound is structurally similar to that in \Cref{thm:n-weak-concave-same-theta}, but differs in its dependence on $e_{(3)}$ rather than $e_{(4)}$, thereby imposing a stricter requirement to guarantee standard concavity. The reliance on $e_{(3)}$ ensures that demand of type $i$ is matched with at most one other type, which allows the problem to be decomposed into a collection of two-type ($N=2$) subproblems. Consequently, the concavity result in \Cref{prop:n2-concave-same-theta} extends to this setting.

 The lower bounds in both \Cref{thm:n-weak-concave-same-theta} and \Cref{prop:n-concave-same-theta} suggest that (weak) concavity is easier to attain when demands exhibit a higher level of patience (i.e., a smaller $\theta$). As an extreme case, when $\theta_i = 0$ for all $i \in [N]$, concavity holds trivially, which is consistent with \Cref{prop:theta_zero} where the value function becomes linear. In addition, (weak) concavity is easier to achieve when demands display greater heterogeneity in matching efficiencies (i.e., smaller $e_{(4)}$ or $e_{(3)}$). Intuitively, both conditions promote the dominance of self-matching and limit the extent of cross-matching, a structural property that facilitates the proof of (weak) concavity.

\smallskip

This result also suggests that when matching efficiencies are high among many demand types (i.e., $e_{(3)}$ and $e_{(4)}$ are large), clustering similar types---specifically, those with high mutual matching efficiencies---provides a simple yet effective way to restore concavity. Aggregating types increases the effective arrival rates $\underline{\lambda}_i$ while simultaneously reducing the critical matching efficiency parameters $e_{(3)}$ and $e_{(4)}$.

Instead of fully merging demand types, an alternative approach is to aggregate only the ratio constraints \eqref{equ:cn-bound} while preserving granular, type-specific information. To this end, we partition $[N]$ into disjoint subsets $\mathcal{S} = \{S_1, \dots, S_K\}$ such that $[N] = \bigcup_{k=1}^K S_k$. For any $i \in [N]$, let $S(i) \in \mathcal{S}$ denote the unique subset containing $i$. We then define a modified linear program, which relaxes \eqref{equ:cn}, as follows:
\begin{subequations} \label{equ:cn-agg}
\begin{align} 
c(\boldsymbol{\lambda}; \mathcal{S}) = \min_{\boldsymbol{x},\boldsymbol{y}} \quad& \sum_{i\in[N]} \sum_{j\in[N]} c_{(i,j)} x_{i,j} + \sum_{i\in[N]} c_{(i)} y_i \\
\text{s.t.} \quad 
& \theta \sum_{j \in S(i)} x_{i,j} \leq y_i \sum_{j \in S(i)} \lambda_j , & \forall i \in [N], \label{equ:cn-bound-agg-1}\\
& \theta x_{i,j} \leq y_i \lambda_j , & \forall i \in [N], j \notin S(i), \label{equ:cn-bound-agg-2}\\
& \eqref{equ:cn-flow}, \eqref{equ:cn-nonneg-x}, \eqref{equ:cn-nonneg-y}. \notag
\end{align}
\end{subequations}
This LP is defined with respect to the partition $\mathcal{S}$. For each demand type $i \in [N]$, we aggregate the ratio constraints over all types $j$ that belong to the same subset $S(i)$. The objective and all other constraints remain unchanged. Consequently, this formulation yields a finer approximation than fully merging demand types, as it directly targets the subsets of constraints that most hinder concavity.

The following proposition establishes a sufficient condition for weak concavity of $c(\boldsymbol{\lambda}; \mathcal{S})$.
\begin{proposition}\label{prop:n-concave-agg} 
For any partition $\mathcal{S}$, $c(\boldsymbol{\lambda}; \mathcal{S})$ is weakly concave on $\Lambda$ if 
$$
\sum_{j \in S(i)} \underline{\lambda}_j > \theta \frac{\bar{e}}{1-2\bar{e}},\ \forall i \in [N],
$$
where $\bar{e} = \max_{i,j\in [N]:\ S(i) \neq S(j)} e_{i,j}$.
\end{proposition} 
In \Cref{prop:n-concave-agg}, the threshold $\bar{e}$ is the maximum matching efficiency between any two demand types that belong to different subsets of the partition. The bound highlights a clear trade-off: as the sets $S(i)$ become larger---thereby reducing $\bar{e}$ and increasing $\sum_{j \in S(i)} \underline{\lambda}_j$---the sufficient condition for weak concavity becomes easier to satisfy. In the extreme case where $\mathcal{S}=\{[N]\}$ consists of a single set, we have $\bar{e}=0$, and weak concavity holds unconditionally.

\smallskip

We finally turn to the regime in which $\theta$ is relatively large, corresponding to highly impatient demands. In this setting, weak concavity can still be established by imposing upper bounds on $\overline{\lambda}_i$, as formalized in the following proposition. 
\begin{proposition} \label{prop:n-weak-concave-same-theta-upper-bound} 
When $N \geq 4$ and $\theta_i = \theta$, $\forall i\in[N]$, $c(\boldsymbol{\lambda})$ is weakly concave on $\Lambda$ if 
$
\overline{\lambda}_i < {\theta}/({N-3}),\ \forall i \in [N].
$
\end{proposition}
As $\theta$ increases, this upper bound becomes less restrictive. Intuitively, when demands are highly impatient, a substantial fraction of arrivals exits the system unmatched, which naturally facilitates the condition $\boldsymbol{y}^* > \boldsymbol{0}$ (see \Cref{prop:n-weak-concave-nonzero}). We also refer the reader to \Cref{prop:n-weak-concave-same-theta-both-bound} in Appendix~\ref{sec:appendix_theory}, which establishes weak concavity by simultaneously imposing lower bounds on $\underline{\lambda}_i$ and upper bounds on $\overline{\lambda}_i$.

\section{Heterogeneous Patience Levels} \label{sec:diff_theta}

In this section, we extend our analysis of (weak) concavity to the setting in which demands exhibit heterogeneous patience levels. For clarity and without loss of generality, we assume throughout that $\theta_1 \le \theta_2 \le \cdots \le \theta_N$. We further define $\nu := \theta_N/\theta_1 > 1$, so that a larger $\nu$ corresponds to greater heterogeneity in patience levels.

\subsection{Analysis of Two Types of Demands}

Similar to \Cref{sec:n2}, we start with the simplest case where $N = 2$ and $\theta_1 < \theta_2$. We show that the optimal BFS to the LP can still be fully characterized, though the solution structure becomes much more complex than the homogeneous case presented in \Cref{lemma:n2-solution-same-theta}. We first generalize the definition of $\Delta_1(\lambda_1,\lambda_2)$ to the heterogeneous setting and define $\Delta_2(\lambda_1,\lambda_2)$ and $\Delta_3(\lambda_1,\lambda_2)$ as follows:
\begin{align*}
    \Delta_1(\lambda_1, \lambda_2) & := c_{(1)}\frac{\theta_1+\lambda_1}{\theta_1+2\lambda_1} + c_{(2)}\frac{\theta_2+\lambda_2}{\theta_2+2\lambda_2} - c_{(1,2)},\\
    \Delta_2(\lambda_1,\lambda_2) & := \frac{c_{(1)}}{2}\left[1 - \frac{\theta_1(\theta_2+\lambda_1+2\lambda_2)}{\lambda_2(\theta_2+2\lambda_2)}\right] +c_{(2)}\frac{\theta_2+\lambda_2}{\theta_2+2\lambda_2} - c_{(1,2)},\\
    \Delta_3(\lambda_1,\lambda_2) & := \lambda_1 - \lambda_2 - \theta_1.
\end{align*}
Then the following lemma describes the optimal solution to \eqref{equ:cn}.
\begin{lemma} \label{lemma:n2-solution} When $N=2$, an optimal solution to \eqref{equ:cn} is characterized by \Cref{tab:n2-solution}. In addition,

    (i) all optimal solutions to \eqref{equ:cn} satisfy $y_1^*>0$;
    
    (ii) all optimal solutions to \eqref{equ:cn} satisfy $y_2^*>0$ if and only if $\Delta_2(\lambda_1,\lambda_2) < 0$ or $\Delta_3(\lambda_1,\lambda_2) < 0$.
\begin{table}[ht]
\scriptsize
\centering
\caption{Characterization of optimal solutions to program~\eqref{equ:cn} when $N=2$.}
\label{tab:n2-solution}
\resizebox{\textwidth}{!}{
\begin{tabular}{llllllllll}
\toprule
\multirow{2}{*}{Case} & \multicolumn{3}{c}{Conditions}                                                                         & \multicolumn{6}{c}{Optimal Solutions}                                     \\ \cmidrule(lr){2-4}  \cmidrule(lr){5-10} 
                      & $\Delta_1(\lambda_1, \lambda_2)$ & $\Delta_2(\lambda_1, \lambda_2)$ & $\Delta_3(\lambda_1, \lambda_2)$ & $y_1^*$ & $y_2^*$ & $x_{1,1}^*$ & $x_{1,2}^*$ & $x_{2,1}^*$ & $x_{2,2}^*$ \\ \hline
(i)                   & $\leq 0$                         &  & & $ > 0$  & $ > 0$  &$=\lambda_1y_1^*/\theta_1$& $= 0$       & $= 0$       &$=\lambda_2y_2^*/\theta_2$\\
(ii)                  & $> 0$ & $< 0$ & & $ > 0$  & $ > 0$  &$=\lambda_1y_1^*/\theta_1$&$=\lambda_2y_1^*/\theta_1$&$=\lambda_1y_2^*/\theta_2$&$=\lambda_2y_2^*/\theta_2$\\
(iii)                 & & $\geq 0$ & $< 0$ & $ > 0$  & $ > 0$  & $= 0$       &$=\lambda_2y_1^*/\theta_1$&$=\lambda_1y_2^*/\theta_2$&$=\lambda_2y_2^*/\theta_2$\\
(iv)                  & & $\geq 0$ & $\geq 0$ & $ > 0$  & $= 0$   & $=(\lambda_1-y_1^*-x_{1,2}^*)/2$&$=\lambda_2y_1^*/\theta_1$& $= 0$ & $= 0$      \\
\bottomrule
\end{tabular}
}
\end{table}
\end{lemma}

\Cref{lemma:n2-solution} identifies four possible types of optimal BFSs, determined by the signs of three functions. Note that in \Cref{tab:n2-solution}, when $x_{i,j}^* = 0$, the corresponding constraint in \eqref{equ:cn-nonneg-x} is binding; when $x_{i,j}^\ast = \lambda_j y_i^\ast/\theta_i$, the corresponding constraint in \eqref{equ:cn-bound} is binding. The first solution category corresponds to the case with no cross-matching between type 1 and type 2 demands (i.e., $x_{1,2}^\ast= x_{2,1}^\ast = 0$). The second category is the fully-matched case, where all types of matchings occur and all constraints in \eqref{equ:cn-bound} are binding. In the third category, there is no self-matching between type 1 demands ($x_{1,1}^\ast = 0$). The last category is characterized by $y_2^\ast = 0$, where all type 2 demands (with low patience) are passively matched with type 1 demands (with high patience). The complete closed-form expressions of optimal solutions are provided in Appendix~\ref{proof:n2-solution}.

Although we can derive a closed-form expression for $c(\boldsymbol{\lambda})$, the formula becomes significantly intricate when $\theta_1 \neq \theta_2$, and thus establishing standard concavity becomes highly challenging; \Cref{example:non-concave-diff-theta} also demonstrates that $c(\boldsymbol{\lambda})$ is not generally concave in this case of $N=2$. Nonetheless, \Cref{lemma:n2-solution} still helps us to identify the region where $c(\boldsymbol{\lambda})$ is weakly concave. Based on the results of \Cref{lemma:n-weak-concave}, \Cref{prop:n-weak-concave-nonzero} and \Cref{lemma:n2-solution}, we see that if, for every $ (\lambda_1,\lambda_2)\in\Lambda$, at least one of the inequalities $\Delta_2(\lambda_1,\lambda_2) < 0$ or $\Delta_3(\lambda_1,\lambda_2) < 0$ holds, then $y_1^\ast,y_2^\ast > 0$ is guaranteed. The following proposition further explores this condition and gives an explicit characterization of $\Lambda$ satisfying such conditions.
\begin{proposition} \label{prop:n2-concave}
    Define 
    \begin{align*}
        \tau_1 &:= c_{(1)}(\theta_2-3\theta_1)+2c_{(2)}\theta_2-2c_{(1,2)}\theta_2, \\
        \tau_2 &:= \tau_1^2-8c_{(1)}(2c_{(1,2)}-c_{(1)}-c_{(2)})\theta_1(\theta_1+\theta_2),
    \end{align*}
    then $c(\boldsymbol{\lambda})$ is weakly concave on $\Lambda$ under any of the following conditions:
    
    (i) $\tau_1 \leq 0$ or $\tau_2<0$;
    
    (ii) $\tau_1>0$, $\tau_2 \geq 0$, $2c_{(1,2)}-c_{(1)}-c_{(2)} \neq 0$ and
    $
    \underline{\lambda}_2 > (\tau_1+\sqrt{\tau_2})/\left[4(2c_{(1,2)}-c_{(1)}-c_{(2)})\right];
    $

    {(iii) $\tau_1>0$, $\tau_2 \geq 0$, and $\underline{\lambda}_2 > \overline{\lambda}_1 - \theta_1$.}
\end{proposition}

\Cref{prop:n2-concave} provides three sufficient conditions for weak concavity, depending on the signs of $\tau_1$ and $\tau_2$. In the first case, when $\tau_1 \le 0$ or $\tau_2 < 0$, $c(\boldsymbol{\lambda})$ is weakly concave on any closed interval $\Lambda$. In the second case, weak concavity holds under a lower bound on $\underline{\lambda}_2$, and in the third case it requires restrictions on both $\underline{\lambda}_2$ and $\overline{\lambda}_1$. In general, smaller values of $\tau_1$ and $\tau_2$ are favorable for establishing weak concavity. Since both $\tau_1$ and $\tau_2$ decrease in $\theta_1$, a larger $\theta_1$ (equivalently, less heterogeneity because $\theta_1<\theta_2$) promotes weak concavity. Similarly, because $\tau_1$ and $\tau_2$ decrease in $c_{(1,2)}$, a larger $c_{(1,2)}$ (i.e., lower matching efficiency) also facilitates weak concavity.

\begin{remark}
    Although \Cref{prop:n2-concave} provides only sufficient conditions, we find that the resulting bounds are quite tight. In Appendix~\ref{sec:examples}, we give a counterexample (\Cref{example:non-weak-concavity}) demonstrating that as soon as $\lambda_2$ falls just below our derived bound in the second case, weak concavity is violated.
\end{remark}

\begin{remark}
In the second case of \Cref{prop:n2-concave}, we require $2c_{(1,2)}-c_{(1)}-c_{(2)} \neq 0$, or equivalently, $e_{1,2} \neq 0.5$. This condition means that the two demand types cannot have perfect matching efficiency. In practical applications such as ridesharing, this essentially precludes the possibility that two types of demands have exactly the same OD pairs. In Appendix~\ref{sec:examples}, we also provide a counterexample (\Cref{example:non-weak-concavity-special}) demonstrating that weak concavity is violated when $e_{1,2}=0.5$ and $\theta_1 \neq \theta_2$.     
\end{remark}

Given the complexity of the expressions for $\tau_1$ and $\tau_2$ in \Cref{prop:n2-concave}, we also state the following corollary, which provides a more intuitive sufficient condition in a simplified form.

\begin{corollary} \label{corl:n2-concave}
$c(\boldsymbol{\lambda})$ is weakly concave on $\Lambda$ under any of the following conditions:

(i) $\nu < 3$;

(ii) $\nu \geq 3$, $e_{1,2} \neq 0.5$ and 
    $
    \underline{\lambda}_2 > \theta_2/\left[4(1-2e_{1,2})\right]
    $;

{
(iii) $\nu \geq 3$, and $\overline{\lambda}_1 \leq \theta_1$.
}
\end{corollary}

According to \Cref{corl:n2-concave}, the function is guaranteed to be weakly concave whenever the disparity in patience levels is moderate ($\nu < 3$). However, if the demand types exhibit significantly different patience levels ($\nu \geq 3$), we require either a sufficient arrival rate for type 2 (i.e., a lower bound on $\underline{\lambda}_2$) or a limited arrival rate for type 1 (i.e., an upper bound on $\overline{\lambda}_1$). Regarding condition (ii), the lower bound on $\underline{\lambda}_2$ depends on the patience of type 2 (i.e., $\theta_2$) and the cross-matching efficiency (i.e., $e_{12}$). Consistent with \Cref{thm:n-weak-concave-same-theta}, a smaller $\theta_2$ (higher and less heterogeneous patience) or a smaller $e_{1,2}$ (lower matching efficiency) relaxes this bound, expanding the region of $\Lambda$ where weak concavity holds. Conversely, condition (iii) imposes an upper bound on $\overline{\lambda}_1$ relative to $\theta_1$. Analogous to \Cref{prop:n-weak-concave-same-theta-upper-bound}, this condition becomes easier to satisfy as $\theta_1$ increases. Intuitively, when type 1 (and type 2) demands are sufficiently impatient, weak concavity is also preserved.

\subsection{Analysis of Multiple Types of Demands}

For general $N$, the next theorem gives the results related to weak concavity.
{
\begin{theorem} \label{thm:n-weak-concave} 
$c(\boldsymbol{\lambda})$ is weakly concave on $\Lambda$  if $\nu \leq 2$, $e_{(3)} \neq 0.5$ and 
$
\underline{\lambda}_i > \theta_i {e_{(3)}}/({1-2e_{(3)}}),\ \forall i \in [N].
$
\end{theorem}

\Cref{thm:n-weak-concave} provides a sufficient condition for weak concavity for the regime where $\nu \leq 2$. The lower bound is analogous to the one in the homogeneous case (\Cref{thm:n-weak-concave-same-theta}), but now depends on $e_{(3)}$ rather than $e_{(4)}$. This implies that the conditions required to maintain weak concavity become stricter when patience levels are heterogeneous. 

We also state our results regarding standard concavity and the tightness of LP as follows.
\begin{proposition} \label{prop:n-concave} 
$c(\boldsymbol{\lambda})$ is concave on $\Lambda$ and coincides
with the optimal value function of the corresponding MDP if $e_{(2)} \neq 0.5$ and 
$
\underline{\lambda}_i > \theta_i {e_{(2)}}/({1-2e_{(2)}}),\ \forall i \in [N].
$
\end{proposition}

Analogous to \Cref{prop:n-concave-same-theta}, which depended on $e_{(3)}$, \Cref{prop:n-concave} imposes a lower bound based on the second critical matching efficiency, $e_{(2)}$. This stricter bound ensures that in the optimal solution, each demand type is matched exclusively within its own type (i.e., no cross-matching occurs). This effectively decomposes the general problem into independent single-type ($N=1$) subproblems, for which concavity was previously established in \Cref{prop:n1-concave}.

Consistent with our findings for the homogeneous case (\Cref{thm:n-weak-concave-same-theta} and \Cref{prop:n-concave-same-theta}), our derived bounds suggest that it is easier to achieve (weak) concavity when demands have a higher overall patience level or a greater heterogeneity in matching efficiencies. Furthermore, when patience levels are not exactly the same, the smaller difference (i.e., $\nu \leq 2$) also makes weak concavity easier to achieve.

Some discussions at the end of \Cref{sec:concavity} can also be extended to the setting with heterogeneous patience levels.
For instance, clustering demand types with high mutual matching efficiencies remains an effective strategy for restoring concavity: aggregating types can relax the sufficient conditions required by \Cref{thm:n-weak-concave} and \Cref{prop:n-concave}. Moreover, analogous to \Cref{prop:n-weak-concave-same-theta-upper-bound}, when the patience parameters $\theta_i$ are sufficiently large, weak concavity can be established by imposing upper bounds on $\overline{\lambda}_i$ rather than lower bounds on $\underline{\lambda}_i$. We formalize this observation in the following proposition. For more general cases, we refer the reader to \Cref{prop:n-weak-concave-both-bound} in Appendix~\ref{sec:appendix_theory}, which establishes weak concavity by simultaneously imposing lower and upper bounds.
\begin{proposition} \label{prop:n-weak-concave-upper-bound} 
When $(N-1)\nu>2$, $c(\boldsymbol{\lambda})$ is weakly concave on $\Lambda$ if 
$$
\overline{\lambda}_i < \theta_i \max\left\{\frac{1}{N-1}, \frac{1}{(N-1)\nu-2}\right\},\ \forall i \in [N].
$$
\end{proposition}

\section{Application to Pricing}
\label{sec:pricing}

In this section, we leverage the (weak) concavity of $c(\boldsymbol{\lambda})$ to design an efficient algorithm for solving the upper-level pricing problem. We first introduce the algorithm details in \Cref{subsec:pricing_mm}, and then conduct a case study using the Chicago ridesharing dataset in \Cref{subsec:chicago}. We also explore several extensions of the pricing problem, including limited supply (e.g., a limited pool of drivers or couriers), multi-product variants (e.g., solo versus shared service), and matching disutility (e.g., detours and additional waiting time), in \Cref{sec:pricing_extension}.

\subsection{A Minorization-Maximization Algorithm}
\label{subsec:pricing_mm}
For the pricing problem \eqref{equ:gn}, although the revenue function is assumed to be concave, $c(\boldsymbol{\lambda})$ is generally not convex, and thus \eqref{equ:gn} is not, in general, a convex optimization problem. Moreover, \Cref{example:multimodality} shows that $g(\boldsymbol{\lambda})$ is neither unimodal nor smooth. Nevertheless, as discussed in \Cref{sec:concavity}, $c(\boldsymbol{\lambda})$ often exhibits concavity or weak concavity. Leveraging this structure, we can express the objective $g(\boldsymbol{\lambda})$ as a difference of two concave functions:
\[
g(\boldsymbol{\lambda})
=
\left(\sum_{i\in[N]} \lambda_i p_i(\lambda_i) - \frac{1}{2}\rho \|\boldsymbol{\lambda}\|^2 \right)
-
\left( c(\boldsymbol{\lambda}) - \frac{1}{2}\rho \|\boldsymbol{\lambda}\|^2 \right),
\]
where $\rho \ge 0$ is chosen so that $c(\boldsymbol{\lambda}) - \rho\|\boldsymbol{\lambda}\|^2/2$ is concave. We can then adapt the Minorization-Maximization (MM) algorithm to compute a solution with a convergence guarantee.

At each step of the MM algorithm, we construct a surrogate linear function to replace the concave function $c(\boldsymbol{\lambda}) - \rho \|\boldsymbol{\lambda}\|^2/2$. 
Denote the value of $\boldsymbol{\lambda}$ after the $t$-th iteration by $\boldsymbol{\lambda}^{(t)}$. At the $(t+1)$-th iteration, the surrogate linear function is generated based on the supergradient of this concave function (analogous to the subgradient in convex functions; see \citealt{lu2014generalized}) at the point of $\boldsymbol{\lambda}^{(t)}$. Let $\boldsymbol{\gamma}^*=(\gamma_i^*)_{i\in[N]}$ and $\boldsymbol{\eta}^*=(\eta_{i,j}^*)_{i,j\in[N]}$ be the optimal dual variables corresponding to \eqref{equ:cn-flow} and \eqref{equ:cn-bound} respectively. (Note that under the nondegeneracy condition in \Cref{lemma:n-weak-concave} or the nonzero unmatched-rate condition in \Cref{prop:n-weak-concave-nonzero}, the dual optimal solution $(\boldsymbol{\gamma}^*, \boldsymbol{\eta}^*)$ is unique.) Then, according to the envelope theorem (see, e.g., Corollary 5 in \citealt{milgrom2002envelope}), the supergradient, denoted by $\boldsymbol{v}=(v_i)_{i\in[N]}$, can be given as
\begin{align} \label{equ:subgradient}
\boldsymbol{v}:=\left(\gamma_i^* + \sum_{j\in[N]} y_j^* \eta_{j,i}^* - \rho \lambda_i^{(t)}\right)_{i\in[N]} \in \partial \left( c(\boldsymbol{\lambda}^{(t)}) -  \frac{1}{2}\rho \|\boldsymbol{\lambda}^{(t)}\|^2\right),
\end{align}
and the surrogate linear function is given as
$$
c(\boldsymbol{\lambda}^{(t)}) - \frac{1}{2}\rho \|\boldsymbol{\lambda}^{(t)}\|^2+\boldsymbol{v}^T \left(\boldsymbol{\lambda}-\boldsymbol{\lambda}^{(t)}\right) \geq c(\boldsymbol{\lambda}) - \frac{1}{2}\rho \|\boldsymbol{\lambda}\|^2,\quad\forall\boldsymbol{\lambda}\in\Lambda,
$$
where the inequality holds because $c(\boldsymbol{\lambda}) - \rho \|\boldsymbol{\lambda}\|^2/2$ is a concave function.

Define $Q(\boldsymbol{\lambda} \mid \boldsymbol{\lambda}^{(t)}) = \sum_{i\in[N]} \lambda_ip_i(\lambda_i) - \rho \|\boldsymbol{\lambda}\|^2/2 - c(\boldsymbol{\lambda}^{(t)}) + \rho \|\boldsymbol{\lambda}^{(t)}\|^2/2-\boldsymbol{v}^T(\boldsymbol{\lambda}-\boldsymbol{\lambda}^{(t)})$. It can be easily verified that $Q(\boldsymbol{\lambda} \mid \boldsymbol{\lambda}^{(t)})$ is a concave function since the original non-concave component in $g(\boldsymbol{\lambda})$ has been substituted with a linear function. Therefore, we can update the value of $\boldsymbol{\lambda}$ by letting $\boldsymbol{\lambda}^{(t+1)} \gets \arg\max_{\boldsymbol{\lambda} \in \Lambda} Q(\boldsymbol{\lambda} \mid \boldsymbol{\lambda}^{(t)})$, where the maximization of $Q(\boldsymbol{\lambda} \mid \boldsymbol{\lambda}^{(t)})$ can be solved using standard convex optimization techniques. 

Note that
\begin{align*}
Q(\boldsymbol{\lambda} \mid \boldsymbol{\lambda}^{(t)}) 
& \leq \sum_{i\in[N]} \lambda_ip_i(\lambda_i) - \frac{1}{2}\rho \|\boldsymbol{\lambda}\|^2  - \left( c(\boldsymbol{\lambda}) - \frac{1}{2}\rho \|\boldsymbol{\lambda}\|^2\right) = g(\boldsymbol{\lambda}),\quad\forall\boldsymbol{\lambda}\in\Lambda,
\end{align*}
and
$
Q(\boldsymbol{\lambda}^{(t)} \mid \boldsymbol{\lambda}^{(t)}) = g(\boldsymbol{\lambda}^{(t)})
$, we have the following inequality:
$$
g(\boldsymbol{\boldsymbol{\lambda}}^{(t+1)}) \geq Q(\boldsymbol{\lambda}^{(t+1)} \mid \boldsymbol{\lambda}^{(t)}) \geq Q(\boldsymbol{\lambda}^{(t)} \mid \boldsymbol{\lambda}^{(t)}) = g(\boldsymbol{\lambda}^{(t)}),
$$
which guarantees that the objective function value is non-decreasing at each iteration. This monotonic improvement then leads to monotone convergence to a local optimum \citep{hunter2004tutorial}.

\begin{algorithm} 
\caption{Minorization-Maximization (MM) Algorithm %
} \label{alg:mm}
\begin{algorithmic}[1]
\State \textbf{Input:} Initial guess $\boldsymbol{\lambda}^{(0)} \in \Lambda$, $\varepsilon > 0$, $\delta_{\text{MM}} > 0$.
\State \textbf{Initialize:} $t \gets 0$
\Repeat
    \State Solve $c(\boldsymbol{\lambda}^{(t)})$, obtain optimal primal variables $\boldsymbol{y}^*=(y_i^*)_{i\in[N]}$, $\boldsymbol{x}^*=(x_{i,j}^*)_{i,j\in[N]}$ and dual variables $\boldsymbol{\gamma}^*=(\gamma_i^*)_{i\in[N]}$, $\boldsymbol{\eta}^*=(\eta_{i,j}^*)_{i,j\in[N]}$.
    \State $\rho \gets 0$
    \Repeat
    \State Calculate $\boldsymbol{v} = (v_i)_{i\in[N]}$, where
    $
    v_i= \gamma_i^* + \sum_{j\in[N]} y_j^* \eta_{j,i}^* - \rho \lambda_i^{(t)}.
    $
    \State Define $Q(\boldsymbol{\lambda} \mid \boldsymbol{\lambda}^{(t)}) = \sum_{i\in[N]} \lambda_ip_i(\lambda_i) - \rho \|\boldsymbol{\lambda}\|^2/2 - c(\boldsymbol{\lambda}^{(t)}) + \rho \|\boldsymbol{\lambda}^{(t)}\|^2/2- \boldsymbol{v}^T(\boldsymbol{\lambda}-\boldsymbol{\lambda}^{(t)})$.
    \State Update $\boldsymbol{\lambda}^{(t+1)} \gets \arg\max_{\boldsymbol{\lambda} \in \Lambda} Q(\boldsymbol{\lambda} \mid \boldsymbol{\lambda}^{(t)})$.
    \If{$g(\boldsymbol{\lambda}^{(t+1)}) < g(\boldsymbol{\lambda}^{(t)})$}
    \State Update $\rho \gets \rho + \delta_{\text{MM}}$.
    \EndIf
    \Until{$g(\boldsymbol{\lambda}^{(t+1)}) \geq g(\boldsymbol{\lambda}^{(t)})$}
    \State $t \gets t + 1$
\Until{$|g(\boldsymbol{\lambda}^{(t+1)}) - g(\boldsymbol{\lambda}^{(t)})| < \varepsilon$}
\State \textbf{Output:} $\boldsymbol{\lambda}^{(t)}$
\end{algorithmic}
\end{algorithm}

Summarizing the ideas above, we present the pseudocode of the MM algorithm in \Cref{alg:mm}. %
To specify the value of $\rho$, we slightly modify the standard MM algorithm by introducing a search step for $\rho$. 
Initially, we set $\rho = 0$ (Line 5) and perform one iteration to update $\boldsymbol{\lambda}$ (Lines 7--9). 
If the objective function does not increase after this iteration, it indicates that $c(\boldsymbol{\lambda}) - \rho \|\boldsymbol{\lambda}\|^2/2$ is not concave; we then increment $\rho$ by $\delta_{\text{MM}}$ and repeat the update.
If $c(\boldsymbol{\lambda})$ is weakly concave, it is guaranteed that after a finite number of increments, we will find a $\rho$ such that $c(\boldsymbol{\lambda}) - \rho \|\boldsymbol{\lambda}\|^2/2$ is concave, ensuring that the objective function $g(\boldsymbol{\lambda})$ increases after each iteration. By construction, each iteration yields an increase in the objective, and thus the algorithm is guaranteed to converge. It is worth noting that in the ridesharing case study presented in the next subsection, we \emph{never} observe the need to increase $\rho$; \Cref{alg:mm} consistently converges with $\rho = 0$, requiring \emph{no} stepsize tuning on real data. The same behavior is also observed across numerous synthetically generated instances. In contrast, alternative methods for solving~\eqref{equ:cn}, such as the projected gradient-based approaches, are highly sensitive to stepsize selection and require careful tuning, as will be demonstrated in \Cref{subsec:chicago}.

\subsection{Chicago Ridesharing Case Study}\label{subsec:chicago}

In this subsection, we conduct numerical experiments using the Chicago ridesharing dataset.\footnote{Code and data to reproduce the experiments can be accessed at \url{https://github.com/chen-jl99/pricing-dynamic-matching}.} The dataset, provided by Chicago Data Portal,\footnote{\url{https://data.cityofchicago.org/Transportation/Transportation-Network-Providers-Trips-2018-2022-/m6dm-c72p/about_data}} contains trips reported by ridesharing companies within the city of Chicago. The dataset divides the city (excluding the O'Hare International Airport) into 76 community areas, and each data point includes the passenger's OD pair, represented by the corresponding communities' centroid locations. We use typical rush-hour data (i.e., 7:30 AM--8:30 AM on Mondays), spanning 8 weeks between October and November 2019, prior to the COVID-19 pandemic, to conduct the experiment.

To aggregate riders into distinct types, we apply the $K$-means clustering algorithm to the origin-destination (OD) pairs of all trips, with the number of clusters set to $N \in \{100, 200,1000\}$.\footnote{In the dataset, there are in total 3042 distinct OD pairs without clustering.} For each cluster (i.e., rider type $i \in [N]$), we treat the cluster center as the representative OD for that type, and the average number of requests per hour within the cluster as the maximum possible hourly arrival rate $\overline{\lambda}_i$. %
Given the OD of each rider type, we compute the corresponding solo trip distances and pooled trip distances (in miles), denoted by $[\ell_{(i)}]_{i \in [N]}$ and $[\ell_{(i,j)}]_{i, j \in [N]}$, respectively. We let $c_{(i,j)} = c\ell_{(i,j)}$ and $c_{(i)}=c\ell_{(i)}$, where $c$ is the per-mile cost. We consider three possible values of $c\in\{0.7,0.9,1.1\}$ in the experiment, with a unit of dollar per mile. We assume in this section that the riders’ per-mile willingness to pay is uniformly distributed in $[0,1]$; consequently, we have 
$
p_i(\lambda_i) = \ell_{(i)}\left(1-\lambda_i/\overline{\lambda}_i\right)
$, $\forall i\in[N]$. {To evaluate the robustness of our results, we also test an alternative exponential demand model in Appendix~\ref{appendix:exponential_demand}.}

\smallskip

We compare the following two methods to solve the pricing problem \eqref{equ:gn}:

\begin{itemize}
    \item \textit{Projected Gradient (PG) method.} 
    We use PG as the baseline method. At each iteration, $\boldsymbol{\lambda}$ is updated along with the direction of the gradient while restricting within $\Lambda$. When $g(\boldsymbol{\lambda})$ is differentiable at $\boldsymbol{\lambda}$, its gradient is given by
\begin{align}
    \nabla g(\boldsymbol{\lambda}) = \left[ \ell_{(i)} \left(1 - \frac{2\lambda_i}{\overline{\lambda}_i}\right) - \gamma_i^* - \sum_{j \in [N]} y_j^* \eta_{j,i}^* \right]_{i \in [N]}, \label{equ:pg-gradient}
\end{align}
where $\boldsymbol{\gamma}^* = (\gamma_i^*)_{i \in [N]}$ and $\boldsymbol{\eta}^* = (\eta_{i,j}^*)_{i,j \in [N]} \leq \boldsymbol{0}$ are the optimal dual variables as defined in \Cref{alg:mm}. If $g(\boldsymbol{\lambda})$ is non-differentiable, we continue to use~\eqref{equ:pg-gradient} to perform the iteration, interpreting it as a subgradient or supergradient when such gradients exist. The initial step size $\delta_{\text{PG}}$ is chosen as $\{100, 10, 1\}$. We also dynamically adjust the step size, by checking whether the objective function increases after each iteration. If $g(\boldsymbol{\lambda}^{(t+1)}) < g(\boldsymbol{\lambda}^{(t)})$, we decrease the step size by half.
    \item \textit{Minorization-Maximization (MM) method.} This is our \Cref{alg:mm}. %
    Note that since we now have an explicit form of $p_i(\lambda_i)$, we can give a closed form solution of $\boldsymbol{\lambda}^{(t+1)} \gets \arg\max_{\boldsymbol{\lambda} \in \Lambda} Q(\boldsymbol{\lambda} \mid \boldsymbol{\lambda}^{(t)})$ (Line 9 of \Cref{alg:mm}) using its first-order condition:
$$
\lambda_i^{(t+1)} = \max\left\{\underline{\lambda}_i, \min\left\{\overline{\lambda}_i, \overline{\lambda}_i\frac{\ell_{(i)}- \left(\gamma_i^* + \sum_{j\in[N]} y_j^* \eta_{j,i}^*\right) + \rho \lambda_i^{(t)}}{2\ell_{(i)}+\rho\overline{\lambda}_i}\right\}\right\}.
$$
Besides, since we never observe the need to increase $\rho$ in our experiments---i.e., \Cref{alg:mm} consistently converges with $\rho = 0$---the choice of $\delta_{\text{MM}}$ in \Cref{alg:mm} is not relevant.

\end{itemize}

For both algorithms, we set $\underline{\lambda}_i = 10^{-3}$ for all $i \in [N]$. Initial solution $\boldsymbol{\lambda}^{(0)}$ is randomly sampled from the feasible set $\Lambda = [\underline{\boldsymbol{\lambda}}, \overline{\boldsymbol{\lambda}}]$ and kept the same across all algorithms within the same run. The stopping criterion is defined as $|g(\boldsymbol{\lambda}^{(t+1)}) - g(\boldsymbol{\lambda}^{(t)})| < \varepsilon = 10^{-3}$. We also track the running time of
each algorithm; if it exceeds 20 minutes, the algorithm terminates and returns the result from the latest iteration. We report and compare the following three metrics, averaged over three different random seeds:\footnote{The randomness comes from the initialization of $K$-means clustering, as well as the sampling of $\boldsymbol{\lambda}^{(0)}$ and $(\theta_i)_{i\in[N]}$.} (i) running time (in seconds);\footnote{The running time includes the time to build LP models in Gurobi.} (ii) number of iterations; and (iii) final objective value.

\begin{table}[t] \centering \scriptsize
\caption{Numerical results of solving \eqref{equ:gn} when $\theta_1=\cdots=\theta_N=\theta$.} \label{tab:results-same-theta}
\resizebox{\textwidth}{!}{%
\begin{tabular}{lclllllllllllll}
\toprule
\multicolumn{3}{c}{$c$ (\$/mile)} & \multicolumn{4}{c}{$0.7$} & \multicolumn{4}{c}{$0.9$} & \multicolumn{4}{c}{$1.1$} \\ \cmidrule(lr){4-7} \cmidrule(lr){8-11} \cmidrule(lr){12-15} 
\multicolumn{3}{c}{$\theta$ (min$^{-1}$)} & $1/5$ & $1/3$ & $1$ & $2$ & $1/5$ & $1/3$ & $1$ & $2$ & $1/5$ & $1/3$ & $1$ & $2$ \\ \hline
\multicolumn{1}{c}{\multirow{12}{*}{$N=100$}} & \multirow{4}{*}{\begin{tabular}[c]{@{}c@{}}Running\\ Time (sec)\end{tabular}} & PG ($\delta_{\text{PG}}=100$) & 10.2 & 9.9 & 9.5 & 10.4 & 9.6 & 9.9 & 10.1 & 10.4 & 27.1 & 21.3 & 14.2 & 12.2 \\
\multicolumn{1}{c}{} &  & PG ($\delta_{\text{PG}}=10$) & 8.5 & 7.9 & 8.6 & 9.6 & 7.5 & 7.6 & 8.8 & 8.5 & 17.8 & 13.9 & 12.0 & 10.0 \\
\multicolumn{1}{c}{} &  & PG ($\delta_{\text{PG}}=1$) & 57.8 & 58.9 & 44.2 & 51.1 & 48.6 & 45.7 & 47.6 & 44.5 & 26.2 & 24.1 & 37.8 & 20.2 \\
\multicolumn{1}{c}{} &  & MM & 1.7 & 1.7 & 1.7 & 1.8 & 1.9 & 2.0 & 2.3 & 2.1 & 3.0 & 4.2 & 5.6 & 6.4 \\ \cline{2-15} 
\multicolumn{1}{c}{} & \multirow{4}{*}{\begin{tabular}[c]{@{}c@{}}Num.\\ Iterations\end{tabular}} & PG ($\delta_{\text{PG}}=100$) & 30.3 & 29.3 & 30.0 & 31.0 & 30.3 & 29.7 & 29.7 & 31.3 & 85.3 & 64.0 & 43.7 & 34.0 \\
\multicolumn{1}{c}{} &  & PG ($\delta_{\text{PG}}=10$) & 24.7 & 23.3 & 27.0 & 28.3 & 22.3 & 22.0 & 25.7 & 25.3 & 53.3 & 41.3 & 34.3 & 27.0 \\
\multicolumn{1}{c}{} &  & PG ($\delta_{\text{PG}}=1$) & 179.3 & 180.0 & 133.7 & 155.3 & 148.3 & 139.0 & 144.7 & 140.7 & 79.3 & 75.0 & 112.0 & 55.7 \\
\multicolumn{1}{c}{} &  & MM & 4.3 & 4.0 & 4.0 & 4.3 & 4.3 & 5.0 & 6.0 & 5.3 & 8.0 & 13.0 & 16.3 & 17.3 \\ \cline{2-15} 
\multicolumn{1}{c}{} & \multirow{4}{*}{\begin{tabular}[c]{@{}c@{}}Objective\\ Function\end{tabular}} & PG ($\delta_{\text{PG}}=100$) & 64.89 & 59.70 & 46.47 & 38.48 & 40.08 & 34.37 & 21.15 & 14.30 & 11.35 & 10.60 & 3.85 & 1.44 \\
\multicolumn{1}{c}{} &  & PG ($\delta_{\text{PG}}=10$) & 64.89 & 59.70 & 46.46 & 38.47 & 40.08 & 34.37 & 21.15 & 14.30 & 12.27 & 10.12 & 3.95 & 1.68 \\
\multicolumn{1}{c}{} &  & PG ($\delta_{\text{PG}}=1$) & 64.71 & 59.56 & 46.18 & 38.33 & 39.96 & 34.23 & 21.01 & 13.97 & 17.38 & 13.60 & 5.06 & 1.60 \\
\multicolumn{1}{c}{} &  & MM & 64.89 & 59.70 & 46.47 & 38.48 & 40.08 & 34.37 & 21.15 & 14.30 & 19.97 & 14.68 & 5.13 & 1.73 \\ \hline
\multirow{12}{*}{$N=200$} & \multirow{4}{*}{\begin{tabular}[c]{@{}c@{}}Running\\ Time (sec)\end{tabular}} & PG ($\delta_{\text{PG}}=100$) & 38.3 & 39.0 & 46.8 & 55.2 & 39.0 & 40.9 & 52.7 & 63.5 & 184.3 & 251.1 & 88.0 & 77.7 \\
 &  & PG ($\delta_{\text{PG}}=10$) & 28.4 & 30.8 & 37.2 & 46.0 & 28.2 & 31.3 & 43.4 & 53.5 & 193.4 & 247.5 & 57.7 & 59.1 \\
 &  & PG ($\delta_{\text{PG}}=1$) & 279.9 & 302.2 & 136.3 & 155.4 & 108.6 & 124.2 & 168.4 & 228.2 & 253.6 & 159.1 & 76.8 & 89.6 \\
 &  & MM & 5.5 & 5.9 & 7.5 & 9.3 & 7.3 & 7.2 & 10.3 & 12.1 & 19.7 & 24.4 & 39.6 & 25.0 \\ \cline{2-15} 
 & \multirow{4}{*}{\begin{tabular}[c]{@{}c@{}}Num.\\ Iterations\end{tabular}} & PG ($\delta_{\text{PG}}=100$) & 33.0 & 32.3 & 32.3 & 31.7 & 32.7 & 32.0 & 33.3 & 34.0 & 137.0 & 164.3 & 46.0 & 36.3 \\
 &  & PG ($\delta_{\text{PG}}=10$) & 25.7 & 25.7 & 25.7 & 25.7 & 24.3 & 24.7 & 27.0 & 27.0 & 143.0 & 161.7 & 29.7 & 28.0 \\
 &  & PG ($\delta_{\text{PG}}=1$) & 268.0 & 269.7 & 96.3 & 85.3 & 99.7 & 99.7 & 103.3 & 109.0 & 193.7 & 111.7 & 37.3 & 38.0 \\
 &  & MM & 4.0 & 4.0 & 4.0 & 4.0 & 5.3 & 5.0 & 5.3 & 5.0 & 15.7 & 16.3 & 17.3 & 9.7 \\ \cline{2-15} 
 & \multirow{4}{*}{\begin{tabular}[c]{@{}c@{}}Objective\\ Function\end{tabular}} & PG ($\delta_{\text{PG}}=100$) & 60.42 & 55.13 & 44.02 & 37.69 & 34.65 & 29.37 & 19.21 & 13.68 & 12.27 & 10.11 & 4.42 & 1.64 \\
 &  & PG ($\delta_{\text{PG}}=10$) & 60.42 & 55.17 & 44.02 & 37.69 & 34.61 & 29.34 & 19.21 & 13.68 & 12.73 & 10.12 & 4.41 & 1.55 \\
 &  & PG ($\delta_{\text{PG}}=1$) & 60.13 & 54.88 & 43.77 & 37.47 & 34.39 & 29.11 & 18.94 & 13.41 & 14.26 & 10.85 & 4.03 & 1.26 \\
 &  & MM & 60.42 & 55.17 & 44.02 & 37.69 & 34.65 & 29.37 & 19.21 & 13.69 & 15.15 & 11.27 & 4.60 & 1.69\\ \hline
 \multirow{12}{*}{$N=1000$} & \multirow{4}{*}{\begin{tabular}[c]{@{}c@{}}Running\\ Time (sec)\end{tabular}} & PG ($\delta_{\text{PG}}=100$) & 1200 & 1200 & 1200 & 1200 & 1200 & 1200 & 1200 & 1200 & 1200 & 1200 & 1200 & 1200 \\
 &  & PG ($\delta_{\text{PG}}=10$) & 1200 & 1200 & 1200 & 1200 & 1200 & 1200 & 1200 & 1200 & 1200 & 1200 & 1200 & 1200 \\
 &  & PG ($\delta_{\text{PG}}=1$) & 1200 & 1200 & 1200 & 1200 & 1200 & 1200 & 1200 & 1200 & 1200 & 1200 & 1200 & 1200 \\
 &  & MM & 394.5 & 422.8 & 411.6 & 383.3 & 435.0 & 428.4 & 436.4 & 455.4 & 745.3 & 659.1 & 818.6 & 686.8 \\ \cline{2-15} 
 & \multirow{4}{*}{\begin{tabular}[c]{@{}c@{}}Num.\\ Iterations\end{tabular}} & PG ($\delta_{\text{PG}}=100$) & 15.0 & 15.7 & 14.0 & 13.7 & 15.0 & 15.0 & 14.3 & 14.3 & 14.7 & 13.7 & 14.7 & 14.3 \\
 &  & PG ($\delta_{\text{PG}}=10$) & 15.0 & 14.3 & 14.0 & 13.3 & 15.3 & 14.7 & 14.0 & 13.0 & 14.7 & 14.0 & 13.7 & 14.3 \\
 &  & PG ($\delta_{\text{PG}}=1$) & 14.7 & 14.0 & 14.0 & 13.0 & 15.0 & 14.7 & 14.3 & 15.0 & 14.7 & 14.7 & 15.7 & 18.0 \\
 &  & MM & 4.0 & 4.0 & 4.0 & 4.0 & 4.3 & 4.3 & 4.7 & 5.0 & 8.3 & 8.0 & 11.0 & 10.0 \\ \cline{2-15} 
 & \multirow{4}{*}{\begin{tabular}[c]{@{}c@{}}Objective\\ Function\end{tabular}} & PG ($\delta_{\text{PG}}=100$) & -166.61 & -93.60 & -167.52 & -122.04 & -245.08 & -260.43 & -136.56 & -166.61 & -203.38 & -171.90 & -204.54 & -140.68 \\
 &  & PG ($\delta_{\text{PG}}=10$) & -98.66 & -48.80 & -23.47 & -127.40 & -104.12 & -102.30 & -31.37 & -266.86 & -141.46 & -70.36 & -106.12 & -77.64 \\
 &  & PG ($\delta_{\text{PG}}=1$) & 21.38 & 34.04 & 27.04 & -68.93 & -21.31 & -10.76 & -3.22 & -1.65 & -4.51 & -3.68 & 1.54 & -0.86 \\
 &  & MM & 56.95 & 53.08 & 43.96 & 38.09 & 31.46 & 27.62 & 19.00 & 13.76 & 13.48 & 10.43 & 4.18 & 1.19\\
 \bottomrule
\end{tabular}
}
\end{table}

\Cref{tab:results-same-theta} presents the results of experiments where we assume $\theta_1 = \cdots = \theta_N = \theta \in \{1/5, 1/3, 1, 2\}$ per minute. Accordingly, the mean sojourn times for riders are $\{5, 3, 1, 0.5\}$ minutes, respectively. The results demonstrate that the MM algorithm substantially reduces both the running time and the number of iterations compared to the PG method. For $N \in \{100, 200\}$, all algorithms converge within the 20-minute limit. Compared to PG with three different step sizes $\delta_{\text{PG}} \in \{ 100, 10, 1 \}$, the MM algorithm reduces average running time by 81.5\%, 78.7\%, and 92.0\%, respectively, and reduces the average number of iterations by 83.1\%, 79.7\%, and 93.9\%. Moreover, the MM algorithm achieves a higher objective value---on average, 2.78\%, 2.61\%, and 1.44\% higher than the PG methods. When $N = 1000$, the PG algorithms fail to converge within 20 minutes.

Besides, we also observe that the performance of PG is sensitive to step size selection.
Although a step size of 100 or 10 yields a better average performance when $N \leq 200$, 
it is not optimal in every instance; for example, a step size of 1 performs better when $c=1.1$ and $\theta \leq 1/3$, or when $N = 1000$. In contrast, MM demonstrates strong robustness and outperforms PG across all cases and step-size rules.

It can also be observed that as the number of types $N$ increases, the running time of both PG and MM increases, primarily because solving $c(\boldsymbol{\lambda})$ involves $N + N^2$ variables at each iteration. However, the number of iterations required by the MM algorithm does not scale proportionally with $N$, and in many cases remains nearly constant at a single-digit value. Even when $N = 1000$, the MM algorithm remains efficient and consistently produces results within a short CPU time.

Appendix~\ref{appx:additional_exp} contains additional experiments. 
In Appendix~\ref{appendix:different_theta}, we present results when $\theta_i$ differs. 
In Appendix~\ref{appendix:exponential_demand}, we report results under an alternative exponential demand model. In Appendix~\ref{appendix:simulation}, we run simulations by implementing the optimal prices obtained from our MM algorithm under a dynamic matching policy, and compare them with a baseline pricing scheme that ignores finite demand patience (i.e., where constraint~\eqref{equ:cn-bound} is relaxed). We find that our pricing approach consistently outperforms the baseline.

\section{Extensions}
\label{sec:pricing_extension}

In this section, we consider several practical extensions of the pricing problem \eqref{equ:gn} that commonly arise in real-world operations. In particular, we study the following three extensions: (i) pricing with limited supply (e.g., limited pool of drivers or couriers); (ii) multi-product pricing, in which the platform simultaneously offers differentiated services (e.g., solo versus shared rides in ridesharing, or a premium direct-delivery option in food delivery); and (iii) pricing in the presence of match-related disutility (e.g., detour or additional waiting time).

\subsection{Pricing with Limited Supply}

We first consider a setting in which the platform operates under limited supply. Let $L$ denote the total supply capacity (e.g., the number of available drivers or couriers). By Little's Law, the system must satisfy the following constraint:
$$
\sum_{i\in[N]} \sum_{j\in[N]} t_{(i,j)} x_{i,j} + \sum_{i\in[N]} t_{(i)} y_i \leq L,
$$
where $t_{(i,j)}$ denotes the service duration for a matched pair of types $i$ and $j$, and $t_{(i)}$ denotes the duration for a solo service of type $i$. Incorporating this constraint into the lower-level matching problem yields the capacitated cost function:
\begin{subequations} \label{equ:cn-capacity}
\begin{align} 
\tilde{c}(\boldsymbol{\lambda}) = \min_{\boldsymbol{x},\boldsymbol{y}} \quad& \sum_{i\in[N]} \sum_{j\in[N]}  c_{(i,j)} x_{i,j} + \sum_{i\in[N]}  c_{(i)} y_i\\
\text{s.t.} \quad & \sum_{i\in[N]} \sum_{j\in[N]} t_{(i,j)} x_{i,j} + \sum_{i\in[N]} t_{(i)} y_i \leq L, \label{equ:cn-capacity-constr}\\
& \eqref{equ:cn-flow}- \eqref{equ:cn-nonneg-y}, \notag
\end{align}
\end{subequations}
where we define $\tilde{c}(\boldsymbol{\lambda})=+\infty$ if the program is infeasible. The upper-level pricing problem is then formulated as
$
\max_{\boldsymbol{\lambda} \in \Lambda}{\sum_{i\in[N]}~\lambda_ip_i(\lambda_i) - \tilde{c}(\boldsymbol{\lambda})}
$.

In many transportation and logistics applications, such as ridesharing, food delivery, and LTL freight shipping, operational costs are proportional to the trip length. Following the notation in \Cref{subsec:chicago}, we assume that $c_{(i,j)} = c\ell_{(i,j)}$ and $c_{(i)}=c\ell_{(i)}$, where $\ell_{(i,j)}$ and $\ell_{(i)}$ represent trip distance and $c$ is the unit cost per distance. Assuming a uniform speed $v$, the service times are given by $t_{(i,j)} = \ell_{(i,j)}/v = c_{(i,j)}/(cv)$ and $t_{(i)} = \ell_{(i)}/v = c_{(i)}/(cv)$. Under these assumptions, the capacity constraint \eqref{equ:cn-capacity-constr} is mathematically equivalent to a budget constraint on the operational cost---the optimal cost must not exceed $cvL$. Consequently, rather than complicating the lower-level model, we retain the original cost function $c(\boldsymbol{\lambda})$ and impose the capacity constraint directly in the upper-level pricing problem:
\begin{subequations}
\label{equ:gn-capacity}
\begin{align}
\max_{\boldsymbol{\lambda} \in \Lambda}\quad&{\sum_{i\in[N]}~\lambda_ip_i(\lambda_i) - c(\boldsymbol{\lambda})} , \\
\text{s.t.} \quad & c(\boldsymbol{\lambda}) \leq cvL.
\end{align}
\end{subequations}
To solve \eqref{equ:gn-capacity}, we adopt a Lagrangian relaxation approach:
\begin{align} \notag%
\min_{\mu \geq 0}\  \sup_{\boldsymbol{\lambda} \in \Lambda}\quad&{\sum_{i\in[N]}~\lambda_ip_i(\lambda_i) - (1+\mu)c(\boldsymbol{\lambda})} + \mu  cv L.
\end{align}
For a fixed dual variable $\mu \geq 0$, the inner maximization retains the concave-minus-concave structure and can be solved using the MM algorithm (\Cref{subsec:pricing_mm}). The optimal multiplier $\mu^*$ can then be determined via a one-dimensional line search.

\subsection{Multi-Product Pricing}

We now consider another variant of our problem, where the platform provides two or more products. As shown in \Cref{fig:screenshots}, the multi-product setting is common in applications such as ridesharing and online food delivery. For ridesharing, riders are allowed to choose from multiple service options, some options (e.g., UberX Share) offer economy shared rides, whereas some options (e.g., UberX and UberXL) offer solo rides. For online food delivery, customers can choose between the default option that allows batched delivery and the prioritized ``direct to you'' option.

Denote by $M$ the total number of products provided by the platform. The platform now quotes a price vector $\boldsymbol{p}_i = (p_{im})_{m\in[M]}$ to potential demands of type $i\in[N]$, where $p_{im}$ is the price for the $m$-th product, resulting in arrival rates $\boldsymbol{\lambda}_i =(\lambda_{im})_{m\in[M]}$. Analogous to \Cref{sec:pricing}, we make the assumptions on the one-to-one correspondence between $\boldsymbol{\lambda}_i$ and $\boldsymbol{p}_i$, as well as the concavity of the revenue function with respect to the arrival rate. Such assumptions fit common multi-product choice models such as Multinomial Logit and Nested Logit \citep{li2011pricing}.

The multi-product pricing problem then reduces to finding the optimal vector $\boldsymbol{\lambda} = (\lambda_{im})_{i \in [N], m\in[M]}$ that maximizes total profit, formulated as
\begin{align}
\max_{\boldsymbol{\lambda} \in \Lambda}{\sum_{i\in[N]}\sum_{m\in[M]}~\lambda_{im} p_{im}(\boldsymbol{\lambda}_i) - \sum_{m\in[M]}c_m(\boldsymbol{\lambda}_{\cdot m})},\notag %
\end{align}
where $\boldsymbol{\lambda}_{\cdot m} = (\lambda_{im})_{i\in[N]}$ and $c_m(\boldsymbol{\lambda}_{\cdot m})$ represents the total cost of providing the $m$-th product. Specifically, when the $m$-th product provides only solo services, $c_m(\boldsymbol{\lambda}_{\cdot m})$ is linear; when the $m$-th product provides shared services, $c_m(\boldsymbol{\lambda}_{\cdot m})$ follows the formulation \eqref{equ:cn} and the (weak) concavity results still hold. Therefore, the multi-product pricing problem still preserves a concave-minus-concave structure, and the MM algorithm remains applicable.

\subsection{Pricing with Matching Disutility}

As a final extension, we consider a scenario where the demand entry decision depends not only on the quoted price $p_i$ but also on the disutility associated with matching. This effect is prevalent in real-world applications. For instance, ridesharing passengers often perceive disutility from detours required to accommodate pooled riders, while food delivery customers may experience dissatisfaction due to additional waiting times caused by order batching.

Let $\sigma_{i,j}$ denote the normalized, unit-free measure of disutility incurred by a type-$i$ demand when matched with a type-$j$ demand. The expected normalized disutility for a type-$i$ demand is then given by $\sum_{j\in[N]} (x_{i,j}+x_{j,i})\sigma_{i,j}/\lambda_i$, where $x_{i,j}$ represents the match rate between type-$i$ active demand and type-$j$ passive demand under the optimal matching policy. Assuming a linear demand function and following the notation in \Cref{subsec:chicago}, the inverse demand function is given by:
$$
p_i(\lambda_i) = \ell_{(i)}\left[1-\frac{\sum_{j\in[N]} (x_{i,j}+x_{j,i})\sigma_{i,j}}{\lambda_i}-\frac{\lambda_i}{\overline{\lambda}_i}\right], \quad \forall i\in[N].
$$
Substituting this into the objective yields the total profit function:
\begin{align*}
& \sum_{i\in[N]}~\lambda_ip_i(\lambda_i) - \left[ \sum_{i\in[N]} \sum_{j\in[N]} c_{(i,j)} x_{i,j} + \sum_{i\in[N]} c_{(i)} y_i\right] \\
= & \sum_{i\in[N]}~\ell_{(i)}\lambda_i\left[1-\frac{\lambda_i}{\overline{\lambda}_i}\right] - \left[ \sum_{i\in[N]} \sum_{j\in[N]} (c_{(i,j)} + \ell_{(i)} \sigma_{i,j} + \ell_{(j)} \sigma_{j,i}) x_{i,j} + \sum_{i\in[N]} c_{(i)} y_i\right],
\end{align*}
where $y_i$ is the rate of type-$i$ demand unmatched under the optimal matching policy.

This transformed objective preserves a similar structure to that in \eqref{equ:gn} and \eqref{equ:cn}. Consequently, the MM algorithm remains applicable by simply solving the problem under an adjusted cost structure, where the effective matching cost is defined as $\tilde{c}_{(i,j)} = c_{(i,j)} + \ell_{(i)} \sigma_{i,j} + \ell_{(j)} \sigma_{j,i}$ and the solo cost remains $\tilde{c}_{(i)} = c_{(i)}$.

\section{Concluding Remarks}
\label{sec:conclusion}

Motivated by the growing use of pricing in centralized matching markets, we study the value function of the linear program introduced by \citet{aouad2022dynamic} for cost-minimizing dynamic stochastic matching, with a particular focus on its (weak) concavity with respect to demand arrival rates. Our main result shows that weak concavity is guaranteed if every optimal basic feasible solution of the LP is nondegenerate. Building on this result, we further establish that weak concavity holds whenever all demand types have strictly positive unmatched rates. Leveraging this insight, we show that weak concavity is guaranteed when there are at most three demand types with a common patience level. More generally, for an arbitrary number of types, we find that (weak) concavity is more likely to hold when demands are either highly patient or highly impatient, when patience levels are similar across types, or when matching efficiencies are sufficiently heterogeneous.

We then turn our attention to the pricing application and design an efficient MM algorithm that requires little stepsize tuning. To evaluate its performance, we apply our MM algorithm to the Chicago ridesharing dataset. Experimental results demonstrate that, compared to a projected gradient-based method whose performance is highly sensitive to stepsize selection, our approach significantly reduces both running time and the number of iterations, while simultaneously achieving better optimality. Several extensions of this pricing application are also discussed.

Our final remark is that the nonzero unmatched-rate conditions required in \Cref{prop:n-weak-concave-nonzero} for weak concavity arise naturally in dynamic stochastic matching systems when agents have limited patience. Indeed, the counterexamples that violate weak concavity in Appendix~\ref{sec:examples} no longer apply once one considers the value function of the MDP. It is therefore reasonable to conjecture that the optimal value function of the corresponding MDP is always weakly concave. Nevertheless, computing the exact optimal value function (and its supergradient with respect to the arrival rates) remains subject to the curse of dimensionality. Our theoretical and empirical results provide evidence that highly effective and scalable price optimization algorithms can be developed by using the LP value as a proxy for the true matching cost.

\ACKNOWLEDGMENT{%
The authors thank Ying Cui from UC Berkeley for helpful discussions. Chen acknowledges the support of the INFORMS Transportation Science and Logistics Society. Yan also acknowledges the support of the National
Science Foundation under Award Number 2517861.
}%

\begin{APPENDICES}
\section{MDP Formulation}
\label{appendix:mdp}

In this section, we formulate the problem of finding an optimal matching policy under a given price vector $\boldsymbol{\lambda}$ as an average-cost, continuous-time Markov Decision Process (MDP). Let the state of the system at time $t$ be $\mathbf{s}_t = (s_{t,1}, \dots, s_{t,N}) \in \mathcal{S}$, where $s_{t,i} \in \mathbb{Z}_{\ge0}$ is the number of type $i$ demands at time $t$ waiting to be matched, and $\mathcal{S}$ is the state space. The platform deploys a state-dependent matching policy $\phi_i: \mathcal{S} \mapsto [N] \cup \{0\}$---when a type $i$ demand arrives at state $\mathbf{s}_t$, this matching policy either (1) immediately matches the arriving request with some available active demand of type $j \in [N]$ with $s_{t,j} > 0$, or (2) lets the arriving request wait in the system for future dispatch, which we label as option 0. 
Denote the matching policy by $\boldsymbol{\phi}=(\phi_i)_{i\in[N]}$. %
Let $M^{\boldsymbol{\lambda},\boldsymbol{\phi}}(t)$ be the set of matches formed up to time $t$, and $R^{\boldsymbol{\lambda},\boldsymbol{\phi}}(t)$ be the set of unmatched events up to time $t$. The total cost incurred by time $t$ is given by
$$
C^{\boldsymbol{\lambda},\boldsymbol{\phi}}(t) = \sum_{(i,j) \in M^{\boldsymbol{\lambda},\boldsymbol{\phi}}(t)} c_{(i,j)} + \sum_{i \in R^{\boldsymbol{\lambda},\boldsymbol{\phi}}(t)} c_{(i)}.
$$
We consider stationary policies and hereafter drop the subscript $t$ in the state variable $\mathbf{s}_t$.

The objective of the platform is to minimize the expected long-run average cost via
$$
c^{\text{MDP}}(\boldsymbol{\lambda}) := \inf_{\boldsymbol{\phi}} \limsup_{T \to \infty} \frac{1}{T} \mathbb{E} [C^{\boldsymbol{\lambda},\boldsymbol{\phi}}(T)]
$$
where $c^{\text{MDP}}(\boldsymbol{\lambda})$ is defined as the minimum expected long-run average cost per unit time and the expectation is taken over all the randomness of demand arrivals, matching, and abandonment.

We state a prior result regarding a property under an optimal matching policy.
\begin{lemma}[\citet{yan2023pricing}, Lemma 1]
    There exists an optimal matching policy $\boldsymbol{\phi}$ under which any state $\mathbf{s}$ with $s_i \geq 3$ for some $i \in [N]$ is transient.
\end{lemma}
This result suggests that, without loss of optimality, one can limit the state space to $\mathcal{S} = \{0, 1, 2\}^N$. This renders the state space finite. As a consequence, under any matching policy, the resulting continuous-time Markov chain is a finite-state \textit{unichain}. We can thus rewrite our continuous-time MDP as a discrete-time MDP using the uniformization method \citep{puterman2014markov}. Let $M = \sum_{i=1}^N (\lambda_i + 2\theta_i)$ be the maximum rate of transition. According to Theorem 8.4.5 of \citet{puterman2014markov}, an optimal policy exists and there exists a \textit{relative value function} $V: \mathcal{S} \mapsto \mathbb{R}$ satisfying the following optimality equation:

\begin{align*}
    V(\mathbf{s}) + \frac{c^{\text{MDP}}(\boldsymbol{\lambda}) }{M} = \frac{1}{M} \Bigg( & \sum_{j \in [N]} \lambda_j \min \Bigg\{ \underbrace{V(\mathbf{s} + \mathbf{e}_j)}_{\text{let } j \text{ wait}}, \min_{i \in [N]: s_i > 0} \underbrace{(c_{(i,j)} + V(\mathbf{s} - \mathbf{e}_i))}_{\text{match } j \text{ with active } i} \Bigg\} \\
    & + \sum_{i \in [N]} s_i \theta_i \underbrace{(c_{(i)} + V(\mathbf{s} - \mathbf{e}_i))}_{\text{reneging of type } i} \\
    & + \underbrace{\left( M - \sum_{j \in [N]} \lambda_j - \sum_{i \in [N]} s_i \theta_i \right)}_{\text{no event occurs}} V(\mathbf{s}) \Bigg), \quad \forall \mathbf{s} \in \mathcal{S},
\end{align*}
where $c^{\text{MDP}}(\boldsymbol{\lambda}) /M$ is the minimum cost per period, and $\mathbf{e}_j$ is an $N$-dimensional unit vector whose $j^{\text{th}}$ element is one and the rest are zeroes.

\section{Additional Theoretical Results}
\label{sec:appendix_theory}

In this section, we present additional sufficient conditions for the weak concavity of $c(\boldsymbol{\lambda})$, by imposing both lower bounds on $\underline{\lambda}_i$ and upper bounds on $\overline{\lambda}_i$. We first present the result for the case of homogeneous patience levels.
\begin{proposition} \label{prop:n-weak-concave-same-theta-both-bound}
 When $\theta_i = \theta$, $\forall i\in[N]$, $c(\boldsymbol{\lambda})$ is weakly concave on $\Lambda$ if there exists an integer $K \geq 5$ such that
$$
\theta \frac{e_{(K)}}{1-2e_{(K)}} < \underline{\lambda}_i < \overline{\lambda}_i < \theta \frac{1}{K-4} \left(1+\frac{e_{(K)}}{1-2e_{(K)}}\right),\ \forall i \in [N].
$$
\end{proposition}
\Cref{prop:n-weak-concave-same-theta-both-bound} defines a set of feasible intervals for arrival rates parameterized by $K \geq 5$. As $K$ increases, $e_{(K)}$ decreases, leading to smaller values for both the lower and upper bounds, thus shifting the feasible intervals. The weak concavity holds as long as there exists a value of $K$, such that both $\underline{\lambda}_i$ and $\overline{\lambda}_i$ fall within the bounds. Also note that \Cref{prop:n-weak-concave-same-theta-upper-bound} is a special case of this proposition by letting $K = N+1$.

The next proposition gives the results when the patience levels are heterogeneous. 
\begin{proposition} \label{prop:n-weak-concave-both-bound}
$c(\boldsymbol{\lambda})$ is weakly concave on $\Lambda$ when there exists an integer $K \geq 3$ such that $(K - 2)\nu  > 2$, $e_{(K)} \neq 0.5$ and 
$$
\theta_i \frac{e_{(K)}}{1-2e_{(K)}} < \underline{\lambda}_i < \overline{\lambda}_i < \theta_i \max\left\{\frac{1}{K-2}, \frac{1}{(K-2)\nu-2}\right\}\left(1+\frac{1}{\nu}\frac{e_{(K)}}{1-2e_{(K)}}\right),\ \forall i \in [N].
$$
\end{proposition} 
Similar to \Cref{prop:n-weak-concave-same-theta-both-bound}, the bound depends on the choice of $K$ and the value of $\nu$. Again, the parameter $K$ determines the maximum number of other demand types that can be matched with demand type $i\in[N]$, and $\nu$ measures the difference in patience levels across demand types. As $\nu$ decreases---indicating more similar patience levels among demands---the upper bounds become less restrictive, expanding the allowable region. Therefore, a $\nu$ close to 1 helps to achieve weak concavity. Also note that \Cref{prop:n-weak-concave-upper-bound} is a special case of this proposition by letting $K = N+1$.

\section{Examples} \label{sec:examples}

\begin{example}[Non-concavity of $c(\boldsymbol{\lambda})$ when $N=3$, $\theta_1=\theta_2=\theta_3$] \label{example:non-concave-same-theta}
Consider the case when $N=3, \theta_1 = \theta_2 = \theta_3 = 1$. Let $[c_{(i,j)}]_{i,j\in\{1,2,3\}}$ be
$
\begin{bmatrix}
    0.72 & 1.23 & 1.25 \\
    1.23 & 0.95 & 1.02 \\
    1.25 & 1.02 & 1
\end{bmatrix}
$. The function $c(\boldsymbol{\lambda})$ is non-concave at the point $\boldsymbol{\lambda} = (1, 0.2, 0.2)$---the function is smooth at this point but the largest eigenvalue is $0.03 > 0$. 
\end{example}

\begin{example}[Non-concavity of $c(\boldsymbol{\lambda})$ when $N=2, \theta_1\neq\theta_2$] \label{example:non-concave-diff-theta}
 Let $N=2$, $\theta_1 = 1$, $\theta_2 = 2$, $c_{(1)} = c_{(2)} = 1$, and $c_{(1,2)} = c_{(2,1)} = 1.01$. The function $c(\boldsymbol{\lambda})$ is non-concave at the point $\boldsymbol{\lambda} = (0.1, 0.1)$---the function is smooth at this point but the largest eigenvalue of its Hessian matrix is $0.03 > 0$. \end{example}

\begin{example}[Violation of weak concavity of $c(\boldsymbol{\lambda})$ when $N=2, \theta_1\neq\theta_2$] \label{example:non-weak-concavity}
Let $N=2$, $\theta_1 = 1$, $\theta_2 = 8$, $c_{(1)} = c_{(2)} = 1$, and $c_{(1,2)} = c_{(2,1)} = 1.05$. According to \Cref{prop:n2-concave}, we have $\tau_1 = 4.2$ and $\tau_2 = 10.44$, and $c(\boldsymbol{\lambda})$ is weak concave when $\underline{\lambda}_2>18.57$. We can show that this bound is quite tight. For example, consider the point $(\lambda_1,\lambda_2)=(19.5, 18.5)$, we can calculate the left and right partial derivatives with respect to $\lambda_1$ at this point:
$
    {\partial c(\boldsymbol{\lambda)}}/{\partial \lambda_1^-} \approx 0.49986 < {\partial c(\boldsymbol{\lambda)}}/{\partial \lambda_1^+} = 0.5.
$
For any $\rho > 0$, a necessary condition for $c(\boldsymbol{\lambda)} - \rho\|\boldsymbol{\lambda}\|^2/2$ to be concave is that for any point $\boldsymbol{\lambda}\in\Lambda$, the left partial derivative is always larger than or equal to the right partial derivative (see e.g., \citealt{beckenbach1948convex}). However, the left partial derivative $c(\boldsymbol{\lambda)} - \rho\|\boldsymbol{\lambda}\|^2/2$ is always smaller than its right partial derivative in our example at $(\lambda_1,\lambda_2)$, since
$$
\frac{\partial c(\boldsymbol{\lambda)} - \rho\|\boldsymbol{\lambda}\|^2/2}{\partial \lambda_1^-} = \frac{\partial c(\boldsymbol{\lambda)}}{\partial \lambda_1^-} - \rho \lambda_1 < \frac{\partial c(\boldsymbol{\lambda)}}{\partial \lambda_1^+} - \rho \lambda_1 = \frac{\partial c(\boldsymbol{\lambda)} - \rho\|\boldsymbol{\lambda}\|^2/2}{\partial \lambda_1^+}.
$$
Therefore, in our example, $c(\boldsymbol{\lambda)}$ is not weakly concave.

\end{example}

\begin{example}[Violation of weak concavity of $c(\boldsymbol{\lambda})$ when $N=2, \theta_1\neq\theta_2$ and $e_{1,2}=0.5$] \label{example:non-weak-concavity-special}
Let $N=2$, $\theta_1 = 1$, $\theta_2 = 8$, $c_{(1)} = c_{(2)} = c_{(1,2)} = c_{(2,1)} = 1$, then we have $e_{1,2}=0.5$ in this case, and $\tau_1=5>0$. 

We can show that in this example, weak concavity is violated even for large values of $\lambda_1$ and $\lambda_2$, by again comparing left and right partial derivatives. Let $(\lambda_1, \lambda_2) = (11,10)$, we have
$
    \partial c(\boldsymbol{\lambda)}/\partial \lambda_1^- \approx 0.43574 < \partial c(\boldsymbol{\lambda)}/\partial \lambda_1^+ = 0.5.
$
Let $(\lambda_1, \lambda_2) = (101,100)$, we have
$
    {\partial c(\boldsymbol{\lambda)}}/{\partial \lambda_1^-} \approx 0.48837 < {\partial c(\boldsymbol{\lambda)}}/{\partial \lambda_1^+} = 0.5.
$
Let $(\lambda_1, \lambda_2) = (1001,1000)$, we have
$
    {\partial c(\boldsymbol{\lambda)}}/{\partial \lambda_1^-} \approx 0.49876 < {\partial c(\boldsymbol{\lambda)}}/{\partial \lambda_1^+} = 0.5.
$
Let $(\lambda_1, \lambda_2) = (10001,10000)$, we have
$
    {\partial c(\boldsymbol{\lambda)}}/{\partial \lambda_1^-} \approx 0.49988 < {\partial c(\boldsymbol{\lambda)}}/{\partial \lambda_1^+} = 0.5.
$
Similar to \Cref{example:non-weak-concavity}, $c(\boldsymbol{\lambda})$ is not weakly concave at these points.
\end{example}

\begin{example}[Multimodality of $g(\boldsymbol{\lambda})$] \label{example:multimodality}
Let $N=2$, $\theta_1 = \theta_2 = 0.3$,  $c_{(1)} = c_{(2)} = 1.1$, and $c_{(1,2)} = c_{(2,1)} = 1.65$. The price-conversion function is given by $p_i(\lambda_i)=1-\lambda_i$, $i=1,2$. Then $g(\boldsymbol{\lambda})$ is multimodal and not differentiable, as shown in \Cref{fig:example_multi_modality}.
\end{example}
\begin{figure}
    \centering
    \includegraphics[width=0.55\linewidth]{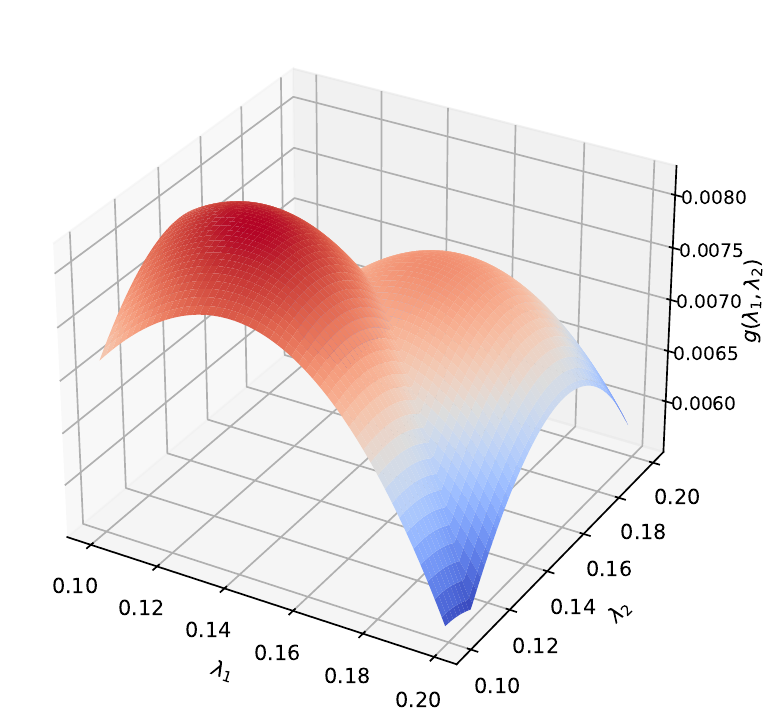}
    \caption{Example showing the multimodality and non-differentiability of $g(\boldsymbol{\lambda})$ when $N=2$.}
    \label{fig:example_multi_modality}
\end{figure}

\end{APPENDICES}

\bibliographystyle{informs2014trsc} %
\bibliography{reference} %

\begin{ECSwitch}
    \ECHead{\centering{Online Appendix}}
    \renewcommand{\theHsection}{A\arabic{section}}
    \section{Proofs and Auxiliary Results}

Throughout the proof, denote the optimal solution of $c(\boldsymbol{\lambda})$ as $\boldsymbol{x}^*=(x_{i,j}^*)_{i,j\in[N]}$ and $\boldsymbol{y}^*=(y_i^*)_{i\in[N]}$.

\subsection{Proof of \Cref{prop:theta_zero}}
When $\theta_i = 0$, $\forall i\in[N]$, constraint \eqref{equ:cn-bound} is trivially satisfied, and we have $y_i^*=0$, $\forall i\in[N]$---otherwise, by selecting a small enough value of $\varepsilon>0$, we can always decrease $y_i^*$ by $2\varepsilon$ and increase $x_{i,i}^*$ by $\varepsilon$, reducing the total cost by $c_{(i)}\varepsilon$ while still satisfying the constraint \eqref{equ:cn-flow}.

Similarly, we can prove that there exists an optimal solution satisfying $x_{i,j}^*=0$, $\forall i \neq j$. Conversely, if the opposite holds, we can decrease $x_{i,j}^*$ by $\varepsilon$, increase $x_{i,i}^*$ and $x_{j,j}^*$ by $\varepsilon/2$, then we can obtain a new feasible solution with the total cost less than or equal to the original total cost, since
$$
c_{(i,j)} \geq \frac{1}{2} c_{(i)} + \frac{1}{2} c_{(j)},\ \forall i,j\in[N].
$$

Therefore, an optimal solution can be given by $x_{i,i}^*=\lambda_i/2$, $\forall i\in[N]$. The optimal cost is given by $\sum_{i\in[N]} c_{(i)}\lambda_i/2$ and is thus linear. \hfill $\square$

\subsection{Proof of \Cref{prop:eos}}

Consider $\boldsymbol{x}^*$ and $\boldsymbol{y}^*$ as an optimal solution of $c(\boldsymbol{\lambda})$. By multiplying constraints \eqref{equ:cn-flow} and \eqref{equ:cn-bound} by $\alpha$, we have
\begin{align*}
\sum_{j\in[N]} (\alpha x_{j,i}^*) + \sum_{j\in[N]} (\alpha x_{i,j}^*) + (\alpha y_i^*) &= (\alpha \lambda_i), & \forall i \in [N],\\
\theta_i (\alpha x_{i,j}^*) &\leq (\alpha \lambda_j) y_i^*, & \forall i,j \in [N],
\end{align*}
since $\alpha > 1$, $(\alpha \lambda_j) y_i^* < (\alpha \lambda_j) (\alpha y_i^*)$, and thus $\alpha \boldsymbol{x}^*$ and $\alpha \boldsymbol{y}^*$ are feasible solutions when $\alpha \boldsymbol{\lambda}$ is the arrival rate. Then we have
$$
\alpha c(\boldsymbol{\lambda}) = \sum_{i\in[N]} \sum_{j\in[N]} c_{(i,j)} \left(\alpha x^*_{i,j}\right) + \sum_{i\in[N]} c_{(i)} \left(\alpha y^*_i \right) \geq c(\alpha\boldsymbol{\lambda}),
$$
and thus $c(\alpha\boldsymbol{\lambda})/\| \alpha\boldsymbol{\lambda} \| \leq  c(\boldsymbol{\lambda})/\| \boldsymbol{\lambda} \|$.
\hfill $\square$

\subsection{Proofs of \Cref{lemma:n-weak-concave} and \Cref{prop:n-weak-concave-nonzero}}
\label{proof:n-weak-concave}

The proof relies on the following three lemmas. The first lemma below establishes the connection between local weak concavity and global weak concavity. For ease of presentation, let $B(\hat{\boldsymbol{\lambda}}, \varepsilon) = \{ \boldsymbol{\lambda}\ |\ \|\boldsymbol{\lambda} - \hat{\boldsymbol{\lambda}}\| < \varepsilon \}$ denote the open ball centered at $\hat{\boldsymbol{\lambda}}$. 

\begin{lemma} \label{lemma:n-weak-concave-1}
    If for any $\hat{\boldsymbol{\lambda}}\in\Lambda$, there exists $\varepsilon_{\hat{\boldsymbol{\lambda}}} > 0$ such that $c(\boldsymbol{\lambda})$ is weakly concave on $B(\hat{\boldsymbol{\lambda}}, \varepsilon_{\hat{\boldsymbol{\lambda}}}) \cap \Lambda$, then $c(\boldsymbol{\lambda})$ is weakly concave on $\Lambda$.
\end{lemma}

\begin{proof} {Proof.}
    By the definition of weak concavity, for every $\hat{\boldsymbol{\lambda}} \in \Lambda$, there exists $\varepsilon_{\hat{\boldsymbol{\lambda}}} > 0$ and a scalar $\rho_{\hat{\boldsymbol{\lambda}}} \ge 0$ such that the function 
    $c(\boldsymbol{\lambda}) - \rho_{\hat{\boldsymbol{\lambda}}} \| \boldsymbol{\lambda} \|^2 / 2$ 
    is concave on the set $B(\hat{\boldsymbol{\lambda}}, \varepsilon_{\hat{\boldsymbol{\lambda}}}) \cap \Lambda$. 
    The collection of open sets $\mathcal{U} = \{ B(\hat{\boldsymbol{\lambda}}, \varepsilon_{\hat{\boldsymbol{\lambda}}}) \}_{\hat{\boldsymbol{\lambda}} \in \Lambda}$ forms an open cover of $\Lambda$. Since $\Lambda$ is a compact set, there exists a finite subcover $\{ B(\boldsymbol{\lambda}_k, \varepsilon_{\boldsymbol{\lambda}_k}) \}_{k=1}^K$ that covers $\Lambda$. We define $\bar{\rho}$ as the maximum of the finite set of local parameters:
    $$
    \bar{\rho} := \max \{ \rho_{\boldsymbol{\lambda}_1}, \dots, \rho_{\boldsymbol{\lambda}_K} \} < +\infty.
    $$
    For any $k \in \{1, \dots, K\}$, since $\bar{\rho} \geq \rho_{\boldsymbol{\lambda}_k}$, the function $c(\boldsymbol{\lambda}) - {\bar{\rho}} \| \boldsymbol{\lambda} \|^2/2$ remains concave on $B(\boldsymbol{\lambda}_k, \varepsilon_k) \cap \Lambda$. Consequently, the function
    $c(\boldsymbol{\lambda}) - {\bar{\rho}} \| \boldsymbol{\lambda} \|^2/2
    $
    is locally concave everywhere on $\Lambda$.
    
    We now utilize the equivalence between the concavity of a function and the convexity of its hypograph (see e.g., Section 3.1.7 in \citealt{Boyd_Vandenberghe_2004}). The hypograph is defined as:
    $$
    H=\left\{ (\boldsymbol{\lambda}, t)\ \bigg|\ \boldsymbol{\lambda}\in \Lambda, t\leq c(\boldsymbol{\lambda}) - \frac{\bar{\rho} \| \boldsymbol{\lambda}\|^2}{2} \right\}.
    $$
    Since $c(\boldsymbol{\lambda}) - {\bar{\rho}} \| \boldsymbol{\lambda} \|^2/2
    $ is locally concave, the hypograph is a locally convex set. Furthermore, because $c(\boldsymbol{\lambda})$ is continuous (implied by the Maximum Theorem, \citealt{ok2011real}) and $\Lambda$ is a compact convex set, $H$ is closed and connected.
    
    According to the Tietze-Nakajima Theorem, a closed, connected, and locally convex set is convex \citep{tietze1928konvexheit, nakajima1928konvexe, bjorndahl2010revisiting}. This implies that $H$ is a convex set. By the hypograph equivalence, the convexity of $H$ implies that $h(\boldsymbol{\lambda})$ is a concave function on $\Lambda$. 
    
    Since $c(\boldsymbol{\lambda}) - \bar{\rho} \| \boldsymbol{\lambda} \|^2/2$ is concave on $\Lambda$, it follows by definition that $c(\boldsymbol{\lambda})$ is weakly concave on $\Lambda$.
    
    \hfill $\square$
\end{proof}

The second lemma below is straightforward. It gives the boundedness property of a function if it can be expressed as a ratio of two polynomial functions and the denominator is always positive.

\begin{lemma} \label{lemma:n-weak-concave-2}
    Assume $\boldsymbol{\lambda}\in \hat{\Lambda} \subset \mathbb{R}^N$, where $\hat{\Lambda}$ is any compact and convex set. Let $p(\boldsymbol{\lambda})$ and $q(\boldsymbol{\lambda})$ be two (multivariate) polynomial functions of $\boldsymbol{\lambda} \in \hat{\Lambda}$. If $q(\boldsymbol{\lambda}) \neq 0$ always holds, then $p(\boldsymbol{\lambda})/q(\boldsymbol{\lambda})$ and its partial derivatives of any order are bounded on $\hat{\Lambda}$.
\end{lemma}
\begin{proof}{Proof.}
    Since $p(\boldsymbol{\lambda})$ and $q(\boldsymbol{\lambda})$ are two polynomial functions defined in a compact set $\hat{\Lambda}$, they are both continuous and bounded. By the Extreme Value Theorem, $q(\boldsymbol{\lambda})$ attains minimum and maximum values $q_{\min}, q_{\max}$ on $\hat{\Lambda}$, and either $q_{\min} \leq q_{\max} < 0$ or $0 < q_{\min} \leq q_{\max}$. Therefore, $p(\boldsymbol{\lambda})/q(\boldsymbol{\lambda})$ is bounded. Now consider the first-order partial derivative of $p(\boldsymbol{\lambda})/q(\boldsymbol{\lambda})$:
    $$
    \frac{\partial p(\boldsymbol{\lambda})/q(\boldsymbol{\lambda})}{\partial \lambda_i} = \frac{\frac{\partial p(\boldsymbol{\lambda})}{\partial\lambda_i} q(\boldsymbol{\lambda}) - \frac{\partial q(\boldsymbol{\lambda})}{\partial\lambda_i} p(\boldsymbol{\lambda})}{[q(\boldsymbol{\lambda})]^2}.
    $$
    It can be seen that the numerator and denominator are still polynomial functions, and the denominator still has a positive lower bound of $\min\{q_{\min}^2, q_{\max}^2\}$. Therefore, the aforementioned property still holds, and thus the first-order partial derivative is still bounded. By induction, any-order derivatives exist and are all bounded. \hfill $\square$
\end{proof}

Lastly, the third lemma gives the weakly concave property for a twice differentiable function with a bounded Hessian.

\begin{lemma} \label{lemma:n-weak-concave-3}
Assume $\boldsymbol{\lambda}\in \hat{\Lambda} \subset \mathbb{R}^N$, where $\hat{\Lambda}$ is any compact and convex set. If a function $f(\boldsymbol{\lambda})$ is twice differentiable and each element in its Hessian matrix is bounded for $\boldsymbol{\lambda} \in \hat{\Lambda}$, then $f(\boldsymbol{\lambda})$ is weakly concave on $\hat{\Lambda}$.
\end{lemma}
\begin{proof} {Proof.}
    Let $H(f(\boldsymbol{\lambda}))$ denote the Hessian of $f(\boldsymbol{\lambda})$, and its elements are denoted as $h_{i,j}(\boldsymbol{\lambda})$, $i,j\in[N]$. Since each element is bounded, we let $\bar{h}_{i,j}$ denote the upper bound of the absolute value of $h_{i,j}(\boldsymbol{\lambda})$, that is, we have
    $$
    |h_{i,j}(\boldsymbol{\lambda})| \leq \bar{h}_{i,j}, \forall \boldsymbol{\lambda}\in \hat{\Lambda},\ i,j\in[N].
    $$
    
    Given that $f(\boldsymbol{\lambda})$ is twice differentiable, so is $f(\boldsymbol{\lambda}) - \rho \|\boldsymbol{\lambda}\|^2$ for any $\rho \geq 0$. Then the Hessian of $f(\boldsymbol{\lambda}) - \rho \|\boldsymbol{\lambda}\|^2/2$ can be expressed as
    $
    H(f(\boldsymbol{\lambda})) - \rho \boldsymbol{I}
    $, where $\boldsymbol{I}$ is the identity matrix. The diagonal elements of this Hessian matrix are $h_{i,i}(\boldsymbol{\lambda})-\rho,\ i\in[N]$, and the non-diagonal elements are still $h_{i,j}(\boldsymbol{\lambda}),\ i,j\in[N]$. We select a value of $\rho$ that satisfies
    $$
    \rho > \max_{i\in[N]} \sum_{j\in[N]} \bar{h}_{i,j},
    $$
    then $h_{i,i}(\boldsymbol{\lambda})-\rho \leq \bar{h}_{i,i} - \rho \leq 0,\ i\in[N]$, and
    \begin{align*}
    |h_{i,i}(\boldsymbol{\lambda})-\rho| - \sum_{j\neq i} |h_{i,j}(\boldsymbol{\lambda})| &=  \rho -  h_{i,i}(\boldsymbol{\lambda})-   \sum_{j\neq i} |h_{i,j}(\boldsymbol{\lambda})| \\
    & \geq \rho -  \sum_{j\in [N]} |h_{i,j}(\boldsymbol{\lambda})| \\
    & \geq \rho -  \sum_{j\in [N]} \bar{h}_{i,j} \\
    & > 0,\ \forall i\in[N].
    \end{align*}
    This means that the matrix $
    H(f(\boldsymbol{\lambda})) - \rho \boldsymbol{I}
    $ is strictly diagonally dominant and its diagonal elements are all non-positive. The matrix is thus negative semidefinite (see, e.g., Theorem 6.1.10 in \citealt{horn2012matrix}).

    For the twice differentiable function $f(\boldsymbol{\lambda}) - \rho \|\boldsymbol{\lambda}\|^2/2$, since its Hessian $
    H(f(\boldsymbol{\lambda})) - \rho \boldsymbol{I}
    $ is negative semidefinite, it is concave. Thus $f(\boldsymbol{\lambda})$ is weakly concave on $\hat{\Lambda}$. \hfill $\square$
\end{proof}

Now we are ready to give the proof of \Cref{lemma:n-weak-concave}.

\begin{proof} {Proof of \Cref{lemma:n-weak-concave}.} We first write program \eqref{equ:cn} as the standard form consisting of only equality constraints and nonnegative decision variables:
\begin{subequations} \label{equ:cn-standard-form}
\begin{align} 
c(\boldsymbol{\lambda}) = \min_{\boldsymbol{x},\boldsymbol{y}}\quad & \sum_{i\in[N]} \sum_{j\in[N]} c_{(i,j)} x_{i,j} + \sum_{i\in[N]} c_{(i)} y_i, \\
\text{s.t.} \quad & \sum_{j\in[N]} x_{j,i} + \sum_{j\in[N]} x_{i,j} + y_i = \lambda_i, & \forall i \in [N], \\
& \theta_i x_{i,j} +\delta_{i,j}= \lambda_j y_i, & \forall i,j \in [N], \label{equ:cn-bound-standard-form}\\
& x_{i,j} \geq 0, & \forall i,j \in [N], \\
& \delta_{i,j} \geq 0, & \forall i,j \in [N], \label{equ:cn-nonneg-delta}\\
& y_{i} \geq 0, & \forall i \in [N].
\end{align}
\end{subequations}
For ease of expression, let
$
    \boldsymbol{b}(\boldsymbol{\lambda}) = (\boldsymbol{\lambda}, \boldsymbol{0}_{N^2})^T
$ denote the right-hand-side coefficients.

Given a value of $\hat{\boldsymbol{\lambda}}\in\Lambda$, assume there are a total $K\geq1$ optimal BFSs to program \eqref{equ:cn} when $\boldsymbol{\lambda}=\hat{\boldsymbol{\lambda}}$. For the $k$-th BFS, let $\mathcal{B}_k$ be the set of indices for its basis, and $\boldsymbol{B}^{(k)}(\boldsymbol{\lambda})$ be the basis matrix. The BFS can then be expressed as
\begin{equation}
    \mathbf{z}_{\mathcal{B}_k}(\hat{\boldsymbol{\lambda}}) = {\boldsymbol{B}^{(k)}(\hat{\boldsymbol{\lambda}})}^{-1} \boldsymbol{b}(\hat{\boldsymbol{\lambda}}) = \frac{\text{adj}(\boldsymbol{B}^{(k)}(\hat{\boldsymbol{\lambda}}))}{\text{det}(\boldsymbol{B}^{(k)}(\hat{\boldsymbol{\lambda}}))}\hat{\boldsymbol{\lambda}},\ \forall k\in[K], \label{equ:bfs}
\end{equation}
and we denote the value function produced by this BFS as $c_{\mathcal{B}_k}(\hat{\boldsymbol{\lambda}})$.
According to our assumption, all optimal BFSs are nondegenerate. Therefore, $ \mathbf{z}_{\mathcal{B}_k}(\hat{\boldsymbol{\lambda}}) > \boldsymbol{0} $ and $\text{det}(\boldsymbol{B}^{(k)}(\hat{\boldsymbol{\lambda}})) \neq 0$ for all $k \in [K]$. 

Now consider a small neighborhood of $\hat{\boldsymbol{\lambda}}$, defined as $ B(\hat{\boldsymbol{\lambda}}, \varepsilon) \cap \Lambda$, where $\varepsilon>0$ is a sufficiently small value. For every optimal BFS at $\hat{\boldsymbol{\lambda}}$, we can use it to produce another solution at any $ \boldsymbol{\lambda} \in B(\hat{\boldsymbol{\lambda}}, \varepsilon) \cap \Lambda$, by fixing the same non-basic variables to zero. Specifically, for the $k$-th optimal BFS at $\hat{\boldsymbol{\lambda}}$, by substituting $\hat{\boldsymbol{\lambda}}$ with $ \boldsymbol{\lambda}$ in \eqref{equ:bfs}, we obtain the following solution at $ \boldsymbol{\lambda} \in B(\hat{\boldsymbol{\lambda}}, \varepsilon) \cap \Lambda$:
\begin{align*}
    \mathbf{z}_{\mathcal{B}_k}(\boldsymbol{\lambda}) = {\boldsymbol{B}^{(k)}(\boldsymbol{\lambda})}^{-1} \boldsymbol{b}(\boldsymbol{\lambda}) = \frac{\text{adj}(\boldsymbol{B}^{(k)}(\boldsymbol{\lambda}))}{\text{det}(\boldsymbol{B}^{(k)}(\boldsymbol{\lambda}))}\boldsymbol{\lambda},\ \forall k\in[K],
\end{align*}
and value function $c_{\mathcal{B}_k}(\boldsymbol{\lambda})$ is similarly obtained by substituting $\hat{\boldsymbol{\lambda}}$ with $\boldsymbol{\lambda}$.

For any $k\in[K]$, since $ \mathbf{z}_{\mathcal{B}_k}(\hat{\boldsymbol{\lambda}}) > \boldsymbol{0} $, by selecting a small enough value of $\varepsilon$, we can guarantee that $\mathbf{z}_{\mathcal{B}_k}(\boldsymbol{\lambda}) > \boldsymbol{0}$ still holds for any $ \boldsymbol{\lambda} \in B(\hat{\boldsymbol{\lambda}}, \varepsilon) \cap \Lambda$. Therefore, these $K$ solutions are still BFS at $ \boldsymbol{\lambda} \in B(\hat{\boldsymbol{\lambda}}, \varepsilon) \cap \Lambda$. Furthermore, because the total number of bases is finite, there is a strict positive gap between the optimal value and the values of any non-optimal BFS at $\hat{\boldsymbol{\lambda}}$. By continuity, this gap prevents any other feasible basis from becoming optimal within the neighborhood. Similarly, by selecting a small enough value of $\varepsilon$, we can also prevent any other non-feasible basis from becoming feasible within the neighborhood. That is, we have the following results:
\begin{align*}\forall\hat{\boldsymbol{\lambda}}\in\Lambda,\ \exists\varepsilon >0, \text{ s.t. } c(\boldsymbol{\lambda}) = \min_{k\in[K]}\{ c_{\mathcal{B}_k}(\boldsymbol{\lambda})\},\ \forall \boldsymbol{\lambda} \in B(\hat{\boldsymbol{\lambda}}, \varepsilon) \cap \Lambda.
\end{align*}

On the other hand, note that all elements in $\boldsymbol{B}^{(k)}(\boldsymbol{\lambda})$ are polynomial in $\boldsymbol{\lambda}$, and the same is also true for $\text{adj}(\boldsymbol{B}^{(k)}(\boldsymbol{\lambda}))$ and $\text{det}(\boldsymbol{B}^{(k)}(\boldsymbol{\lambda}))$. Therefore, each element in $\mathbf{z}_{\mathcal{B}_k}(\boldsymbol{\lambda})$ can be expressed as a ratio of two polynomial functions of $\boldsymbol{\lambda}$, where the denominator is $\text{det}(\boldsymbol{B}^{(k)}(\boldsymbol{\lambda}))$. Since the basis matrix is invertible at $\hat{\boldsymbol{\lambda}}$, $\text{det}(\boldsymbol{B}^{(k)}(\boldsymbol{\lambda}))$ is non-zero. By continuity, it is still non-zero on $B(\hat{\boldsymbol{\lambda}}, \varepsilon) \cap \Lambda$ by choosing a sufficiently small $\varepsilon$. Therefore, based on \Cref{lemma:n-weak-concave-2}, each element in $\mathbf{z}_{\mathcal{B}_k}(\boldsymbol{\lambda})$ has infinitely many derivatives and is bounded as $\boldsymbol{\lambda}$ changes. The same results of differentiability and boundedness also hold for $c_{\mathcal{B}_k}(\boldsymbol{\lambda})$.

Denote $\boldsymbol{H}(c_{\mathcal{B}_k}(\boldsymbol{\lambda}))$ as the Hessian matrix of the value function $c_{\mathcal{B}_k}(\boldsymbol{\lambda})$. Based on \Cref{lemma:n-weak-concave-2}, every element of $\boldsymbol{H}(c_{\mathcal{B}_k}(\boldsymbol{\lambda}))$ is still bounded as $\boldsymbol{\lambda}$ changes. Then, based on \Cref{lemma:n-weak-concave-3}, $c_{\mathcal{B}_k}(\boldsymbol{\lambda})$ is weakly concave: there exists a positive scalar $\rho_k$ such that $c_{\mathcal{B}_k}(\boldsymbol{\lambda}) - \rho_k \|\boldsymbol{\lambda}\|^2/2$ is concave on $B(\hat{\boldsymbol{\lambda}}, \varepsilon) \cap \Lambda$. Let
$
\bar{\rho}=\max_{k\in[K]}\rho_k
$, then $c_{\mathcal{B}_k}(\boldsymbol{\lambda}) - \bar{\rho} \|\boldsymbol{\lambda}\|^2/2$ is still concave since $\bar{\rho} \ge \rho_k$. Given that $c(\boldsymbol{\lambda}) = \min_{k\in[K]}\{ c_{\mathcal{B}_k}(\boldsymbol{\lambda})\}$, we know that $c(\boldsymbol{\lambda}) - \bar{\rho} \|\boldsymbol{\lambda}\|^2/2$ is concave because the minimization preserves concavity (see Section 3.2.3, \citealt{Boyd_Vandenberghe_2004}), and thus $c(\boldsymbol{\lambda})$ is (locally) weakly concave for $\boldsymbol{\lambda} \in B(\hat{\boldsymbol{\lambda}}, \varepsilon) \cap \Lambda$. Lastly, using \Cref{lemma:n-weak-concave-1}, we can prove that $c(\boldsymbol{\lambda})$ is weakly concave on $\Lambda$. \hfill $\square$

\end{proof}

\begin{proof}{Proof of \Cref{prop:n-weak-concave-nonzero}}
For the standard form \eqref{equ:cn-standard-form}, it involves $2N^2+N$ variables and $N^2+N$ equality constraints. It can be verified that the equality constraints are linearly independent as we assumed that $\boldsymbol{\lambda} \geq \underline{\boldsymbol{\lambda}} > \boldsymbol{0}$: constraint \eqref{equ:cn-flow} spans $N$ dimensions because of the unique occurrence of $y_i$, $\forall i\in[N]$, and \eqref{equ:cn-nonneg-delta} spans additional $N^2$ dimensions because of $\delta_{i,j}$, $\forall i,j\in[N]$.

Based on the definition of BFS and the linear independence of constraints, there must be at least $(2N^2+N) - (N^2 + N) = N^2$ variables equal to zero in each BFS \citep{bertsimas1997introduction}. On the other hand, as we assumed, $y_i>0, \forall i\in[N]$, and for any $i,j\in[N]$, $x_{i,j}$ and $\delta_{i,j}$ cannot be zero at the same time (otherwise constraint \eqref{equ:cn-bound-standard-form} is violated), meaning that there are at most $N^2$ zero variables for each BFS. Therefore, the BFS has exactly $N^2$ zero variables and is thus nondegenerate. 

\hfill $\square$
\end{proof}

\subsection{Proofs of \Cref{lemma:n1-solution} and \Cref{prop:n1-concave}}

\begin{proof} {Proof of \Cref{lemma:n1-solution}}
When $N=1$, constraint \eqref{equ:cn-flow} can be written as $2x_{1,1} +y_1=\lambda_1$, and constraint \eqref{equ:cn-bound} can be written as $\theta x_{1,1}\leq \lambda_1y_1$. We first state that constraint \eqref{equ:cn-bound} should be binding. Otherwise, if the optimal solution satisfies $\theta x_{1,1}^* < \lambda_1y_1^*$, we can obtain a new feasible solution by letting
$
  \hat{x}_{1,1} = x_{1,1}^*+\varepsilon,\ \hat{y}_{1} = y_{1}^*-2\varepsilon.
$
By selecting a small enough value of $\varepsilon>0$, $\theta \hat{x}_{1,1} < \lambda_1\hat{y}_{1}$ is still satisfied. Constraint \eqref{equ:cn-flow} also holds.
Under this new feasible solution, the value function is decreased by $c_{(1)}\varepsilon$, which is a contradiction. 

Therefore, constraint \eqref{equ:cn-bound} is indeed binding. Solving $x_{1,1}$ and $y_{1}$ subject to two equations \eqref{equ:cn-flow} and \eqref{equ:cn-bound} leads to $y_1^*=\lambda_1\theta /(\theta +2\lambda_1)$ and $x_{1,1}^* = \lambda_1^2/(\theta +2\lambda_1)$. $\hfill\square$
\end{proof}

\begin{proof} {Proof of \Cref{prop:n1-concave}}
Given \Cref{lemma:n1-solution},
$$
c(\lambda_1)=c_{(1)}\frac{\lambda_1(\theta +\lambda_1)}{\theta +2\lambda_1}.
$$
Taking the second-order derivative of $c(\lambda_1)$ leads to
$$
\frac{d^2 c(\lambda_1)}{d \lambda_1^2}=-2c_{(1)}\frac{\theta ^2}{(\theta +2\lambda_1)^3} < 0,
$$
thus $c(\lambda_1)$ is concave. \hfill $\square$
\end{proof}

\subsection{Proofs of \Cref{lemma:n2-solution-same-theta} and \Cref{lemma:n2-solution}} \label{proof:n2-solution}
For ease of presentation, we directly consider the general case where $\theta_1$ can be different from $\theta_2$, and we explicitly write all constraints in the case of $N=2$ as follows.
\begin{subequations} \label{equ:c2} 
\begin{align}
c(\boldsymbol{\lambda}) = \min_{\boldsymbol{x},\boldsymbol{y}}\quad & \ c_{(1)}(x_{1,1}+y_1) + c_{(2)}(x_{2,2}+y_2) + c_{(1,2)} (x_{1,2} + x_{2,1}), \label{equ:c2-obj}\\
\text{s.t.} \quad & 2x_{1,1}+x_{1,2}+x_{2,1}+y_1=\lambda_1, \label{equ:c2-flow-1}\\
& 2x_{2,2}+x_{1,2}+x_{2,1}+y_2=\lambda_2, \label{equ:c2-flow-2}\\
& \theta_1 x_{1,1} \leq \lambda_1 y_1, \label{equ:c2-bound-11}\\
& \theta_1 x_{1,2} \leq \lambda_2 y_1, \label{equ:c2-bound-12}\\
& \theta_2 x_{2,1} \leq \lambda_1 y_2, \label{equ:c2-bound-21}\\
& \theta_2 x_{2,2} \leq \lambda_2 y_2, \label{equ:c2-bound-22} \\
& x_{1,1},x_{1,2},x_{2,1},x_{2,2} \geq 0, \label{equ:c2-nonneg-x} \\
& y_1, y_2 \geq 0. \label{equ:c2-nonneg-y} 
\end{align}
\end{subequations}

The proof relies on the following lemmas that partially characterize the optimal solution structure.

\begin{lemma} \label{lemma:n2-one-binding}
     All optimal solutions satisfy $y_1^* = \max\{\theta_1x_{1,1}^*/\lambda_1, \theta_1x_{1,2}^*/\lambda_2\}, y_2^* = \max\{\theta_2x_{2,1}^*/\lambda_1, \theta_2x_{2,2}^*/\lambda_2\}$. That is, at least one of \eqref{equ:c2-bound-11} and \eqref{equ:c2-bound-12} is binding, and at least one of \eqref{equ:c2-bound-21} and \eqref{equ:c2-bound-22} is binding.
\end{lemma}
\begin{proof} {Proof.}
    Assume by way of contradiction that there exists an optimal solution that satisfies $y_1^* \neq \max\{\theta_1x_{1,1}^*/\lambda_1, \theta_1x_{1,2}^*/\lambda_2\}$. According to \eqref{equ:c2-bound-11} and \eqref{equ:c2-bound-12}, $y_1^* > \theta_1x_{1,1}^*/\lambda_1$ and $y_1^* >  \theta_1x_{1,2}^*/\lambda_2$. In this case, we can increase the value of $x_{1,1}^*$ by $\varepsilon > 0$, decrease the value of $y_1^*$ by $2\varepsilon$, and keep the remaining variables unchanged. This still satisfies constraint \eqref{equ:c2-flow-1} and other constraints as long as $\varepsilon$ is small enough. However, the total cost will be decreased by $c_{(1)}(2\varepsilon - \varepsilon) = c_{(1)} \varepsilon$, which violates the assumption of optimality.

    Similarly, it can be proved that $y_2^* = \max\{\theta_2x_{2,1}^*/\lambda_1, \theta_2x_{2,2}^*/\lambda_2\}$. \hfill $\square$
\end{proof}

\begin{lemma} \label{lemma:n2-cross-binding}
When $\Delta_1(\lambda_1,\lambda_2)\leq0$, there exists an optimal solution satisfying $x_{1,2}^* = x_{2,1}^* = 0$. When $\Delta_1(\lambda_1,\lambda_2)>0$, all optimal solutions satisfy $\theta_1x_{1,2}^*=\lambda_2y_1^*$ and $\theta_2x_{2,1}^*=\lambda_1y_2^*$.
\end{lemma}

\begin{proof} {Proof.}
Consider the first case where $\Delta_1(\lambda_1,\lambda_2)\leq0$. Assume, again by way of contradiction, that either $x_{1,2}^* > 0$ or $x_{2,1}^* > 0$. If $x_{1,2}^* > 0$, let $\hat{x}_{1,2} = x_{1,2}^* - \varepsilon$, $\hat{x}_{2,1}=x_{2,1}^*$, and let $\hat{y}_{i} = y_{i}^* + \varepsilon \theta_i/(\theta_i+2\lambda_i)$ and $\hat{x}_{i,i} = x_{i,i}^* + \varepsilon \lambda_i/(\theta_i+2\lambda_i)$ for $i=1,2$. Since $\hat{y}_{i} + 2\hat{x}_{i,i} + \hat{x}_{1,2} + \hat{x}_{2,1} = y_i^*+2x_{i,i}^* + x_{1,2}^* + x_{2,1}^* = \lambda_i$ ($i=1,2$), \eqref{equ:c2-flow-1} and \eqref{equ:c2-flow-2} are still satisfied. \eqref{equ:c2-bound-11} and \eqref{equ:c2-bound-22} are also satisfied because for $i=1,2$, $\theta_i \hat{x}_{i,i} = \theta_i x_{i,i}^* + \theta_i \varepsilon \lambda_i/(\theta_i+2\lambda_i) \leq \lambda_i y_{i}^* + \lambda_i\theta_i \varepsilon /(\theta_i+2\lambda_i) = \lambda_i\hat{y}_{i}$. \eqref{equ:c2-bound-12} and \eqref{equ:c2-bound-21} also hold because $\hat{y}_{1}, \hat{y}_{2}$ are larger than $y_{1}^*$ and $y_{2}^*$, whereas $\hat{x}_{1,2}, \hat{x}_{2,1}$ are smaller than $x_{1,2}^*$ and $x_{2,1}^*$. Lastly, the non-negative requirements \eqref{equ:c2-nonneg-x} and \eqref{equ:c2-nonneg-y} are satisfied by choosing a small value of $\varepsilon < x_{1,2}^*$. Therefore, the new solution $\boldsymbol{\hat{y}}, \boldsymbol{\hat{x}}$ is feasible.

Under this feasible solution, the total cost is given by
\begin{align*}
    & c_{(1)}(\hat{x}_{1,1}+\hat{y}_1) + c_{(2)}(\hat{x}_{2,2}+\hat{y}_2) + c_{(1,2)} (\hat{x}_{1,2} + \hat{x}_{2,1}) \\
    = & c_{(1)}(x_{1,1}^*+y_1^*) + c_{(2)}(x_{2,2}^*+y_2^*) + c_{(1,2)} (x_{1,2}^* + x_{2,1}^*) + \varepsilon\left[ c_{(1)}\frac{\theta_1+\lambda_1}{\theta_1+2\lambda_1} + c_{(2)}\frac{\theta_2+\lambda_2}{\theta_2+2\lambda_2} - c_{(1,2)}\right] \\
    = & c_{(1)}(x_{1,1}^*+y_1^*) + c_{(2)}(x_{2,2}^*+y_2^*) + c_{(1,2)} (x_{1,2}^* + x_{2,1}^*) + \varepsilon\Delta_1(\lambda_1,\lambda_2) \\
    \stackrel{(a)}{\leq} & c_{(1)}(x_{1,1}^*+y_1^*) + c_{(2)}(x_{2,2}^*+y_2^*) + c_{(1,2)} (x_{1,2}^* + x_{2,1}^*),
\end{align*}
where (a) is because of the assumption that $\Delta_1(\lambda_1,\lambda_2) \leq0$. Therefore, if $x_{1,2}^*>0$, by decreasing the value of $x_{1,2}$, we can always obtain another feasible solution without increasing the value function, meaning that $x_{1,2}^*=0$ is an optimal solution. Similarly, $x_{2,1}^*=0$.

Now consider the second case where $\Delta_1(\lambda_1,\lambda_2)>0$. In this case, we show that both \eqref{equ:c2-bound-12} and \eqref{equ:c2-bound-21} are binding at optimality. Again assume the opposite is true and consider the following three sub-cases:

\smallskip

(i) If both \eqref{equ:c2-bound-12} and \eqref{equ:c2-bound-21} are not binding, that is, $\theta_1 x_{1,2}^* < \lambda_2 y_1^*$ and $\theta_2 x_{2,1}^* < \lambda_1 y_2^*$, we then have $y_1^*>0$ and $y_2^* > 0$. Also according to \Cref{lemma:n2-one-binding}, \eqref{equ:c2-bound-11} and \eqref{equ:c2-bound-22} are binding, thus $x_{1,1}^*, x_{2,2}^*>0$. We can let $\hat{x}_{1,2} = x_{1,2}^* + \varepsilon$, $\hat{x}_{2,1}=x_{2,1}^*$, and let $\hat{y}_{i} = y_{i}^* - \varepsilon \theta_i/(\theta_i+2\lambda_i)$ and $\hat{x}_{i,i} = x_{i,i}^* - \varepsilon \lambda_i/(\theta_i+2\lambda_i)$ for $i=1,2$, where $\varepsilon > 0$. Similar to the first case, one can verify that all constraints are satisfied under this new solution when $\varepsilon$ is sufficiently small, and the new cost is given by
\begin{align*}
    & c_{(1)}(\hat{x}_{1,1}+\hat{y}_1) + c_{(2)}(\hat{x}_{2,2}+\hat{y}_2) + c_{(1,2)} (\hat{x}_{1,2} + \hat{x}_{2,1}) \\
    = & c_{(1)}(x_{1,1}^*+y_1^*) + c_{(2)}(x_{2,2}^*+y_2^*) + c_{(1,2)} (x_{1,2}^* + x_{2,1}^*) - \varepsilon\Delta_1(\lambda_1,\lambda_2)  \\
    \stackrel{(a)}{<} & c_{(1)}(x_{1,1}^*+y_1^*) + c_{(2)}(x_{2,2}^*+y_2^*) + c_{(1,2)} (x_{1,2}^* + x_{2,1}^*),
\end{align*}
where (a) is due to $\Delta_1(\lambda_1,\lambda_2) >0$.

\smallskip

(ii) If only \eqref{equ:c2-bound-21} is not binding, that is, $\theta_1 x_{1,2}^* = \lambda_2 y_1^*$ and $\theta_2 x_{2,1}^* < \lambda_1 y_2^*$, then we have $y_2^*>0$. Also according to \Cref{lemma:n2-one-binding}, $\theta_2x_{2,2}^*=\lambda_2y_2^*$. Here we must have $y_1^*>0$, otherwise if $y_1^*=0$, according to \eqref{equ:c2-bound-11}, \eqref{equ:c2-bound-12}, and \eqref{equ:c2-flow-1}, $x_{1,1}^*=x_{1,2}^*=0$, and $x_{2,1}^*=\lambda_1$. Thus $y_2^*>\theta_2x_{2,1}^*/\lambda_1=\theta_2$, and $x_{2,2}^*=\lambda_2y_2^*/\theta_2>\lambda_2$, which violates \eqref{equ:c2-flow-2}. Therefore, $y_1^*, x_{1,2}^*>0$.

We can then construct new feasible solutions with lower costs depending on whether $x_{1,1}^*$ is 0. If $x_{1,1}^*=0$, then $\theta_1x_{1,1}^*<\lambda_1y_1^*$. We construct a new feasible solution as follows: Given $\varepsilon>0$, let $\hat{x}_{2,1} = x_{2,1}^* + \varepsilon$, $\hat{y}_{1} = y_{1}^* - \varepsilon \theta_1/(\theta_1+\lambda_2)$, $\hat{x}_{1,1} = x_{1,1}^*$, $\hat{x}_{1,2} = x_{1,2}^* - \varepsilon \lambda_2/(\theta_1+\lambda_2)$, $\hat{y}_2 = y_2^* - \varepsilon \theta_1\theta_2/\left[(\theta_1+\lambda_2)(\theta_2+2\lambda_2)\right]$, and $\hat{x}_{2,2} = x_{2,2}^* - \varepsilon \theta_1\lambda_2/\left[(\theta_1+\lambda_2)(\theta_2+2\lambda_2)\right]$. Under the new solution $\boldsymbol{\hat{y}}, \boldsymbol{\hat{x}}$, \eqref{equ:c2-flow-1} and \eqref{equ:c2-flow-2} are satisfied because:
\begin{align*}
&2\hat{x}_{1,1}+\hat{x}_{1,2}+\hat{x}_{2,1}+\hat{y}_1 \\= &2x_{1,1}^*+x_{1,2}^*+x_{2,1}^*+y_1^* - \frac{\lambda_2}{\theta_1+\lambda_2}\varepsilon + \varepsilon - \frac{\theta_1}{\theta_1+\lambda_2}\varepsilon\\ =& \lambda_1, \\[3mm]
& 2\hat{x}_{2,2}+\hat{x}_{1,2}+\hat{x}_{2,1}+\hat{y}_2 \\ = & 2x_{2,2}^*+x_{1,2}^*+x_{2,1}^*+y_2^* -2\frac{\theta_1\lambda_2}{(\theta_1+\lambda_2)(\theta_2+2\lambda_2)}\varepsilon - \frac{\lambda_2}{\theta_1+\lambda_2}\varepsilon + \varepsilon - \frac{\theta_1\theta_2}{(\theta_1+\lambda_2)(\theta_2+2\lambda_2)}\varepsilon \\ =& \lambda_2,
\end{align*}
and the remaining constraints are also satisfied as long as $\varepsilon$ is sufficiently small. The new total cost is
\begin{align*}
    & c_{(1)}(\hat{x}_{1,1}+\hat{y}_1) + c_{(2)}(\hat{x}_{2,2}+\hat{y}_2) + c_{(1,2)} (\hat{x}_{1,2} + \hat{x}_{2,1}) \\
    = & c_{(1)}(x_{1,1}^*+y_1^*) + c_{(2)}(x_{2,2}^*+y_2^*) + c_{(1,2)} (x_{1,2}^* + x_{2,1}^*) \\ &- \varepsilon\left[ c_{(1)}\frac{\theta_1}{\theta_1+\lambda_2} + c_{(2)}\frac{\theta_1(\theta_2+\lambda_2)}{(\theta_1+\lambda_2)(\theta_2+2\lambda_2)} - c_{(1,2)}\frac{\theta_1}{\theta_1+\lambda_2}\right] \\
    = & c_{(1)}(x_{1,1}^*+y_1^*) + c_{(2)}(x_{2,2}^*+y_2^*) + c_{(1,2)} (x_{1,2}^* + x_{2,1}^*)  - \varepsilon \frac{\theta_1}{\theta_1+\lambda_2} \left[ c_{(1)}+ c_{(2)}\frac{\theta_2+\lambda_2}{\theta_2+2\lambda_2} - c_{(1,2)}\right] \\
    < & c_{(1)}(x_{1,1}^*+y_1^*) + c_{(2)}(x_{2,2}^*+y_2^*) + c_{(1,2)} (x_{1,2}^* + x_{2,1}^*)  - \varepsilon \frac{\theta_1}{\theta_1+\lambda_2} \left[ c_{(1)}\frac{\theta_1+\lambda_1}{\theta_1+2\lambda_1}+ c_{(2)}\frac{\theta_2+\lambda_2}{\theta_2+2\lambda_2} - c_{(1,2)}\right] \\
    < & c_{(1)}(x_{1,1}^*+y_1^*) + c_{(2)}(x_{2,2}^*+y_2^*) + c_{(1,2)} (x_{1,2}^* + x_{2,1}^*).
\end{align*}

If $x_{1,1}^*>0$, then we construct another new feasible solution as follows: Given $\varepsilon >0$, let $\hat{x}_{2,1} = x_{2,1}^* + \varepsilon$, $\hat{y}_{1} = y_{1}^* - \varepsilon \theta_1/(\theta_1+2\lambda_1+\lambda_2)$, $\hat{x}_{1,1} = x_{1,1}^* - \varepsilon \lambda_1/(\theta_1+2\lambda_1+\lambda_2)$, $\hat{x}_{1,2} = x_{1,2}^* - \varepsilon \lambda_2/(\theta_1+2\lambda_1+\lambda_2)$, $\hat{y}_2 = y_2^* - \varepsilon (\theta_1+2\lambda_1)\theta_2/\left[(\theta_1+2\lambda_1+\lambda_2)(\theta_2+2\lambda_2)\right]$, and $\hat{x}_{2,2} = x_{2,2}^* - \varepsilon (\theta_1+2\lambda_1)\lambda_2/\left[(\theta_1+2\lambda_1+\lambda_2)(\theta_2+2\lambda_2)\right]$. Under the new solution $\boldsymbol{\hat{y}}, \boldsymbol{\hat{x}}$, \eqref{equ:c2-flow-1} and \eqref{equ:c2-flow-2} are satisfied because:
\begin{align*}
&2\hat{x}_{1,1}+\hat{x}_{1,2}+\hat{x}_{2,1}+\hat{y}_1 \\= &2x_{1,1}^*+x_{1,2}^*+x_{2,1}^*+y_1^* -2\frac{\lambda_1}{\theta_1+2\lambda_1+\lambda_2}\varepsilon - \frac{\lambda_2}{\theta_1+2\lambda_1+\lambda_2}\varepsilon + \varepsilon - \frac{\theta_1}{\theta_1+2\lambda_1+\lambda_2}\varepsilon\\ =& \lambda_1, \\[3mm]
& 2\hat{x}_{2,2}+\hat{x}_{1,2}+\hat{x}_{2,1}+\hat{y}_2 \\ = & 2x_{2,2}^*+x_{1,2}^*+x_{2,1}^*+y_2^* -2\frac{(\theta_1+2\lambda_1)\lambda_2}{(\theta_1+2\lambda_1+\lambda_2)(\theta_2+2\lambda_2)}\varepsilon - \frac{\lambda_2}{\theta_1+2\lambda_1+\lambda_2}\varepsilon + \varepsilon - \frac{(\theta_1+2\lambda_1)\theta_2}{(\theta_1+2\lambda_1+\lambda_2)(\theta_2+2\lambda_2)}\varepsilon \\ =& \lambda_2,
\end{align*}
and the remaining constraints are also satisfied as long as $\varepsilon$ is sufficiently small. Under $\boldsymbol{\hat{y}}, \boldsymbol{\hat{x}}$, the total cost is
\begin{align*}
    & c_{(1)}(\hat{x}_{1,1}+\hat{y}_1) + c_{(2)}(\hat{x}_{2,2}+\hat{y}_2) + c_{(1,2)} (\hat{x}_{1,2} + \hat{x}_{2,1}) \\
    = & c_{(1)}(x_{1,1}^*+y_1^*) + c_{(2)}(x_{2,2}^*+y_2^*) + c_{(1,2)} (x_{1,2}^* + x_{2,1}^*) \\ &- \varepsilon\left[ c_{(1)}\frac{\theta_1+\lambda_1}{\theta_1+2\lambda_1+\lambda_2} + c_{(2)}\frac{(\theta_2+\lambda_2)(\theta_1+2\lambda_1)}{(\theta_2+2\lambda_2)(\theta_1+2\lambda_1+\lambda_2)} - c_{(1,2)}\frac{\theta_1+2\lambda_1}{\theta_1+2\lambda_1+\lambda_2}\right] \\
    = & c_{(1)}(x_{1,1}^*+y_1^*) + c_{(2)}(x_{2,2}^*+y_2^*) + c_{(1,2)} (x_{1,2}^* + x_{2,1}^*)  - \varepsilon \frac{\theta_1+2\lambda_1}{\theta_1+2\lambda_1+\lambda_2} \left[ c_{(1)}\frac{\theta_1+\lambda_1}{\theta_1+2\lambda_1} + c_{(2)}\frac{\theta_2+\lambda_2}{\theta_2+2\lambda_2} - c_{(1,2)}\right] \\
    < & c_{(1)}(x_{1,1}^*+y_1^*) + c_{(2)}(x_{2,2}^*+y_2^*) + c_{(1,2)} (x_{1,2}^* + x_{2,1}^*),
\end{align*}
leading to the contradiction. 

\smallskip

(iii) If only \eqref{equ:c2-bound-12} is not binding, we can also get a contradiction in a similar way to sub-case (ii). 

Therefore, both \eqref{equ:c2-bound-12} and \eqref{equ:c2-bound-21} are binding at optimality, that is, $\theta_1 x_{1,2}^* = \lambda_2 y_1^*$ and $\theta_2 x_{2,1}^* = \lambda_1 y_2^*$ when $\Delta_1(\lambda_1,\lambda_2)>0$.
\hfill $\square$
\end{proof}

\begin{lemma} 
\label{lemma:n2-y1-nonneg}
    All optimal solutions satisfy $y_1^* > 0$ and $\theta_2 x_{2,2}^*=\lambda_2 y_2^*$.
\end{lemma}

\begin{proof} {Proof.}
We prove this lemma by contradiction. First consider the case where $y_1^*=0$. According to \eqref{equ:c2-bound-11}, \eqref{equ:c2-bound-12} and \eqref{equ:c2-nonneg-x}, $x_{1,1}^*=x_{1,2}^*=0$. Then \eqref{equ:c2-flow-1} leads to $x_{2,1}^*=\lambda_1 > 0$, \eqref{equ:c2-bound-21} leads to $y_2^* \geq \theta_2$, and \eqref{equ:c2-flow-2} is equivalent to 
$$
   x_{2,2}^* = \frac{\lambda_2 -x_{2,1}^* - y_2^*}{2}  \leq \frac{\lambda_2 - \lambda_1 -\theta_2}{2} < \lambda_2 \leq \frac{\lambda_2y_2^*}{\theta_2},
$$
meaning that strict inequality holds for \eqref{equ:c2-bound-22}.

For the case where $\theta_2 x_{2,2}^*<\lambda_2 y_2^*$, strict inequality also holds for \eqref{equ:c2-bound-22}. We also have $y_2^* > 0$. According to \Cref{lemma:n2-one-binding}, \eqref{equ:c2-bound-21} must be binding, thus $x_{2,1}^* > 0$.

Now for both cases, we try to give another feasible solution with a smaller total cost. Given a small value of $\varepsilon > 0$, let $\hat{y}_1=y_1^*+\theta_1\varepsilon/(\theta_1+2\lambda_1+\lambda_2)$, $\hat{x}_{1,1} = x_{1,1}^*+\lambda_1\varepsilon/(\theta_1+2\lambda_1+\lambda_2)$, $\hat{x}_{1,2} = x_{1,2}^*+\lambda_2\varepsilon/(\theta_1+2\lambda_1+\lambda_2)$, $\hat{x}_{2,1} = x_{2,1}^* - \varepsilon$, $\hat{y}_2 = y_2^* - \theta_2\varepsilon/\lambda_1$, and $\hat{x}_{2,2} = x_{2,2}^* + \left[(\theta_1+2\lambda_1)/(\theta_1+2\lambda_1+\lambda_2)+\theta_2/\lambda_1 \right]\varepsilon/2$. 

For this new solution $\boldsymbol{\hat{x}}, \boldsymbol{\hat{y}}$, one can verify that \eqref{equ:c2-flow-1}, \eqref{equ:c2-flow-2}, \eqref{equ:c2-bound-11}, \eqref{equ:c2-bound-12}, and \eqref{equ:c2-bound-21} still hold. For example, for \eqref{equ:c2-flow-1}, we have
$$
2\hat{x}_{1,1}+\hat{x}_{1,2}+\hat{x}_{2,1}+\hat{y}_1 = 2x_{1,1}^* + x_{1,2}^*+x_{2,1}^* + y_1^* + \left[\frac{\theta_1+2\lambda_1+\lambda_2}{\theta_1+2\lambda_1+\lambda_2} - 1\right]\varepsilon = \lambda_1.
$$
For \eqref{equ:c2-bound-22}, since $\theta_2 x_{2,2}^* < \lambda_2y_2^*$ holds with strict inequality, we can guarantee that $\theta_2 \hat{x}_{2,2} < \lambda_2 \hat{y}_2$ still holds by choosing a small enough value of $\varepsilon$. Similarly, non-negative constraints \eqref{equ:c2-nonneg-x} and \eqref{equ:c2-nonneg-y} hold since $x_{2,1}^* > 0$ and $y_2^* > 0$.

Under this new feasible solution, the total cost is calculated as
\begin{align*}
        & c_{(1)}(\hat{x}_{1,1}+\hat{y}_1) + c_{(2)}(\hat{x}_{2,2}+\hat{y}_2) + c_{(1,2)} (\hat{x}_{1,2} + \hat{x}_{2,1}) \\
        = & c_{(1)}(x_{1,1}^*+y_1^*) + c_{(2)}(x_{2,2}^*+y_2^*) + c_{(1,2)} (x_{1,2}^* + x_{2,1}^*) \\
         & + \varepsilon\left[ c_{(1)}\frac{\theta_1+\lambda_1}{\theta_1+2\lambda_1+\lambda_2} + \frac{c_{(2)}}{2}\left(\frac{\theta_1+2\lambda_1}{\theta_1+2\lambda_1+\lambda_2} - \frac{\theta_2}{\lambda_1}\right) - c_{(1,2)} \frac{\theta_1+2\lambda_1}{\theta_1+2\lambda_1+\lambda_2}\right] \\
        \stackrel{(a)}{\leq} & c_{(1)}(x_{1,1}^*+y_1^*) + c_{(2)}(x_{2,2}^*+y_2^*) + c_{(1,2)} (x_{1,2}^* + x_{2,1}^*) \\
         & + \varepsilon\left[ c_{(1)}\frac{\theta_1+\lambda_1}{\theta_1+2\lambda_1+\lambda_2} + \frac{c_{(2)}}{2}\left(\frac{\theta_1+2\lambda_1}{\theta_1+2\lambda_1+\lambda_2} - \frac{\theta_2}{\lambda_1}\right) - c_{(1)} \frac{\theta_1+\lambda_1}{\theta_1+2\lambda_1+\lambda_2} - c_{(2)} \frac{\lambda_1}{\theta_1+2\lambda_1+\lambda_2}\right] \\
        = & c_{(1)}(x_{1,1}^*+y_1^*) + c_{(2)}(x_{2,2}^*+y_2^*) + c_{(1,2)} (x_{1,2}^* + x_{2,1}^*) + \frac{\varepsilon c_{(2)}}{2} \left(\frac{\theta_1}{\theta_1+2\lambda_1+\lambda_2}-\frac{\theta_2}{\lambda_1} \right) \\
        \stackrel{(b)}{<} & c_{(1)}(x_{1,1}^*+y_1^*) + c_{(2)}(x_{2,2}^*+y_2^*) + c_{(1,2)} (x_{1,2}^* + x_{2,1}^*),
    \end{align*}
where (a) is because $c_{(1,2)}\geq c_{(1)}$ and $c_{(1,2)} \geq c_{(2)}$, and (b) is because $\theta_1 \leq \theta_2$. Therefore, a new feasible solution with a lower cost is obtained, leading to a contradiction. \hfill $\square$
\end{proof}

\begin{lemma} 
\label{lemma:n2-y2-nonneg}
    All optimal solutions satisfy $y_2^* > 0$ if $\Delta_2(\lambda_1,\lambda_2)<0$ or $\Delta_3(\lambda_1,\lambda_2)<0$, and all optimal solutions satisfy $\theta_1 x_{1,1}^*=\lambda_1 y_1^*$ if $\Delta_2(\lambda_1,\lambda_2)<0$.
\end{lemma}

\begin{proof} {Proof.}

The idea of proving this lemma is the same as \Cref{lemma:n2-y1-nonneg}. Again prove this lemma by contradiction. Consider the case where $y_2^*=0$. We have $x_{2,1}^*=x_{2,2}^*=0, x_{1,2}^*=\lambda_2 > 0$. \eqref{equ:c2-bound-12} leads to $y_1^* \geq \theta_1$, and \eqref{equ:c2-flow-1} leads to 
$$
   x_{1,1}^* = \frac{\lambda_1 -x_{1,2}^* - y_1^*}{2}  \leq \frac{\lambda_1 - \lambda_2 -\theta_1}{2} = \frac{\Delta_3(\lambda_1,\lambda_2)}{2}.
$$
Thus if $\Delta_3(\lambda_1,\lambda_2)<0$, $x_{1,1}^*$ will become negative, which immediately leads to the contradiction.

When $\Delta_3(\lambda_1,\lambda_2)\geq0$, we continue by following the proving process of \Cref{lemma:n2-y1-nonneg} but swap all indexes of 1 and 2. Similarly we can define a solution $\boldsymbol{\hat{x}}, \boldsymbol{\hat{y}}$ with the total cost calculated as
\begin{align*}
        & c_{(1)}(\hat{x}_{1,1}+\hat{y}_1) + c_{(2)}(\hat{x}_{2,2}+\hat{y}_2) + c_{(1,2)} (\hat{x}_{1,2} + \hat{x}_{2,1}) \\
        = & c_{(1)}(x_{1,1}^*+y_1^*) + c_{(2)}(x_{2,2}^*+y_2^*) + c_{(1,2)} (x_{1,2}^* + x_{2,1}^*) \\
         & + \varepsilon\left[\frac{c_{(1)}}{2}\left(\frac{\theta_2+2\lambda_2}{\theta_2+\lambda_1+2\lambda_2} - \frac{\theta_1}{\lambda_2}\right) + c_{(2)}\frac{\theta_2+\lambda_2}{\theta_2+\lambda_1+2\lambda_2} - c_{(1,2)} \frac{\theta_2+2\lambda_2}{\theta_2+\lambda_1+2\lambda_2}\right] \\
        = & c_{(1)}(x_{1,1}^*+y_1^*) + c_{(2)}(x_{2,2}^*+y_2^*) + c_{(1,2)} (x_{1,2}^* + x_{2,1}^*) + \varepsilon\frac{\theta_2+2\lambda_2}{\theta_2+\lambda_1+2\lambda_2}\Delta_2(\lambda_1,\lambda_2),
    \end{align*}
thus if $\Delta_2(\lambda_1,\lambda_2)<0$, a new feasible solution with a lower cost is obtained, also leading to a contradiction. \hfill $\square$
\end{proof}

Using the above lemmas, we are now ready to give a proof of \Cref{lemma:n2-solution}. We prove \Cref{lemma:n2-solution-same-theta} later by treating it as a special case.

\begin{proof}{Proof of \Cref{lemma:n2-solution}.} We first show that an optimal solution to \eqref{equ:c2} is characterized by \Cref{tab:n2-solution}. Note that
\begin{align*}
    \Delta_1(\lambda_1,\lambda_2)-\Delta_2(\lambda_1,\lambda_2) &= c_{(1)}\frac{\theta_1+\lambda_1}{\theta_1+2\lambda_1} - \frac{c_{(1)}}{2}\left[1 - \frac{\theta_1(\theta_2+\lambda_1+2\lambda_2)}{\lambda_2(\theta_2+2\lambda_2)}\right]\\
    &> \frac{c_{(1)}}{2} - \frac{c_{(1)}}{2}=0.
\end{align*}

Below, we prove the four cases in \Cref{tab:n2-solution} one by one.

\smallskip

\textbf{Case (i): } $\Delta_1(\lambda_1,\lambda_2) \leq 0$, and thus $\Delta_2(\lambda_1,\lambda_2)<0$. According to Lemmas \ref{lemma:n2-y1-nonneg} and \ref{lemma:n2-y2-nonneg}, $y_1^*,y_2^*>0$, and $\theta_i x_{i,i}^* = \lambda_i y_i^*$, $i=1,2$. According to \Cref{lemma:n2-cross-binding}, there exists an optimal solution satisfying $x_{1,2}^*=x_{2,1}^*=0$.

\smallskip

\textbf{Case (ii): } $\Delta_1(\lambda_1,\lambda_2) > 0, \Delta_2(\lambda_1,\lambda_2) < 0$. According to Lemmas \ref{lemma:n2-y1-nonneg} and \ref{lemma:n2-y2-nonneg}, $y_1^*,y_2^*>0$, and $\theta_i x_{i,i}^* = \lambda_i y_i^*$, $i=1,2$. According to \Cref{lemma:n2-cross-binding}, the optimal solution must also satisfy $\theta_1x_{1,2}^*=\lambda_2y_1^*$ and $\theta_2x_{2,1}^*=\lambda_1y_2^*$.

\smallskip

\textbf{Case (iii): } $\Delta_2(\lambda_1,\lambda_2) \geq 0, \Delta_3(\lambda_1,\lambda_2) < 0$, and thus $\Delta_1(\lambda_1,\lambda_2) > 0$. According to Lemmas \ref{lemma:n2-y1-nonneg} and \ref{lemma:n2-y2-nonneg}, $y_1^*,y_2^*>0$, and $\theta_2 x_{2,2}^* = \lambda_2 y_2^*$. According to \Cref{lemma:n2-cross-binding}, $\theta_1x_{1,2}^*=\lambda_2y_1^*$ and $\theta_2x_{2,1}^*=\lambda_1y_2^*$. Then all we need is to prove that there is an optimal solution with $x_{1,1}^*=0$.

Assume for contradiction that $x_{1,1}^*>0$. Again we can re-use the same method to construct another feasible solution as in Lemmas \ref{lemma:n2-y1-nonneg} and \ref{lemma:n2-y2-nonneg}. Given a small value of $\varepsilon > 0$, let $\hat{y}_2=y_2^*-\theta_2\varepsilon/(\theta_2+\lambda_1+2\lambda_2)$, $\hat{x}_{2,2} = x_{2,2}^*-\lambda_2\varepsilon/(\theta_2+\lambda_1+2\lambda_2)$, $\hat{x}_{2,1} = x_{2,1}^*-\lambda_1\varepsilon/(\theta_2+\lambda_1+2\lambda_2)$, $\hat{x}_{1,2} = x_{1,2}^* + \varepsilon$, $\hat{y}_1 = y_1^* + \theta_1\varepsilon/\lambda_2$, and $\hat{x}_{1,1} = x_{1,1}^* - \left[(\theta_2+2\lambda_2)/(\theta_2+\lambda_1+2\lambda_2)+\theta_1/\lambda_2 \right]\varepsilon/2$. The new total cost is calculated as
\begin{align*}
        & c_{(1)}(\hat{x}_{1,1}+\hat{y}_1) + c_{(2)}(\hat{x}_{2,2}+\hat{y}_2) + c_{(1,2)} (\hat{x}_{1,2} + \hat{x}_{2,1}) \\
        = & c_{(1)}(x_{1,1}^*+y_1^*) + c_{(2)}(x_{2,2}^*+y_2^*) + c_{(1,2)} (x_{1,2}^* + x_{2,1}^*) \\
         & - \varepsilon\left[\frac{c_{(1)}}{2}\left(\frac{\theta_2+2\lambda_2}{\theta_2+\lambda_1+2\lambda_2} - \frac{\theta_1}{\lambda_2}\right) + c_{(2)}\frac{\theta_2+\lambda_2}{\theta_2+\lambda_1+2\lambda_2} - c_{(1,2)} \frac{\theta_2+2\lambda_2}{\theta_2+\lambda_1+2\lambda_2}\right] \\
        = & c_{(1)}(x_{1,1}^*+y_1^*) + c_{(2)}(x_{2,2}^*+y_2^*) + c_{(1,2)} (x_{1,2}^* + x_{2,1}^*) - \varepsilon\frac{\theta_2+2\lambda_2}{\theta_2+\lambda_1+2\lambda_2}\Delta_2(\lambda_1,\lambda_2) \\
        \leq & c_{(1)}(x_{1,1}^*+y_1^*) + c_{(2)}(x_{2,2}^*+y_2^*) + c_{(1,2)} (x_{1,2}^* + x_{2,1}^*),
    \end{align*}
thus by decreasing $x_{1,1}^*$, we can obtain another feasible solution with a total cost no more than the original cost. This is sufficient to prove that $x_{1,1}^*=0$ corresponds to an optimal solution.

\smallskip

\textbf{Case (iv): } $\Delta_2(\lambda_1,\lambda_2) \geq 0, \Delta_3(\lambda_1,\lambda_2) \geq 0$, and thus $\Delta_1(\lambda_1,\lambda_2) > 0$. \Cref{lemma:n2-y1-nonneg} leads to $y_1^*>0$ and $\theta_2 x_{2,2}^* = \lambda_2 y_2^*$. According to \Cref{lemma:n2-cross-binding}, $\theta_1x_{1,2}^*=\lambda_2y_1^*$ and $\theta_2x_{2,1}^*=\lambda_1y_2^*$. Then all we need is to prove that $y_2^*=0$.

Again assume $y_2^*>0$. Since $\theta_1x_{1,2}^*=\lambda_2y_1^*$ and from \eqref{equ:c2-flow-2}, $x_{1,2}^*<\lambda_2$, we have $y_1^*<\theta_1$. By subtracting \eqref{equ:c2-flow-2} from \eqref{equ:c2-flow-1}, we obtain
\begin{align*}
\lambda_1-\lambda_2 = (y_1^*+2x_{1,1}^*)-(y_2^*+2x_{2,2}^*) < \theta_1 + 2x_{1,1}^*.
\end{align*}
Since $\Delta_3(\lambda_1,\lambda_2) \geq 0$, $\lambda_1-\lambda_2\geq\theta_1$. Therefore, in order to satisfy $\lambda_1-\lambda_2 < \theta_1 + 2x_{1,1}^*$, we must have $x_{1,1}^* > 0$.

Given $y_2^*,x_{2,1}^*,x_{2,2}^*,x_{1,1}^* > 0$, we can immediately re-use the idea in case (iii) to construct a new feasible solution by decreasing $y_2^*$. The new feasible solution has a total cost no more than the original cost, thus $y_2^*=0$ corresponds to an optimal solution.

\smallskip

In addition, from \Cref{lemma:n2-y1-nonneg}, all optimal solutions satisfy $y_1^*>0$. From \Cref{lemma:n2-y2-nonneg}, all optimal solutions satisfy $y_2^*>0$ if $\Delta_2(\lambda_1, \lambda_2)<0$ or $\Delta_3(\lambda_1, \lambda_2)<0$. On the other hand, when $\Delta_2(\lambda_1, \lambda_2), \Delta_3(\lambda_1, \lambda_2)\geq0$, according to \Cref{tab:n2-solution}, there exists an optimal solution with $y_2^*=0$. Therefore, all optimal solutions satisfy $y_2^*>0$ if and only if $\Delta_2(\lambda_1, \lambda_2)<0$ or $\Delta_3(\lambda_1, \lambda_2)<0$.

\smallskip

Lastly, by fully characterizing the optimal solutions, we can explicitly determine the values of $y_1^*$ and $y_2^*$ under four different cases, thereby providing the optimal solutions in closed form.

\smallskip

\textbf{Case (i):} Solving \eqref{equ:c2-flow-1} and \eqref{equ:c2-flow-2} when \eqref{equ:c2-bound-11}, \eqref{equ:c2-bound-22} are binding and $x_{1,2}^*=x_{2,1}^*=0$, we obtain 
\begin{align*}
    y_1^* = \frac{\theta_1\lambda_1}{\theta_1+2\lambda_1},\ y_2^* = \frac{\theta_2\lambda_2}{\theta_2+2\lambda_2},
\end{align*}
and $x_{1,2}^*=x_{2,1}^*=0$, $x_{1,1}^*=\lambda_1y_1^*/\theta_1$, $x_{2,2}^*=\lambda_2y_2^*/\theta_2$.

\smallskip

\textbf{Case (ii):} Solving \eqref{equ:c2-flow-1} and \eqref{equ:c2-flow-2} when \eqref{equ:c2-bound-11}, \eqref{equ:c2-bound-12}, \eqref{equ:c2-bound-21} and \eqref{equ:c2-bound-22} are all binding, we obtain 
\begin{align*}
    y_1^* &= \frac{\theta_1\lambda_1(\theta_2+\lambda_1+\lambda_2)}{2(\lambda_1+\lambda_2)^2+\theta_1(\lambda_1+2\lambda_2)+\theta_2(\lambda_2+2\lambda_1)+\theta_1\theta_2}, \\
    y_2^* &= \frac{\theta_2\lambda_2(\theta_1+\lambda_1+\lambda_2)}{2(\lambda_1+\lambda_2)^2+\theta_1(\lambda_1+2\lambda_2)+\theta_2(\lambda_2+2\lambda_1)+\theta_1\theta_2},
\end{align*}
and $x_{1,1}^*=\lambda_1y_1^*/\theta_1$,$x_{1,2}^*=\lambda_2y_1^*/\theta_1$,$x_{2,1}^*=\lambda_1y_2^*/\theta_2$,$x_{2,2}^*=\lambda_2y_2^*/\theta_2$.

\smallskip 

\textbf{Case (iii):} Solving \eqref{equ:c2-flow-1} and \eqref{equ:c2-flow-2} subject to the binding constraints \eqref{equ:c2-bound-12}, \eqref{equ:c2-bound-21}, \eqref{equ:c2-bound-22}, and $x_{1,1}^*=0$, we obtain
\begin{align*}
    y_1^* &= \frac{\theta_1\lambda_1(\theta_2+\lambda_1+\lambda_2)}{2\lambda_2^2+\theta_1(\lambda_1+2\lambda_2)+\theta_2\lambda_2+\theta_1\theta_2}, \\
    y_2^* &= \frac{\theta_2\lambda_2(\theta_1-\lambda_1+\lambda_2)}{2\lambda_2^2+\theta_1(\lambda_1+2\lambda_2)+\theta_2\lambda_2+\theta_1\theta_2},
\end{align*}
and $x_{1,1}^*=0$,$x_{1,2}^*=\lambda_2y_1^*/\theta_1$,$x_{2,1}^*=\lambda_1y_2^*/\theta_2$,$x_{2,2}^*=\lambda_2y_2^*/\theta_2$.

\smallskip
\textbf{Case (iv):} Solving \eqref{equ:c2-flow-1} and \eqref{equ:c2-flow-2} subject to $y_2^*=x_{2,1}^*=x_{2,2}^*=0$ and the binding constraint \eqref{equ:c2-bound-12}, we obtain
\begin{align*}
    y_1^* = \theta_1,\ x_{1,1}^*= \frac{\lambda_1-\lambda_2-\theta_1}{2},\ x_{1,2}^*=\lambda_2.
\end{align*}
\hfill $\square$
\end{proof}

\begin{proof}{Proof of \Cref{lemma:n2-solution-same-theta}}
Notice that when $\theta_1=\theta_2=\theta$, 
\begin{align*}
    (\theta+2\lambda_2)\Delta_2(\lambda_1,\lambda_2) &= \frac{c_{(1)}}{2}\left[(\theta+2\lambda_2) - \frac{\theta(\theta+\lambda_1+2\lambda_2)}{\lambda_2}\right] +c_{(2)}(\theta+\lambda_2) - c_{(1,2)}(\theta+2\lambda_2) \\
    & \leq \frac{c_{(1)}}{2}\left[(\theta+2\lambda_2) - \frac{\theta(\theta+\lambda_1+2\lambda_2)}{\lambda_2}\right] +c_{(2)}(\theta+\lambda_2) - c_{(1)}\lambda_2 - c_{(2)}(\theta+\lambda_2) \\
    & = \frac{c_{(1)}\theta}{2}\left(1-\frac{\theta+\lambda_1+2\lambda_2}{\lambda_2}\right) \\
    & < 0,
\end{align*}
thus $\Delta_2(\lambda_1,\lambda_2) < 0$ always holds. Therefore, we do not need to consider the last two cases in \Cref{lemma:n2-solution}. For the first two cases in \Cref{lemma:n2-solution}, we can solve \eqref{equ:c2-flow-1} and \eqref{equ:c2-flow-2} with respect to $y_1^*$ and $y_2^*$, and obtain the optimal solution. This leads to the results stated in \Cref{tab:n2-solution-same-theta}. In addition, since $\Delta_2(\lambda_1,\lambda_2) < 0$, according to \Cref{lemma:n2-y1-nonneg} and \Cref{lemma:n2-y2-nonneg}, all optimal solutions satisfy $y_1^*, y_2^*>0$.
\hfill $\square$
\end{proof}

\subsection{Proof of \Cref{prop:n2-concave-same-theta}}
\smallskip
\begin{lemma}
\label{lemma:convex}
Given any 2-dimensional vector $\boldsymbol{\alpha}$ and any nonnegative vector $\boldsymbol{\beta} \geq 0$, the function $(\boldsymbol{\alpha}^T \boldsymbol{\lambda})^2/(\theta+\boldsymbol{\beta}^T\boldsymbol{\lambda})$ is convex with respect to $\boldsymbol{\lambda} \geq \boldsymbol{0}$.
\end{lemma}
\begin{proof} {Proof.}
    Define the function $f(a,b) = a^2/b$, we can verify that $f(a,b)$ is convex when $b >0$, based on its positive semidefinite Hessian matrix
    $\begin{bmatrix}
         2/b & -2a/b^2\\
        -2a/b^2 & 2a^2/b^3
    \end{bmatrix}$: the diagonal elements are positive, and the determinant is 0. Then the function $f((\boldsymbol{\alpha}^T \boldsymbol{\lambda})^2,\theta+\boldsymbol{\beta}^T\boldsymbol{\lambda})$ is also convex since $\theta+\boldsymbol{\beta}^T\boldsymbol{\lambda} > \boldsymbol{0}$ and an affine mapping preserves convexity (see Section 3.2.2, \citealt{Boyd_Vandenberghe_2004}). $\hfill\square$
\end{proof}

\begin{proof}{Proof of \Cref{prop:n2-concave-same-theta}.} Based on the results of \Cref{lemma:n2-solution-same-theta}, define two cost functions as follows:
\begin{align*}
    c^{(1)}(\boldsymbol{\lambda}) & = \frac{c_{(1)}\lambda_1(\theta+\lambda_1)}{\theta+2\lambda_1} + \frac{c_{(2)}\lambda_2(\theta+\lambda_2)}{\theta+2\lambda_2},\\
    c^{(2)}(\boldsymbol{\lambda}) & = \frac{c_{(1)}\lambda_1(\theta+\lambda_1)+c_{(2)}\lambda_2(\theta+\lambda_2)+2c_{(1,2)}\lambda_1\lambda_2}{\theta+2\lambda_1+2\lambda_2}.
\end{align*}
where the first function $c^{(1)}(\boldsymbol{\lambda})$ corresponds to the optimal cost function when $\Delta_1(\lambda_1,\lambda_2)\leq0$, and $c^{(2)}(\boldsymbol{\lambda})$ is the optimal cost when $\Delta_1(\lambda_1,\lambda_2)>0$. 

For $c^{(1)}(\boldsymbol{\lambda})$, notice that
\begin{align*}
    c^{(1)}(\boldsymbol{\lambda}) &= \frac{c_{(1)}\lambda_1(\theta+\lambda_1)}{\theta+2\lambda_1} + \frac{c_{(2)}\lambda_2(\theta+\lambda_2)}{\theta+2\lambda_2} \\
    & = c_{(1)}\lambda_1+c_{(2)}\lambda_2 - c_{(1)}\frac{\lambda_1^2}{\theta+2\lambda_1} - c_{(2)} \frac{\lambda_2^2}{\theta+2\lambda_2}.
\end{align*}
Using \Cref{lemma:convex}, by choosing $\boldsymbol{\alpha} = (1,0)^T$ and $\boldsymbol{\beta} = (2,0)^T$, we obtain the result that $\lambda_1^2/(\theta+2\lambda_1)$ is convex. Similarly, by choosing $\boldsymbol{\alpha} = (0,1)^T$ and $\boldsymbol{\beta} = (0,2)^T$, $\lambda_2^2/(\theta+2\lambda_2)$ is convex. $c^{(1)}(\boldsymbol{\lambda})$ is expressed as a linear function minus convex functions, and is therefore concave.

Note that if $c_{(1)}+c_{(2)}-c_{(1,2)} \leq 0$, $$
\Delta_1(\lambda_1,\lambda_2) = c_{(1)}\frac{\theta+\lambda_1}{\theta+2\lambda_1} + c_{(2)}\frac{\theta+\lambda_2}{\theta+2\lambda_2} - c_{(1,2)} < c_{(1)} + c_{(2)} - c_{(1,2)} \leq 0,$$ thus in this case, $c(\boldsymbol{\lambda}) = c^{(1)}(\boldsymbol{\lambda})$ and is concave. The following analysis considers the case when $c_{(1)}+c_{(2)}-c_{(1,2)} > 0$.

For $c^{(2)}(\boldsymbol{\lambda})$, we can reformulate the function as
\begin{align*}
    c^{(2)}(\boldsymbol{\lambda}) =& \frac{c_{(1)}\lambda_1(\theta+\lambda_1)+c_{(2)}\lambda_2(\theta+\lambda_2)+2c_{(1,2)}\lambda_1\lambda_2}{\theta+2\lambda_1+2\lambda_2} \\
    = &c_{(1)}\lambda_1+c_{(2)}\lambda_2 - \frac{c_{(1)}\lambda_1^2+c_{(2)}\lambda_2^2+2[c_{(1)}+c_{(2)}-c_{(1,2)}]\lambda_1\lambda_2}{\theta+2\lambda_1+2\lambda_2} \\
    = & c_{(1)}\lambda_1+c_{(2)}\lambda_2 - [c_{(1,2)}-c_{(2)}]\frac{\lambda_1^2}{\theta+2\lambda_1+2\lambda_2}-[c_{(1,2)}-c_{(1)}]\frac{\lambda_2^2}{\theta+2\lambda_1+2\lambda_2}\\
    & - [c_{(1)}+c_{(2)}-c_{(1,2)}]\frac{(\lambda_1+\lambda_2)^2}{\theta+2\lambda_1+2\lambda_2}.
\end{align*}
Again, according to \Cref{lemma:convex}, we can prove that $\lambda_1^2/(\theta+2\lambda_1+2\lambda_2)$, $\lambda_2^2/(\theta+2\lambda_1+2\lambda_2)$, and $(\lambda_1+\lambda_2)^2/(\theta+2\lambda_1+2\lambda_2)$ are all convex, by choosing $\boldsymbol{\beta} = (2,2)^T$, $\boldsymbol{\alpha} = (1,0)^T, (0,1)^T$ and $(1,1)^T$ respectively. Since $c_{(1,2)} \geq c_{(1)}$, $c_{(1,2)} \geq c_{(2)}$, and $c_{(1,2)} < c_{(1)} + c_{(2)}$, thus $c^{(2)}(\boldsymbol{\lambda})$ can also be expressed as a linear function minus convex functions, and is therefore concave.

Since two types of optimal solutions given in \Cref{lemma:n2-solution-same-theta} remain feasible for any $\lambda_1,\lambda_2 \geq 0$, we can re-write the optimal cost function $c(\boldsymbol{\lambda})$ as the minimization of two functions:
\begin{align*}
c(\boldsymbol{\lambda}) & = \min \{ c^{(1)}(\boldsymbol{\lambda}), c^{(2)}(\boldsymbol{\lambda}) \}.
\end{align*}
Given that the minimization preserves concavity, $c(\boldsymbol{\lambda})$ is concave. 
\hfill $\square$
\end{proof}

\subsection{Proof of \Cref{prop:n2-tight}}

When $N=1$, according to \Cref{lemma:n1-solution} and \Cref{prop:n1-concave}, the optimal solution to the LP model is given by $y_1^*=\lambda_1\theta /(\theta +2\lambda_1)$ and $x_{1,1}^* = \lambda_1^2/(\theta +2\lambda_1)$, and $c(\boldsymbol{\lambda})$ is concave. Now consider the following matching policy: when an agent arrives, if there is an agent waiting, match these two agents; otherwise, let this agent wait. It is easy to verify that this matching policy results in only 2 states for the system: state 0 means that no agent is waiting, and state 1 means that there is an agent waiting. This leads to the Markov chain shown in \Cref{fig:chain_1} \citep{yan2023pricing}.
\begin{figure}[ht]
\centering
    \subfloat[$N=1$]{\includegraphics[width=0.3\linewidth]{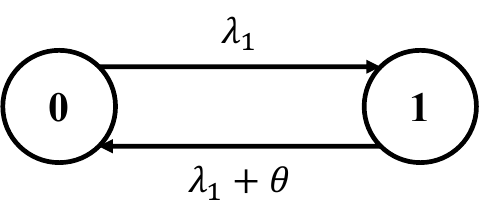} \label{fig:chain_1}}
 \subfloat[$N=2$]{\includegraphics[width=0.4\linewidth]{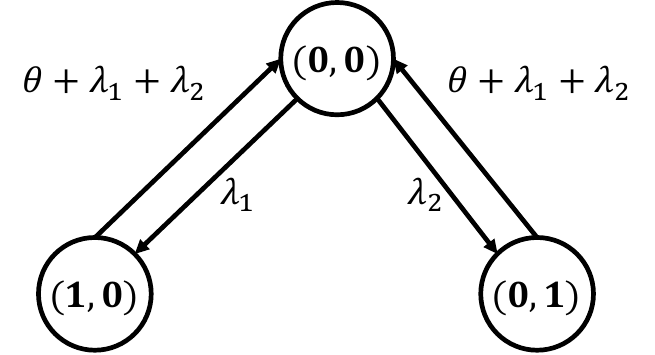} \label{fig:chain_2}}
 \caption{Markov chains when $N\leq 2$.}
\end{figure}

Solving this Markov chain leads to the steady-state distributions of $\pi_0 = (\theta +\lambda_1)/(\theta +2\lambda_1)$ and $\pi_1 = \lambda_1/(\theta +2\lambda_1)$. The rate at which agents get matched is thus $\pi_1 \lambda_1=x_{1,1}^*$, and the rate at which agents depart solo is $\pi_1 \theta  = y_1^*$. Therefore, $x_{1,1}^*$ and $y_1^*$ correspond to the steady-state flow under this matching policy. Combining with the property that the LP model gives a lower bound of the MDP cost, we can see that the LP relaxation is tight.

When $N=2$, according to \Cref{lemma:n2-solution-same-theta}, two possible solution structures can also be determined by checking the sign of $\Delta_1(\lambda_1,\lambda_2)$. When $\Delta_1(\lambda_1,\lambda_2)\leq0$, an optimal solution of LP satisfies $x_{1,2}^*=x_{2,1}^*=0$,  $y_i^*=\lambda_i\theta /(\theta +2\lambda_i)$, and $x_{i,i}^* = \lambda_i^2/(\theta +2\lambda_i)$ ($\forall i=1,2$).  We can implement a matching policy that never matches two types of agents together, and matches the same types of agents as the case of $N=1$. Therefore, the LP relaxation is tight in this case.

When $\Delta_1(\lambda_1,\lambda_2)>0$, consider the matching policy that matches two agents whenever possible, regardless of their types. Under this policy, there will be three possible states: $(0,0)$, $(1,0)$ and $(0,1)$, where the first and second indices denote the number of type 1 and type 2 agents waiting, respectively. This leads to the Markov chain shown in \Cref{fig:chain_2}.

Denote the steady-state distributions of this Markov chain as $\pi_{(0,0)}$, $\pi_{(1,0)}$, and $\pi_{(0,1)}$. The values of $\pi_{(0,0)}$, $\pi_{(1,0)}$, and $\pi_{(0,1)}$ can be uniquely solved using the following equations:
\begin{align*}
    \pi_{(0,0)} + \pi_{(1,0)} + \pi_{(0,1)} &= 1,\\
    (\theta +\lambda_1+\lambda_2) \pi_{(1,0)} &= \lambda_1  \pi_{(0,0)}, \\
    (\theta +\lambda_1+\lambda_2) \pi_{(0,1)} &= \lambda_2  \pi_{(0,0)}.
\end{align*}
The solution is given by $\pi_{(1,0)}=\lambda_1/(\theta + 2\lambda_1 + 2\lambda_2) $, $\pi_{(0,1)}=\lambda_2/(\theta + 2\lambda_1 + 2\lambda_2) $, and  $\pi_{(0,0)}=(\theta+\lambda_1+\lambda_2)/(\theta + 2\lambda_1 + 2\lambda_2) $. We then have $x_{1,j}^*=\lambda_jy_1^*/\theta =\lambda_j \pi_{(1,0)}$ and $x_{2,j}^*=\lambda_jy_2^*/\theta =\lambda_j \pi_{(0,1)}$ for $j=1,2$. Therefore, $\boldsymbol{x}^*$ and $\boldsymbol{y}^*$ correspond to the steady-state flow under this matching policy, and thus the fluid relaxation is tight. $\hfill\square$

\subsection{Proofs of Theorems \ref{thm:n-weak-concave-same-theta}, \ref{thm:n-weak-concave} and Propositions \ref{prop:n-concave-same-theta}, \ref{prop:n-weak-concave-same-theta-upper-bound}, \ref{prop:n-concave}, \ref{prop:n-weak-concave-upper-bound}, \ref{prop:n-weak-concave-same-theta-both-bound}, and \ref{prop:n-weak-concave-both-bound} } \label{appendix:thm-n-concave}

Our proof relies on the following two lemmas characterizing the optimal solution structure. 

\begin{lemma} \label{lemma:n-concave-1}
    $\forall i,j \in [N]$, if
    $$
    c_{(i)}\frac{\theta_i+\lambda_i}{\theta_i+2\lambda_i} + c_{(j)}\frac{\theta_j+\lambda_j}{\theta_j+2\lambda_j} - c_{(i,j)} < 0,
    $$
    the optimal solution must have $x_{i,j}^*=x_{j,i}^*=0$.
\end{lemma} 
\begin{proof} {Proof.}
    This lemma is a direct generalization of the first part of \Cref{lemma:n2-cross-binding}. Assume by contradiction that there exist $i,j\in[N]$, such that either $x_{i,j}^*>0$ or $x_{j,i}^* > 0$ when $c_{(i)}(\theta_i+\lambda_i)/(\theta_i+2\lambda_i) + c_{(j)}(\theta_j+\lambda_j)/(\theta_j+2\lambda_j) - c_{(i,j)} < 0$. If $x_{i,j}^* > 0$, we obtain a new solution by letting
    \begin{align*}
        \hat{x}_{i,j} &= x_{i,j}^* - \varepsilon, \\
        \hat{y}_{i} &= y_{i}^* + \varepsilon \frac{\theta_i}{\theta_i+2\lambda_i}, \\
        \hat{x}_{i,i} &= x_{i,i}^* + \varepsilon \frac{\lambda_i}{\theta_i+2\lambda_i}, \\
        \hat{y}_{j} &= y_{j}^* + \varepsilon \frac{\theta_j}{\theta_j+2\lambda_j}, \\
        \hat{x}_{j,j} &= x_{j,j}^* + \varepsilon \frac{\lambda_j}{\theta_j+2\lambda_j},
    \end{align*}
    and keeping the remaining variables unchanged, where $\varepsilon > 0$ is a sufficiently small positive value. Similar to the approach in \Cref{lemma:n2-cross-binding}, we can verify this is still a feasible solution to \eqref{equ:cn}, since we guarantee that $2\hat{x}_{i,i}+\hat{x}_{i,j}+\hat{y}_i = 2x_{i,i}^* + x_{i,j}^* + y_i^*$ and $2\hat{x}_{j,j}+\hat{x}_{i,j}+\hat{y}_j = 2x_{j,j}^* + x_{i,j}^* + y_j^*$. Compared to the original solution, the difference in cost under this new feasible solution is given by
\begin{align*}
    & \left[c_{(i)}(\hat{x}_{i,i}+\hat{y}_i) + c_{(j)}(\hat{x}_{j,j}+\hat{y}_j) + c_{(i,j)} \hat{x}_{i,j} \right] - \left[ c_{(i)}(x_{i,i}^*+y_i^*) + c_{(j)}(x_{j,j}^*+y_j^*) + c_{(i,j)} x_{i,j}^* \right]\\
    = & \varepsilon\left[ c_{(i)}\frac{\theta_i+\lambda_i}{\theta_i+2\lambda_i} + c_{(j)}\frac{\theta_j+\lambda_j}{\theta_j+2\lambda_j} - c_{(i,j)}\right] \\
    < & 0.
\end{align*}
Therefore, if $x_{i,j}^*>0$, by decreasing the value of $x_{i,j}$, we can always obtain another feasible solution with lower cost, meaning that $x_{i,j}^*=0$ must be an optimal solution. Similarly, $x_{j,i}^*=0$. \hfill $\square$
\end{proof}

\begin{lemma} \label{lemma:n-concave-2}
    For $i\in[N]$, define $X_i = \{j~|~\theta_i x_{i,j}^* = \lambda_j y_i^*\}$. If
    $$
    \sum_{j\in X_i \setminus \{i\}} \lambda_j\frac{\theta_j}{\theta_j+2\lambda_j+\lambda_i}-\theta_i <0 ,
    $$
    we must have $i \in X_i$.
\end{lemma}

\begin{proof} {Proof.}
     This lemma is a generalization of the second statements in Lemmas \ref{lemma:n2-y1-nonneg} and \ref{lemma:n2-y2-nonneg}. Again, assume by contradiction that $i\notin X_i$ (i.e., $\theta_i x_{i,i}^* < \lambda_i y_i^*$), then we must have $y_i^*>0$. Now we try to give another feasible solution with a smaller total cost. Given a small value of $\varepsilon > 0$, let 
     \begin{align*}
        \hat{y}_j&=y_j^*+\varepsilon\lambda_j\frac{\theta_j}{\theta_j+2\lambda_j+\lambda_i}, &\forall j\in X_i \setminus \{i\},\\
        \hat{x}_{j,j}&=x_{j,j}^*+\varepsilon\lambda_j\frac{\lambda_j}{\theta_j+2\lambda_j+\lambda_i}, &\forall j\in X_i \setminus \{i\},\\
        \hat{x}_{j,i}&=x_{j,i}^*+\varepsilon\lambda_j\frac{\lambda_i}{\theta_j+2\lambda_j+\lambda_i}, &\forall j\in X_i \setminus \{i\},\\
        \hat{x}_{i,j}&=x_{i,j}^*-\varepsilon\lambda_j, &\forall j\in X_i \setminus \{i\},\\
        \hat{y}_i&=y_{i}^*-\varepsilon\theta_i,\\
        \hat{x}_{i,i}&=x_{i,i}^*+\frac{\varepsilon}{2}\left(\theta_i+\sum_{j\in X_i \setminus \{i\}} \lambda_j\frac{\theta_j+2\lambda_j}{\theta_j+2\lambda_j+\lambda_i}\right),
     \end{align*}
     and keep the remaining variables unchanged. Similar to Lemmas \ref{lemma:n2-y1-nonneg} and \ref{lemma:n2-y2-nonneg}, since $y_i^*>0$ and $x_{i,j}^*=\lambda_jy_i^*/\theta_i>0,\ \forall j\in X_i$, we can also guarantee that the new solution is still feasible. Then compared to the original solution, the difference in cost under this new feasible solution is given by
     \begin{align*}
    & \left[c_{(i)}(\hat{x}_{i,i}+\hat{y}_i) + \sum_{j\in X_i \setminus \{i\}}c_{(j)}(\hat{x}_{j,j}+\hat{y}_j) + \sum_{j\in X_i \setminus \{i\}}c_{(i,j)} (\hat{x}_{i,j}+\hat{x}_{j,i}) \right] \\
    & - \left[ c_{(i)}(x_{i,i}^*+y_i^*) + \sum_{j\in X_i \setminus \{i\}}c_{(j)}(x_{j,j}^*+y_j^*) + \sum_{j\in X_i \setminus \{i\}}c_{(i,j)} (x_{i,j}^*+x_{j,i}^*) \right]\\
    = & \varepsilon\left[ \frac{c_{(i)}}{2} \left( \sum_{j\in X_i \setminus \{i\}} \lambda_j\frac{\theta_j+2\lambda_j}{\theta_j+2\lambda_j+\lambda_i}-\theta_i \right)
    + \sum_{j\in X_i \setminus \{i\}} \left( c_{(j)}\lambda_j\frac{\theta_j+\lambda_j}{\theta_j+2\lambda_j+\lambda_i} - c_{(i,j)} \lambda_j\frac{\theta_j+2\lambda_j}{\theta_j+2\lambda_j+\lambda_i}\right)\right] \\
    \stackrel{(a)}{\leq} & \varepsilon\left[ \frac{c_{(i)}}{2} \left( \sum_{j\in X_i \setminus \{i\}} \lambda_j\frac{\theta_j+2\lambda_j}{\theta_j+2\lambda_j+\lambda_i}-\theta_i \right)
    - \sum_{j\in X_i \setminus \{i\}} \left( c_{(i,j)}\lambda_j\frac{\lambda_j}{\theta_j+2\lambda_j+\lambda_i} \right)\right] \\
    \stackrel{(b)}{\leq} & \varepsilon\left[ \frac{c_{(i)}}{2} \left( \sum_{j\in X_i \setminus \{i\}} \lambda_j\frac{\theta_j}{\theta_j+2\lambda_j+\lambda_i}-\theta_i \right)\right] \\
    < & 0,
    \end{align*}
    where (a) is due to $c_{(i,j)} \geq c_{(j)}$ and (b) is due to $c_{(i,j)} \geq c_{(i)}$. Therefore, we can obtain another feasible solution with a lower cost, leading to the contradiction. \hfill $\square$
\end{proof}

We prove \Cref{thm:n-weak-concave-same-theta,thm:n-weak-concave} and Propositions \ref{prop:n-weak-concave-same-theta-upper-bound}, \ref{prop:n-weak-concave-upper-bound}, \ref{prop:n-weak-concave-same-theta-both-bound}, and \ref{prop:n-weak-concave-both-bound} by establishing a general theorem that encompasses all cases.

\begin{theorem} \label{thm:n-weak-concave-general}
$c(\boldsymbol{\lambda})$ is weakly concave on $\Lambda$ under either of the following conditions:

(i) there exists an integer $K \geq 1$, such that $(K-2)\nu \leq 2$, $e_{(K)} \neq 0.5$ and 
$$
\underline{\lambda}_i > \theta_i \frac{e_{(K)}}{1-2e_{(K)}},\ \forall i \in [N];
$$

(ii) there exists an integer $K \geq 1$, such that $(K - 2)\nu  > 2$, $e_{(K)} \neq 0.5$ and 
$$
\theta_i \frac{e_{(K)}}{1-2e_{(K)}} < \underline{\lambda}_i < \overline{\lambda}_i < \theta_i \max\left\{\frac{1}{K-2}, \frac{1}{(K-2)\nu-2}\right\}\left(1+\frac{1}{\nu}\frac{e_{(K)}}{1-2e_{(K)}}\right),\ \forall i \in [N].
$$
\end{theorem}

\begin{proof}{Proof.}

(i) Consider the first case where $(K-2)\nu\leq2$ and $\underline{\lambda}_i > \theta_i e_{(K)}/(1-2e_{(K)}),\ \forall i \in [N]$. For any $i,j\in[N]\ (i\neq j)$, we have
\begin{align*}
    & c_{(i)} \frac{\theta_i+\lambda_i}{\theta_i+2\lambda_i} + c_{(j)} \frac{\theta_j+\lambda_j}{\theta_j+2\lambda_j} - c_{(i,j)} \\ =& c_{(i)} + c_{(j)} - c_{(i)} \frac{\lambda_i}{\theta_i+2\lambda_i} - c_{(j)} \frac{\lambda_j}{\theta_j+2\lambda_j}  - c_{(i,j)}\\
    \stackrel{(a)}{\leq} & c_{(i)} + c_{(j)} - c_{(i)} \frac{\underline{\lambda}_i}{\theta_i+2\underline{\lambda}_i} - c_{(j)} \frac{\underline{\lambda}_j}{\theta_j+2\underline{\lambda}_j}  - c_{(i,j)}\\
    \stackrel{(b)}{<} & \left(c_{(i)}+c_{(j)}\right) \left(1- e_{(K)}\right)  - c_{(i,j)}\\
    =& \left(c_{(i)}+c_{(j)}\right) \left(1- e_{(K)} - \frac{c_{(i,j)}}{c_{(i)}+c_{(j)}}\right)\\
    =& \left(c_{(i)}+c_{(j)}\right) \left(e_{i,j}- e_{(K)}\right),
\end{align*}
where (a) and (b) are because the function $\lambda_i/(\theta+2\lambda_i)$ is increasing in $\lambda_i > 0$. This means that, according \Cref{lemma:n-concave-1}, for a given type $i\in[N]$, the number of other demand types $j\neq i$ satisfying $x_{i,j}^*>0$ or $x_{j,i}^*>0$ is at most $K-2$.

Now, consider the following expression
$$
\sum_{j\in X_i \setminus \{i\}} \lambda_j\frac{\theta_j}{\theta_j+2\lambda_j+\lambda_i}-\theta_i \leq \theta_N \sum_{j\in X_i \setminus \{i\}} \frac{\lambda_j}{\theta_j+2\lambda_j+\lambda_i}-\theta_1 \stackrel{(a)}{<} \theta_1 \left[\frac{(K-2)\nu}{2}-1\right] \leq 0,
$$
where (a) is due to the size of the set $X_i \setminus \{i\}$ is at most $K-2$ and $\lambda_j/(\theta_j+2\lambda_j+\lambda_i) < 1/2$. Then based on \Cref{lemma:n-concave-2}, we have $i \in X_i$, that is $\theta_i x_{i,i}^*=\lambda_i y_i^*,\ \forall i\in[N]$.

Then we show that $y_i^*>0$ for all $i\in [N]$. Assume by contradiction that there exists $i \in [N]$, $y_i^*=0$. Based on \eqref{equ:cn-bound}, we have $x_{i,j}^*=0,\ \forall j\in[N]$. Then constraint \eqref{equ:cn-flow} for type $i$ leads to $\sum_{j \in [N], j\neq i} x_{j,i}^* = \lambda_i$. For demand type $j$, we have $y_j^* \geq \theta_j x_{j,i}^*/\lambda_i$, $x_{j,j}^* = \lambda_j y_j^*/\theta_j \geq \lambda_j x_{j,i}^*/\lambda_i$, and $x_{j,i}^*+2x_{j,j}^*+y_j^*\leq\lambda_j$, leading to
$$
x_{j,i}^* \leq 
\lambda_i\frac{\lambda_j}{\theta_j+2\lambda_j+\lambda_i} < \frac{\lambda_i}{2}.
$$
Note that since $K-2 \leq (K-2)\nu \leq 2$, there are at most 2 demand types $j \neq i$ such that $x_{j,i}^*>0$, then we have $\sum_{j \in [N], j\neq i} x_{j,i}^* < 2  \lambda_i/2 = \lambda_i$, which contradicts with constraint \eqref{equ:cn-flow}. 

Therefore, we must have $y_i^*>0$ for all $i\in [N]$. Based on \Cref{lemma:n-weak-concave} and \Cref{prop:n-weak-concave-nonzero}, $c(\boldsymbol{\lambda})$ is weakly concave on $\Lambda$.

(ii) Consider the second case where $(K-2)\nu>2$ and
$$
\theta_i \frac{e_{(K)}}{1-2e_{(K)}} < \underline{\lambda}_i < \overline{\lambda}_i < \theta_i \max\left\{\frac{1}{K-2}, \frac{1}{(K-2)\nu-2}\right\}\left(1+\frac{1}{\nu}\frac{e_{(K)}}{1-2e_{(K)}}\right),\ \forall i \in [N].
$$
From the lower bounds, we can get the similar result as the first case, that for a given type $i\in[N]$, the number of other demand types $j\in[N]$ satisfying $x_{i,j}^*>0$ or $x_{j,i}^*>0$ is at most $K-2$.

For the upper bounds, we begin with the case where $1/(K-2) \leq 1/[(K-2)\nu-2]$, and thus
$$
\overline{\lambda}_i < \frac{\theta_i}{(K-2)\nu-2}\left(1+\frac{1}{\nu}\frac{e_{(K)}}{1-2e_{(K)}}\right),\ \forall i \in [N].
$$
For any demand types $i,j\in[N]$, we have
\begin{align*}
\frac{\lambda_j}{\theta_j+2\lambda_j+\lambda_i} & \stackrel{(a)}{\leq} \frac{\overline{\lambda}_j}{\theta_j+2\overline{\lambda}_j+\underline{\lambda}_i}  \stackrel{(b)}{<} \frac{\frac{\theta_j}{(K-2)\nu-2}\left(1+\frac{1}{\nu}\frac{e_{(K)}}{1-2e_{(K)}}\right)}{\theta_j+2\frac{\theta_j}{(K-2)\nu-2}\left(1+\frac{1}{\nu}\frac{e_{(K)}}{1-2e_{(K)}}\right)+\theta_i \frac{e_{(K)}}{1-2e_{(K)}}} \\
& = \frac{\frac{1}{(K-2)\nu-2}\left(1+\frac{1}{\nu}\frac{e_{(K)}}{1-2e_{(K)}}\right)}{1+2\frac{1}{(K-2)\nu-2}\left(1+\frac{1}{\nu}\frac{e_{(K)}}{1-2e_{(K)}}\right)+ \frac{\theta_i}{\theta_j}\frac{e_{(K)}}{1-2e_{(K)}}} \\
& \leq \frac{\frac{1}{(K-2)\nu-2}\left(1+\frac{1}{\nu}\frac{e_{(K)}}{1-2e_{(K)}}\right)}{1+2\frac{1}{(K-2)\nu-2}\left(1+\frac{1}{\nu}\frac{e_{(K)}}{1-2e_{(K)}}\right)+ \frac{1}{\nu}\frac{e_{(K)}}{1-2e_{(K)}}} \\
& = \frac{1+\frac{1}{\nu}\frac{e_{(K)}}{1-2e_{(K)}}}{2\left(1+\frac{1}{\nu}\frac{e_{(K)}}{1-2e_{(K)}}\right)+ [(K-2)\nu-2] \left(1 + \frac{1}{\nu}\frac{e_{(K)}}{1-2e_{(K)}}\right)} \\
& = \frac{1}{(K-2)\nu},
\end{align*}
where (a) and (b) are due to $\lambda_j/(\theta_j+2\lambda_j+\lambda_i)$ is increasing in $\lambda_j > 0$. Following the idea of the first case, we have
$$
\sum_{j\in X_i \setminus \{i\}} \lambda_j\frac{\theta_j}{\theta_j+2\lambda_j+\lambda_i}-\theta_i \leq \theta_N \sum_{j\in X_i \setminus \{i\}} \frac{\lambda_j}{\theta_j+2\lambda_j+\lambda_i}-\theta_1 < \theta_1 \left[(K-2)\nu\frac{1}{(K-2)\nu}-1\right] \leq 0,
$$
then based on \Cref{lemma:n-concave-2}, we have $i \in X_i$, that is $\theta_i x_{i,i}^*=\lambda_i y_i^*,\ \forall i\in[N]$. 

Same as the first case, we then show that $y_i^*>0$ for all $i\in [N]$. Assume by contradiction that there exists $i \in [N]$, $y_i^*=0$ and $\sum_{j \in [N], j\neq i} x_{j,i}^* = \lambda_i$. For demand type $j$, 
$$
x_{j,i}^* \leq 
\lambda_i\frac{\lambda_j}{\theta_j+2\lambda_j+\lambda_i}.
$$
However, there are at most $K-2$ demand types $j \neq i$ such that $x_{j,i}^*>0$, then we have
$$
\sum_{j \in [N], j\neq i} x_{j,i}^* < \lambda_i \left[\sum_{j \in [N], j\neq i}  \frac{\lambda_j}{\theta_j+2\lambda_j+\lambda_i}\right] < \lambda_i (K-2) \frac{1}{(K-2)\nu} < \lambda_i,
$$which contradicts with constraint \eqref{equ:cn-flow}. Therefore, we must have $y_i^*>0$ for all $i\in [N]$. Based on \Cref{lemma:n-weak-concave} and \Cref{prop:n-weak-concave-nonzero}, $c(\boldsymbol{\lambda})$ is weakly concave on $\Lambda$.
 
For the case where $1/(K-2) > 1/[(K-2)\nu-2]$, we have
$$
\overline{\lambda}_i < \frac{\theta_i}{K-2}\left(1+\frac{1}{\nu}\frac{e_{(K)}}{1-2e_{(K)}}\right),\ \forall i \in [N].
$$
Again assume by contradiction that there exists $i \in [N]$, $y_i^*=0$ and $\sum_{j \in [N], j\neq i} x_{j,i}^* = \lambda_i$. For demand type $j$, we have $y_j^* \geq \theta_j x_{j,i}^*/\lambda_i$ and $x_{j,i}^*+y_j^*\leq\lambda_j$, leading to
$$
x_{j,i}^* \leq 
\lambda_i\frac{\lambda_j}{\theta_j+\lambda_i}.
$$
Then
\begin{align*}
\sum_{j \in [N], j\neq i} x_{j,i}^* & \leq \lambda_i \sum_{j \in [N], j\neq i} \frac{\lambda_j}{\theta_j+\lambda_i} < \lambda_i \sum_{j \in [N], j\neq i} \frac{\overline{\lambda}_j}{\theta_j+\underline{\lambda}_i}  \\
& < \lambda_i \sum_{j \in [N], j\neq i} \frac{\frac{\theta_j}{K-2}\left(1+\frac{1}{\nu}\frac{e_{(K)}}{1-2e_{(K)}}\right)}{\theta_j+\theta_i \frac{e_{(K)}}{1-2e_{(K)}}}\\
& = \lambda_i \sum_{j \in [N], j\neq i} \frac{\frac{1}{K-2}\left(1+\frac{1}{\nu}\frac{e_{(K)}}{1-2e_{(K)}}\right)}{1+\frac{\theta_i}{\theta_j} \frac{e_{(K)}}{1-2e_{(K)}}} \\
& \leq \lambda_i \sum_{j \in [N], j\neq i} \frac{\frac{1}{K-2}\left(1+\frac{1}{\nu}\frac{e_{(K)}}{1-2e_{(K)}}\right)}{1+\frac{1}{\nu} \frac{e_{(K)}}{1-2e_{(K)}}} \\
& \stackrel{(a)}{\leq} \lambda_i,
\end{align*}
where (a) is due to the fact that there are at most $K-2$ demand types $j \neq i$ such that $x_{j,i}^*>0$. This contradicts with $\sum_{j \in [N], j\neq i} x_{j,i}^* = \lambda_i$. Therefore, $y_i^*>0$, $i\in [N]$, and $c(\boldsymbol{\lambda})$ is weakly concave on $\Lambda$.

\hfill $\square$
\end{proof}

The specific results presented in the main text are derived from \Cref{thm:n-weak-concave-general} as follows. Regarding the general heterogeneous patience setting: \Cref{thm:n-weak-concave} corresponds to the first condition with parameters $K=3$ and $\nu \leq 2$, while \Cref{prop:n-weak-concave-both-bound} corresponds to the second condition. Furthermore, by setting $K=N+1$ in the second condition, we obtain \Cref{prop:n-weak-concave-upper-bound}. Finally, the results for the homogeneous case follow immediately by imposing $\nu=1$: \Cref{thm:n-weak-concave-same-theta} and \Cref{prop:n-weak-concave-same-theta-both-bound} are recovered from the general forms, and \Cref{prop:n-weak-concave-same-theta-upper-bound} is obtained from the second condition with $K=N+1$.

\medskip

Now we consider \Cref{prop:n-concave-same-theta} and \Cref{prop:n-concave}.

\begin{proof} {Proof of \Cref{prop:n-concave-same-theta}.}

Consider the case where $\theta_1=\theta_2=\cdots=\theta_N=\theta$ and
$
\underline{\lambda}_i > \theta e_{(3)}/(1-2e_{(3)}),\ \forall i \in [N]. 
$ Similar to the proof of \Cref{thm:n-weak-concave-general}, for any $i,j \in [N]\ (i\neq j)$, we have
$$
c_{(i)} \frac{\theta+\lambda_i}{\theta+2\lambda_i} + c_{(j)} \frac{\theta+\lambda_j}{\theta+2\lambda_j} - c_{(i,j)}  < \left(c_{(i)}+c_{(j)}\right) \left(e_{i,j}- e_{(3)}\right).
$$

Denote $$n(i) = {\arg\max}_{j\in[N],j\neq i} e_{i,j},$$ that is, $n(i)$ is the ``neighbor'' of demand type $i$ with the highest matching efficiency (excluding demand type $i$ itself). Then based on the definition of $e_{(3)}$, $e_{i,j} > e_{(3)} \geq \max\{e_{i,(3)},e_{j,(3)}\}\ (i\neq j)$ is possible only if $j=n(i)$ and $i=n(j)$. Then according to \Cref{lemma:n-concave-1}, for $i,j\in[N], i\neq j$, $x_{i,j}^*>0$ or $x_{j,i}^*>0$ is possible only if $j=n(i)$ and $i=n(j)$. and for all other pairs $(i,j)$, we must have $x_{i,j}^*=x_{j,i}^*=0$. This result allows us to reduce program \eqref{equ:cn} into the following form:
\begin{align*} 
c(\boldsymbol{\lambda}) = \min_{\boldsymbol{x},\boldsymbol{y}}\quad & \sum_{i,j\in[N]: j=n(i), i=n(j)} c_{(i,j)} x_{i,j} + \sum_{i\in[N]} c_{(i)} \left(y_i+x_{i,i}\right),\\
\text{s.t.} \quad & \sum_{j\in[N]: j=n(i), i=n(j)} \left(x_{i,j}+x_{j,i}\right)+x_{i,i} + y_i = \lambda_i, & \forall i \in [N],\\
& \eqref{equ:cn-bound}, \eqref{equ:cn-nonneg-x}, \eqref{equ:cn-nonneg-y}. \nonumber
\end{align*}

This reduced formulation is further decomposable, since all $N$ types of demands can be formed into groups of at most 2, and the matching only happens within the group. Specifically, when a pair $(i,j)$ satisfies $j=n(i)$ and $i=n(j)$, they might be matched together and
thus form a group of 2. The remaining demand types satisfy $i\neq n(n(i))$ and will not be matched with any other demand types. Let $c_{i}(\lambda_i)$ denote the optimal value function assuming only type-$i$ demand exists, then $c_{i}(\lambda_i)$ can be given by \Cref{lemma:n1-solution} and is concave. Let $c_{i,j}(\lambda_i, \lambda_j)$ denote the optimal value function assuming only type-$i$ and type-$j$ demands exist, then  $c_{i,j}(\lambda_i, \lambda_j)$ is concave and tight according to \Cref{prop:n2-concave-same-theta} and \Cref{prop:n2-tight}. By decomposition, we have
$$
c(\boldsymbol{\lambda}) = \sum_{i,j\in[N]: i<j,j=n(i), i=n(j)} c_{i,j}(\lambda_i, \lambda_j) + \sum_{i\in[N]: i\neq n(n(i))} c_{i}(\lambda_i),
$$
and is thus concave and tight. \hfill $\square$
\end{proof}

\begin{proof} {Proof of \Cref{prop:n-concave}.}

For this case, we have $
\underline{\lambda}_i > \theta_i e_{(2)}/(1-2e_{(2)}),\ \forall i \in [N].
$ Similar to the proof of \Cref{thm:n-weak-concave-general}, for any $i,j \in [N]\ (i\neq j)$, we have
$$
c_{(i)} \frac{\theta_i+\lambda_i}{\theta_i+2\lambda_i} + c_{(j)} \frac{\theta_j+\lambda_j}{\theta_j+2\lambda_j} - c_{(i,j)}  < \left(c_{(i)}+c_{(j)}\right) \left(e_{i,j}- e_{(2)}\right).
$$
This means that, according to \Cref{lemma:n-concave-1}, for any demand type $i\in[N]$, it can only be matched with itself. Thus the program \eqref{equ:cn} can be reduced into:
$$
c(\boldsymbol{\lambda}) = \sum_{i\in[N]} c_{i}(\lambda_i),
$$
which is concave and tight according to \Cref{prop:n1-concave} and \Cref{prop:n2-tight}. \hfill $\square$
\end{proof}

\subsection{Proof of \Cref{prop:n-concave-agg}}

We first prove that when $
\sum_{j\in\mathcal{S}(i)} \underline{\lambda}_j > \theta {\bar{e}}/(1-2\bar{e}),\ \forall i \in [N]$, then $x_{i,j}^* = 0$ holds if $j \notin \mathcal{S}(i)$. Assume by contradiction that $x_{i,j}^* > 0$, we can obtain a new solution by letting
    \begin{align*}
        \hat{x}_{i,j} &= x_{i,j}^* - \varepsilon, \\
        \hat{y}_{i} &= y_{i}^* + \varepsilon \frac{\theta}{\theta+2\sum_{k\in\mathcal{S}(i)}\lambda_k}, \\
        \hat{x}_{i,i} &= x_{i,i}^* + \varepsilon \frac{\sum_{k\in\mathcal{S}(i)}\lambda_k}{\theta+2\sum_{k\in\mathcal{S}(i)}\lambda_k}, \\
        \hat{y}_{j} &= y_{j}^* + \varepsilon \frac{\theta}{\theta+2\sum_{k\in\mathcal{S}(j)}\lambda_k}, \\
        \hat{x}_{j,j} &= x_{j,j}^* + \varepsilon \frac{\sum_{k\in\mathcal{S}(j)}\lambda_k}{\theta+2\sum_{k\in\mathcal{S}(j)}\lambda_k},
    \end{align*}
    and keeping the remaining variables unchanged, where $\varepsilon > 0$ is a sufficiently small positive value. Similar to the way in \Cref{lemma:n2-cross-binding}, we can verify this is still a feasible solution to \eqref{equ:cn-agg}. Compared to the original solution, the difference in cost under this new feasible solution is given by
\begin{align*}
    & \left[c_{(i)}(\hat{x}_{i,i}+\hat{y}_i) + c_{(j)}(\hat{x}_{j,j}+\hat{y}_j) + c_{(i,j)} \hat{x}_{i,j} \right] - \left[ c_{(i)}(x_{i,i}^*+y_i^*) + c_{(j)}(x_{j,j}^*+y_j^*) + c_{(i,j)} x_{i,j}^* \right]\\
    = & \varepsilon\left[ c_{(i)}\frac{\theta+\sum_{k\in\mathcal{S}(i)}\lambda_k}{\theta+2\sum_{k\in\mathcal{S}(i)}\lambda_k} + c_{(j)}\frac{\theta+\sum_{k\in\mathcal{S}(j)}\lambda_k}{\theta+2\sum_{k\in\mathcal{S}(j)}\lambda_k} - c_{(i,j)}\right] \\
    < & \varepsilon\left[ (c_{(i)} + c_{(j)}) (1-\bar{e})- c_{(i,j)}\right] = \varepsilon (c_{(i)} + c_{(j)}) \left[ e_{i,j}-\bar{e}\right]\\
    \stackrel{(a)}{\leq} & 0,
\end{align*}
where (a) is due to the condition that if $j\notin \mathcal{S}(i)$, then $e_{i,j} \leq \bar{e}$. Therefore, if $x_{i,j}^*>0$, by decreasing the value of $x_{i,j}$, we can always obtain another feasible solution with lower cost, meaning that $x_{i,j}^*=0$ must hold.

We can also prove that there exists an optimal solution such that $x_{i,j}^*=0$ holds if $j\in\mathcal{S}(i)$ ($i\neq j$). Again, assume by contradiction that $x_{i,j}^*>0$, and we obtain a new solution by letting
\begin{align*}
        \hat{x}_{i,j} &= x_{i,j}^* - \varepsilon, \\
        \hat{y}_{i} &= y_{i}^* - \varepsilon \frac{\theta}{\theta+2\sum_{k\in\mathcal{S}(i)}\lambda_k}, \\
        \hat{x}_{i,i} &= x_{i,i}^* + \varepsilon \frac{\theta+\sum_{k\in\mathcal{S}(i)}\lambda_k}{\theta+2\sum_{k\in\mathcal{S}(i)}\lambda_k}, \\
        \hat{y}_{j} &= y_{j}^* + \varepsilon \frac{\theta}{\theta+2\sum_{k\in\mathcal{S}(j)}\lambda_k}, \\
        \hat{x}_{j,j} &= x_{j,j}^* + \varepsilon \frac{\sum_{k\in\mathcal{S}(j)}\lambda_k}{\theta+2\sum_{k\in\mathcal{S}(j)}\lambda_k},
    \end{align*}
and keeping the remaining variables unchanged, where $\varepsilon > 0$ is a sufficiently small positive value. The feasibility of the new solution can also be easily verified. Compared to the original solution, the difference in cost under this new feasible solution is given by
\begin{align*}
    & \left[c_{(i)}(\hat{x}_{i,i}+\hat{y}_i) + c_{(j)}(\hat{x}_{j,j}+\hat{y}_j) + c_{(i,j)} \hat{x}_{i,j} \right] - \left[ c_{(i)}(x_{i,i}^*+y_i^*) + c_{(j)}(x_{j,j}^*+y_j^*) + c_{(i,j)} x_{i,j}^* \right]\\
    = & \varepsilon\left[ c_{(i)}\frac{\sum_{k\in\mathcal{S}(i)}\lambda_k}{\theta+2\sum_{k\in\mathcal{S}(i)}\lambda_k} + c_{(j)}\frac{\theta+\sum_{k\in\mathcal{S}(j)}\lambda_k}{\theta+2\sum_{k\in\mathcal{S}(j)}\lambda_k} - c_{(i,j)}\right] \\
    < & \varepsilon c_{(i,j)} \left[ \frac{\sum_{k\in\mathcal{S}(i)}\lambda_k}{\theta+2\sum_{k\in\mathcal{S}(i)}\lambda_k} + \frac{\theta+\sum_{k\in\mathcal{S}(j)}\lambda_k}{\theta+2\sum_{k\in\mathcal{S}(j)}\lambda_k} - 1\right] \\
    \stackrel{(a)}{=} &  0,
\end{align*}
where (a) is due to $\mathcal{S}(i) = \mathcal{S}(j)$. Therefore, we can always obtain another feasible solution with a lower cost, meaning that $x_{i,j}^*=0$ always holds.

These two results lead to the conclusion that all optimal solutions satisfy $x_{i,j}^*=0$ for any $i\neq j$. This allows us to simplify $c(\boldsymbol{\lambda}; \mathcal{S})$ to

\begin{align*} 
c(\boldsymbol{\lambda}; \mathcal{S}) = \min_{\boldsymbol{x},\boldsymbol{y}} \quad& \sum_{i\in[N]} c_{(i)} (y_i + x_{i,i})\\
\text{s.t.} \quad 
& \theta x_{i,i} \leq y_i \sum_{j \in S(i)} \lambda_j , & \forall i \in [N], \\
& 2 x_{i,i} + y_i = \lambda_i, & \forall i \in [N], \\
& x_{i,i}, y_i \geq 0, & \forall i \in [N],
\end{align*}
of which the optimal solution can be trivially given by 
\begin{align*}
    x_{i,i}^* & = \lambda_i \frac{\sum_{j \in S(i)} \lambda_j}{\theta + 2\sum_{j \in S(i)} \lambda_j}, & \forall i \in [N], \\
    y_i^* & = \lambda_i \frac{\theta}{\theta + 2\sum_{j \in S(i)} \lambda_j}, & \forall i \in [N].
\end{align*}

Therefore,
$$
c(\boldsymbol{\lambda}; \mathcal{S}) = \sum_{i\in[N]} c_{(i)} \lambda_i \frac{\theta + \sum_{j \in S(i)} \lambda_j}{\theta + 2\sum_{j \in S(i)} \lambda_j}
$$ is weakly concave, given its smoothness as well as the boundedness of its Hessian matrix (see \Cref{lemma:n-weak-concave-3}).  \hfill $\square$

\subsection{Proof of \Cref{prop:n2-concave}}

According to \Cref{lemma:n2-solution}, $y_1^*>0$ is guaranteed, and $y_2^*>0$ if and only if $\Delta_2(\lambda_1,\lambda_2) < 0$ or $\Delta_3(\lambda_1,\lambda_2) < 0$ holds for all $(\lambda_1,\lambda_2) \in \Lambda$. Define 
the set $\mathcal{D} := \{(\lambda_1,\lambda_2) |\Delta_2(\lambda_1,\lambda_2) \geq 0, \Delta_3(\lambda_1,\lambda_2) \geq 0, \lambda_1,\lambda_2>0 \}$. If $\Lambda \cap \mathcal{D} = \emptyset$, then $\boldsymbol{y}^*>\boldsymbol{0}$ always hold for $\boldsymbol{\lambda}\in\Lambda$. According to \Cref{lemma:n-weak-concave} and \Cref{prop:n-weak-concave-nonzero}, $c(\boldsymbol{\lambda})$ is then weakly concave on $\Lambda$.

Firstly, it is trivial to verify that the last condition leads to the weak concavity. When $\underline{\lambda}_2 > \overline{\lambda}_1 - \theta_1$, $\Delta_3(\lambda_1,\lambda_2)<0$ for all $\boldsymbol{\lambda}\in\Lambda$, thus $\Lambda \cap \mathcal{D} = \emptyset$. Now we prove that the first two conditions also lead to weak concavity. For the set $\mathcal{D}$, note that 
\begin{align*}
    &\Delta_2(\lambda_1,\lambda_2) \geq 0 \\
    \Leftrightarrow\ & \frac{c_{(1)}}{2}\left[1 - \frac{\theta_1(\theta_2+\lambda_1+2\lambda_2)}{\lambda_2(\theta_2+2\lambda_2)}\right] +c_{(2)}\frac{\theta_2+\lambda_2}{\theta_2+2\lambda_2} - c_{(1,2)}  \geq 0 \\
    \Leftrightarrow\ & c_{(1)} \left[\lambda_2(\theta_2+2\lambda_2)- \theta_1(\theta_2+\lambda_1+2\lambda_2)\right] + 2c_{(2)}\lambda_2(\theta_2+\lambda_2) - 2c_{(1,2)}\lambda_2(\theta_2+2\lambda_2) \geq 0\\
    \Leftrightarrow\ & \lambda_1 \leq -2\frac{2c_{(1,2)}-c_{(1)}-c_{(2)}}{c_{(1)}\theta_1} \lambda_2^2 + \left( \frac{\tau_1}{c_{(1)}\theta_1}+1\right) \lambda_2 - \theta_2,
\end{align*}
where $\tau_1 = c_{(1)}(\theta_2-3\theta_1)+2c_{(2)}\theta_2-2c_{(1,2)}\theta_2$, and $\Delta_3(\lambda_1,\lambda_2) \geq 0$ is equivalent to $ \lambda_1 \geq \lambda_2+\theta_1$. Thus
\begin{align*}
    \mathcal{D}
    = \left\{(\lambda_1,\lambda_2) |\lambda_2+\theta_1 \leq \lambda_1 \leq -2\frac{2c_{(1,2)}-c_{(1)}-c_{(2)}}{c_{(1)}\theta_1} \lambda_2^2 + \left( \frac{\tau_1}{c_{(1)}\theta_1}+1\right) \lambda_2 - \theta_2, \lambda_1,\lambda_2>0 \right\}.
\end{align*}

According to \Cref{assum:cost}, $2c_{(1,2)}-c_{(1)}-c_{(2)} \geq 0$. Consider the case where $2c_{(1,2)}-c_{(1)}-c_{(2)} = 0$, then $\tau_2 = \tau_1^2 \geq 0$. If $\tau_1 \leq 0$ in this case, $\mathcal{D}$ is equivalent to
$$\left\{(\lambda_1,\lambda_2) |\lambda_2+\theta_1 \leq \lambda_1 \leq \left( \frac{\tau_1}{c_{(1)}\theta_1}+1\right) \lambda_2 - \theta_2, \lambda_1,\lambda_2>0 \right\},$$ which is an empty set because $\left( \tau_1/(c_{(1)}\theta_1)+1\right) \lambda_2 - \theta_2 \leq \lambda_2 - \theta_2 < \lambda_2 + \theta_1$ for any $\lambda_1, \lambda_2 > 0$. Thus $\Lambda \cap \mathcal{D} = \emptyset$ when $2c_{(1,2)}-c_{(1)}-c_{(2)} = 0$ and $\tau_1\leq0$.

For the remainder of the proof, we assume $2c_{(1,2)}-c_{(1)}-c_{(2)} > 0$. In this case, $\lambda_1$ in $\mathcal{D}$ is larger than a linear function of $\lambda_2$ and smaller than a quadratic function with a negative leading coefficient. Then a sufficient condition of $\Lambda \cap \mathcal{D} = \emptyset$ is for any $\lambda_2\in[\underline{\lambda}_2, \overline{\lambda}_{2}]$,
\begin{align*}
    & \lambda_2+\theta_1 > -2\frac{2c_{(1,2)}-c_{(1)}-c_{(2)}}{c_{(1)}\theta_1} \lambda_2^2 + \left( \frac{\tau_1}{c_{(1)}\theta_1}+1\right) \lambda_2 - \theta_2\\
    \Leftrightarrow & - \lambda_2^2 + \frac{\tau_1}{2(2c_{(1,2)}-c_{(1)}-c_{(2)})} \lambda_2 - \frac{c_{(1)}\theta_1(\theta_1+\theta_2)}{2(2c_{(1,2)}-c_{(1)}-c_{(2)})}<0\\
    \Leftrightarrow & -\left[\lambda_2 - \frac{\tau_1}{4(2c_{(1,2)}-c_{(1)}-c_{(2)})}\right]^2+\frac{\tau_2}{16(2c_{(1,2)}-c_{(1)}-c_{(2)})^2} < 0.
\end{align*}

If $\tau_2 < 0$, the above inequality naturally holds for any $\lambda_2 \in R$. If $\tau_2 \geq 0$, the above inequality leads to
$$
\lambda_2 > \frac{\tau_1+\sqrt{\tau_2}}{4(2c_{(1,2)}-c_{(1)}-c_{(2)})},\ \text{or}\ \lambda_2 < \frac{\tau_1-\sqrt{\tau_2}}{4(2c_{(1,2)}-c_{(1)}-c_{(2)})}.
$$
When $\tau_1 \leq 0$, 
$$
\tau_1+\sqrt{\tau_2} = \tau_1 + \sqrt{\tau_1^2-8c_{(1)}(2c_{(1,2)}-c_{(1)}-c_{(2)})\theta_1(\theta_1+\theta_2)} \leq 0,
$$
meaning that the inequality always holds when $\lambda_2 > 0$. Otherwise, when $\tau_1 >0$ and  $\tau_2 \geq 0$, a sufficient condition to make the inequality hold is when
$$
\underline{\lambda}_2 > \frac{\tau_1+\sqrt{\tau_2}}{4(2c_{(1,2)}-c_{(1)}-c_{(2)})}.
$$
\hfill $\square$

\subsection{Proof of \Cref{corl:n2-concave}}

For the first case when $\theta_2 < 3\theta_1$, we have
$$
\tau_1 = c_{(1)}(\theta_2-3\theta_1)+2c_{(2)}\theta_2-2c_{(1,2)}\theta_2 < 2c_{(2)}\theta_2-2c_{(1,2)}\theta_2 \leq 0.
$$
Then, according to the first case of \Cref{prop:n2-concave}, $c(\boldsymbol{\lambda})$ is weakly concave.

In the second case, according to the definition of $e_{1,2}$, we have
\begin{align*}
\underline{\lambda}_2 > \frac{\theta_2}{4}\frac{1}{1-2e_{1,2}} &= \frac{\theta_2}{4}\frac{c_{(1)}+c_{(2)}}{2c_{(1,2)}-c_{(1)}-c_{(2)}} 
\stackrel{(a)}{\geq }\frac{2\theta_2\left(c_{(1)}+2c_{(2)}-2c_{(1,2)}\right)}{4\left(2c_{(1,2)}-c_{(1)}-c_{(2)}\right)}  > \frac{2\tau_1}{4\left(2c_{(1,2)}-c_{(1)}-c_{(2)}\right)}.
\end{align*}
For (a), $c_{(1)}+c_{(2)} \geq 2\left(c_{(1)}+2c_{(2)}-2c_{(1,2)}\right)$ because $c_{(1,2)} \geq c_{(1)}$ and $c_{(1,2)} \geq c_{(2)}$. If $\tau_1 \leq 0$ or $\tau_2 < 0$, the result holds according to the first case of \Cref{prop:n2-concave}. If $\tau_1 > 0$ and $\tau_2 \geq 0$, note that
$$
\sqrt{\tau_2} = \sqrt{\tau_1^2-8c_{(1)}(2c_{(1,2)}-c_{(1)}-c_{(2)})\theta_1(\theta_1+\theta_2)} < \sqrt{\tau_1^2} = \tau_1,
$$
and thus
$$
\underline{\lambda}_2 > \frac{2\tau_1}{4\left(2c_{(1,2)}-c_{(1)}-c_{(2)}\right)}>\frac{\tau_1+\sqrt{\tau_2}}{4\left(2c_{(1,2)}-c_{(1)}-c_{(2)}\right)}.
$$
According to the second case of \Cref{prop:n2-concave}, $c(\boldsymbol{\lambda})$ is weakly concave. 

For the third case, since $\overline{\lambda}_1 \leq \theta_1$, $\underline{\lambda}_2 > 0 \geq \overline{\lambda}_1 - \theta_1$ always holds and according to the third case of \Cref{prop:n2-concave}, $c(\boldsymbol{\lambda})$ is weakly concave. 
\hfill $\square$

    \section{Additional Computational Results}
\label{appx:additional_exp}
\subsection{Heterogeneous Patience Levels}\label{appendix:different_theta}

\begin{table}[ht] \centering \scriptsize
\caption{Numerical results of solving \eqref{equ:gn} when $\theta_i$ varies across $i$.} \label{tab:results-diff-theta}
\resizebox{\textwidth}{!}{
\begin{tabular}{lcllllllllll}
\toprule
\multicolumn{3}{c}{$c$ (\$/mile)} & \multicolumn{3}{c}{$0.7$} & \multicolumn{3}{c}{$0.9$} & \multicolumn{3}{c}{$1.1$} \\ \cmidrule(lr){4-6} \cmidrule(lr){7-9} \cmidrule(lr){10-12} 
\multicolumn{3}{c}{$(\underline{\theta}, \overline{\theta})$ (min$^{-1}$)} & $(1/5, 1/3)$ & $(1/3, 1)$ & $(1, 2)$ & $(1/5, 1/3)$ & $(1/3, 1)$ & $(1, 2)$ & $(1/5, 1/3)$ & $(1/3, 1)$ &  $(1, 2)$ \\ \hline
\multicolumn{1}{c}{\multirow{12}{*}{$N=100$}} & \multirow{4}{*}{\begin{tabular}[c]{@{}c@{}}Running\\ Time (sec)\end{tabular}} & PG ($\delta_{\text{PG}}=100$) & 10.7 & 10.8 & 10.6 & 10.1 & 10.0 & 10.6 & 52.2 & 32.2 & 46.9 \\
\multicolumn{1}{c}{} &  & PG ($\delta_{\text{PG}}=10$) & 9.4 & 9.6 & 10.8 & 8.1 & 8.7 & 8.7 & 22.3 & 24.3 & 14.3 \\
\multicolumn{1}{c}{} &  & PG ($\delta_{\text{PG}}=1$) & 18.4 & 8.6 & 22.9 & 18.9 & 19.8 & 12.0 & 18.7 & 29.9 & 18.4 \\
\multicolumn{1}{c}{} &  & MM & 1.7 & 1.7 & 1.7 & 2.0 & 2.1 & 2.3 & 4.3 & 5.0 & 5.3 \\ \cline{2-12} 
\multicolumn{1}{c}{} & \multirow{4}{*}{\begin{tabular}[c]{@{}c@{}}Num.\\ Iterations\end{tabular}} & PG ($\delta_{\text{PG}}=100$) & 30.7 & 32.0 & 31.7 & 29.0 & 31.0 & 31.7 & 158.7 & 101.0 & 133.0 \\
\multicolumn{1}{c}{} &  & PG ($\delta_{\text{PG}}=10$) & 27.0 & 28.0 & 31.3 & 23.3 & 25.3 & 25.7 & 66.7 & 74.0 & 42.0 \\
\multicolumn{1}{c}{} &  & PG ($\delta_{\text{PG}}=1$) & 54.0 & 25.3 & 69.3 & 56.7 & 59.3 & 35.7 & 56.3 & 92.3 & 52.3 \\
\multicolumn{1}{c}{} &  & MM & 4.0 & 4.0 & 4.0 & 5.0 & 5.0 & 6.0 & 12.0 & 14.3 & 14.0 \\ \cline{2-12} 
\multicolumn{1}{c}{} & \multirow{4}{*}{\begin{tabular}[c]{@{}c@{}}Objective\\ Function\end{tabular}} & PG ($\delta_{\text{PG}}=100$) & 62.28 & 52.29 & 42.08 & 37.18 & 26.76 & 17.33 & 11.76 & 6.20 & 2.60 \\
\multicolumn{1}{c}{} &  & PG ($\delta_{\text{PG}}=10$) & 62.28 & 52.29 & 42.08 & 37.20 & 26.76 & 17.33 & 11.23 & 6.64 & 2.64 \\
\multicolumn{1}{c}{} &  & PG ($\delta_{\text{PG}}=1$) & 62.19 & 51.97 & 41.96 & 37.12 & 26.67 & 17.17 & 15.40 & 8.29 & 2.96 \\
\multicolumn{1}{c}{} &  & MM & 62.28 & 52.29 & 42.08 & 37.19 & 26.76 & 17.33 & 17.23 & 8.78 & 3.17 \\ \hline
\multirow{12}{*}{$N=200$} & \multirow{4}{*}{\begin{tabular}[c]{@{}c@{}}Running\\ Time (sec)\end{tabular}} & PG ($\delta_{\text{PG}}=100$) & 38.3 & 44.6 & 51.1 & 42.6 & 50.8 & 61.4 & 191.5 & 64.0 & 106.0 \\
 &  & PG ($\delta_{\text{PG}}=10$) & 30.5 & 34.7 & 43.4 & 33.7 & 39.9 & 50.7 & 260.5 & 364.8 & 104.6 \\
 &  & PG ($\delta_{\text{PG}}=1$) & 27.5 & 33.5 & 46.3 & 28.3 & 30.8 & 45.2 & 146.1 & 81.3 & 51.4 \\
 &  & MM & 5.7 & 6.4 & 8.5 & 7.9 & 8.9 & 11.1 & 20.4 & 25.0 & 32.7 \\ \cline{2-12} 
 & \multirow{4}{*}{\begin{tabular}[c]{@{}c@{}}Num.\\ Iterations\end{tabular}} & PG ($\delta_{\text{PG}}=100$) & 33.0 & 33.0 & 31.0 & 32.7 & 34.0 & 34.0 & 131.3 & 38.3 & 49.7 \\
 &  & PG ($\delta_{\text{PG}}=10$) & 25.3 & 25.7 & 25.7 & 26.3 & 27.0 & 27.0 & 178.3 & 195.3 & 47.3 \\
 &  & PG ($\delta_{\text{PG}}=1$) & 23.3 & 26.0 & 27.3 & 23.0 & 20.7 & 23.3 & 102.7 & 43.0 & 22.7 \\
 &  & MM & 4.0 & 4.0 & 4.0 & 5.7 & 5.3 & 5.0 & 14.7 & 13.0 & 13.3 \\ \cline{2-12} 
 & \multirow{4}{*}{\begin{tabular}[c]{@{}c@{}}Objective\\ Function\end{tabular}} & PG ($\delta_{\text{PG}}=100$) & 57.76 & 48.83 & 40.35 & 31.96 & 23.49 & 16.25 & 10.70 & 6.76 & 2.79 \\
 &  & PG ($\delta_{\text{PG}}=10$) & 57.76 & 48.83 & 40.64 & 31.96 & 23.49 & 16.25 & 11.31 & 7.12 & 2.82 \\
 &  & PG ($\delta_{\text{PG}}=1$) & 57.63 & 48.70 & 40.51 & 31.81 & 23.25 & 16.09 & 12.74 & 7.09 & 2.65 \\
 &  & MM & 57.76 & 48.83 & 40.64 & 31.96 & 23.49 & 16.25 & 13.16 & 7.26 & 2.89 \\ \hline
\multirow{12}{*}{$N=1000$} & \multirow{4}{*}{\begin{tabular}[c]{@{}c@{}}Running\\ Time (sec)\end{tabular}} & PG ($\delta_{\text{PG}}=100$) & 1200 & 1200 & 1200 & 1200 & 1200 & 1200 & 1200 & 1200 & 1200 \\
 &  & PG ($\delta_{\text{PG}}=10$) & 1200 & 1200 & 1200 & 1200 & 1200 & 1200 & 1200 & 1200 & 1200 \\
 &  & PG ($\delta_{\text{PG}}=1$) & 1200 & 1200 & 1200 & 1200 & 1200 & 1200 & 1200 & 1200 & 1200 \\
 &  & MM & 408.5 & 436.3 & 386.1 & 457.4 & 399.1 & 435.2 & 710.0 & 611.1 & 796.1 \\ \cline{2-12} 
 & \multirow{4}{*}{\begin{tabular}[c]{@{}c@{}}Num.\\ Iterations\end{tabular}} & PG ($\delta_{\text{PG}}=100$) & 14.0 & 14.7 & 13.7 & 15.0 & 14.3 & 14.7 & 14.3 & 14.3 & 14.3 \\
 &  & PG ($\delta_{\text{PG}}=10$) & 15.0 & 14.0 & 14.3 & 15.0 & 14.3 & 14.0 & 14.0 & 14.3 & 14.0 \\
 &  & PG ($\delta_{\text{PG}}=1$) & 14.0 & 14.0 & 13.3 & 14.0 & 14.3 & 14.7 & 14.7 & 14.3 & 16.7 \\
 &  & MM & 4.0 & 4.3 & 4.0 & 4.3 & 4.3 & 4.7 & 8.0 & 7.7 & 12.0 \\ \cline{2-12} 
 & \multirow{4}{*}{\begin{tabular}[c]{@{}c@{}}Objective\\ Function\end{tabular}} & PG ($\delta_{\text{PG}}=100$) & -51.73 & -145.44 & -123.78 & -255.35 & -132.14 & -247.79 & -164.69 & -161.08 & -160.94 \\
 &  & PG ($\delta_{\text{PG}}=10$) & -101.18 & -23.62 & -59.97 & -168.13 & -89.41 & -39.86 & -149.86 & -99.09 & -60.82 \\
 &  & PG ($\delta_{\text{PG}}=1$) & 30.25 & 25.26 & -33.20 & -15.45 & -2.21 & -12.63 & -4.87 & -0.18 & 1.88 \\
 &  & MM & 54.97 & 47.99 & 40.82 & 29.48 & 22.74 & 16.19 & 11.90 & 6.70 & 2.50 \\
 \bottomrule
\end{tabular}
}
\end{table}

\Cref{tab:results-diff-theta} presents the experimental results for the setting where $\theta_i$ varies across demand types. For type-$i$ demands, $\theta_i$ is randomly drawn from a uniform distribution $U[\underline{\theta}, \overline{\theta}]$, with three parameter choices: $(\underline{\theta}, \overline{\theta}) \in \{(1/5, 1/3), (1/3, 1), (1, 2)\}$. The results exhibit a similar trend to that observed in \Cref{tab:results-same-theta}. When $N \leq 200$, compared to three PG variants with $\delta_{\text{PG}}\in\{100,10,1\}$, the MM algorithm reduces running time by 81.7\%, 83.5\% and 84.7\%; decreases the number of iterations by 85.6\%, 83.0\% and 88.5\%; and improves the objective value by 2.27\%, 1.97\% and 0.94\%, respectively. When $N=1000$, only MM converges within 20 minutes. Besides, the performance of PG is still sensitive to the step size: $\delta_{\text{PG}}=100$ and $\delta_{\text{PG}}=10$ perform better when $N \leq 200$ and $c\leq 0.9$, whereas $\delta_{\text{PG}}=1$ is preferable when $c = 1.1$ or $N=1000$.

\subsection{Exponential Demand Model} \label{appendix:exponential_demand}

In this section, we present additional experimental results under another demand model. For rider type $i\in[N]$, we assume that the demand $\lambda_i$ is exponential with respect to the per-mile price $p_i/\ell_{(i)}$:
$$
\lambda_i = \overline{\lambda}_i \exp\left(-\frac{p_i}{\ell_{(i)}}\right),\quad \forall i\in[N],
$$
and inversely,
$$
p_i(\lambda_i) = \ell_{(i)} \left(\ln \overline{\lambda}_i - \ln \lambda_i\right),\quad \forall i\in[N].
$$

Correspondingly, for the PG method, the gradient of $g(\boldsymbol{\lambda})$ now becomes
$$
    \nabla g(\boldsymbol{\lambda}) = \left[ \ell_{(i)} \left(\ln \overline{\lambda}_i - \ln \lambda_i-1\right) - \gamma_i^* - \sum_{j \in [N]} y_j^* \eta_{j,i}^* \right]_{i \in [N]}.
$$

For the MM method, we are not able to obtain a closed-form solution of $\boldsymbol{\lambda}^{(t+1)} \gets \arg\max_{\boldsymbol{\lambda} \in \Lambda} Q(\boldsymbol{\lambda} \mid \boldsymbol{\lambda}^{(t)})$ under the exponential demand model. Nevertheless, by leveraging the fact that $Q(\boldsymbol{\lambda} \mid \boldsymbol{\lambda}^{(t)})$ is concave, we can still use efficient root finding methods (e.g., bisection method) to solve
$$
\nabla Q(\boldsymbol{\lambda} \mid \boldsymbol{\lambda}^{(t)}) = \left[ \ell_{(i)} \left(\ln \overline{\lambda}_i - \ln \lambda_i-1\right) - \rho \lambda_i- \gamma_i^* - \sum_{j \in [N]} y_j^* \eta_{j,i}^* + \rho\lambda_i^{(t)}\right]_{i \in [N]} = \boldsymbol{0},
$$
thus obtaining the unique solution of $\boldsymbol{\lambda}^{(t+1)}$ at each iteration.

\begin{table}[ht] \centering \tiny 
\caption{Numerical results of solving \eqref{equ:gn} when $\theta_i$ varies across $i$, under exponential demand model.} \label{tab:results-diff-theta-exponential-demand}
\resizebox{\textwidth}{!}{
\begin{tabular}{lcllllllllll}
\toprule
\multicolumn{3}{c}{$c$ (\$/mile)} & \multicolumn{3}{c}{$0.7$} & \multicolumn{3}{c}{$0.9$} & \multicolumn{3}{c}{$1.1$} \\ \cmidrule(lr){4-6} \cmidrule(lr){7-9} \cmidrule(lr){10-12} 
\multicolumn{3}{c}{$(\underline{\theta}, \overline{\theta})$ (min$^{-1}$)} & $(1/5, 1/3)$ & $(1/3, 1)$ & $(1, 2)$ & $(1/5, 1/3)$ & $(1/3, 1)$ & $(1, 2)$ & $(1/5, 1/3)$ & $(1/3, 1)$ &  $(1, 2)$ \\ \hline
\multicolumn{1}{c}{\multirow{12}{*}{$N=100$}} & \multirow{4}{*}{\begin{tabular}[c]{@{}c@{}}Running\\ Time (sec)\end{tabular}} & PG ($\delta_{\text{PG}}=100$) & 12.5 & 12.0 & 11.9 & 11.0 & 10.9 & 11.4 & 11.4 & 11.8 & 12.3 \\
\multicolumn{1}{c}{} &  & PG ($\delta_{\text{PG}}=10$) & 9.4 & 9.9 & 10.1 & 8.9 & 9.0 & 9.7 & 9.9 & 9.6 & 10.7 \\
\multicolumn{1}{c}{} &  & PG ($\delta_{\text{PG}}=1$) & 40.9 & 42.8 & 41.7 & 32.8 & 32.2 & 16.7 & 48.9 & 8.4 & 9.4 \\
\multicolumn{1}{c}{} &  & MM & 1.4 & 1.8 & 1.8 & 1.8 & 1.8 & 1.8 & 1.7 & 1.7 & 1.7 \\ \cline{2-12} 
\multicolumn{1}{c}{} & \multirow{4}{*}{\begin{tabular}[c]{@{}c@{}}Num.\\ Iterations\end{tabular}} & PG ($\delta_{\text{PG}}=100$) & 38.3 & 36.3 & 36.0 & 33.3 & 33.0 & 34.3 & 34.7 & 35.3 & 37.3 \\
\multicolumn{1}{c}{} &  & PG ($\delta_{\text{PG}}=10$) & 27.7 & 30.0 & 30.3 & 27.0 & 26.3 & 29.0 & 29.3 & 29.0 & 32.0 \\
\multicolumn{1}{c}{} &  & PG ($\delta_{\text{PG}}=1$) & 126.3 & 133.0 & 128.0 & 101.3 & 99.3 & 50.3 & 151.7 & 24.7 & 28.3 \\
\multicolumn{1}{c}{} &  & MM & 3.0 & 4.0 & 4.0 & 4.0 & 4.0 & 4.0 & 4.0 & 4.0 & 4.0 \\ \cline{2-12} 
\multicolumn{1}{c}{} & \multirow{4}{*}{\begin{tabular}[c]{@{}c@{}}Objective\\ Function\end{tabular}} & PG ($\delta_{\text{PG}}=100$) & 176.40 & 167.78 & 159.08 & 154.92 & 144.93 & 135.25 & 135.54 & 124.60 & 114.44 \\
\multicolumn{1}{c}{} &  & PG ($\delta_{\text{PG}}=10$) & 176.40 & 167.78 & 159.08 & 154.92 & 144.92 & 135.25 & 135.54 & 124.61 & 114.44 \\
\multicolumn{1}{c}{} &  & PG ($\delta_{\text{PG}}=1$) & 176.30 & 167.68 & 159.01 & 154.83 & 144.84 & 135.21 & 135.46 & 124.54 & 114.37 \\
\multicolumn{1}{c}{} &  & MM & 176.40 & 167.78 & 159.08 & 154.93 & 144.93 & 135.25 & 135.54 & 124.61 & 114.44 \\ \hline
\multirow{12}{*}{$N=200$} & \multirow{4}{*}{\begin{tabular}[c]{@{}c@{}}Running\\ Time (sec)\end{tabular}} & PG ($\delta_{\text{PG}}=100$) & 44.3 & 50.2 & 60.5 & 47.8 & 53.0 & 65.0 & 48.9 & 58.4 & 72.5 \\
 &  & PG ($\delta_{\text{PG}}=10$) & 35.9 & 39.7 & 52.4 & 37.0 & 42.2 & 53.7 & 40.0 & 47.2 & 65.9 \\
 &  & PG ($\delta_{\text{PG}}=1$) & 229.3 & 269.4 & 365.5 & 99.2 & 45.6 & 46.1 & 31.8 & 39.1 & 57.6 \\
 &  & MM & 5.4 & 6.9 & 7.2 & 5.8 & 7.3 & 9.0 & 6.0 & 7.1 & 9.2 \\ \cline{2-12} 
 & \multirow{4}{*}{\begin{tabular}[c]{@{}c@{}}Num.\\ Iterations\end{tabular}} & PG ($\delta_{\text{PG}}=100$) & 35.7 & 36.3 & 36.7 & 37.7 & 37.0 & 37.3 & 37.3 & 40.0 & 41.0 \\
 &  & PG ($\delta_{\text{PG}}=10$) & 29.7 & 29.3 & 31.3 & 29.7 & 30.0 & 31.0 & 31.3 & 32.7 & 37.0 \\
 &  & PG ($\delta_{\text{PG}}=1$) & 206.3 & 205.0 & 208.0 & 86.0 & 33.7 & 26.3 & 25.7 & 27.7 & 32.0 \\
 &  & MM & 3.7 & 4.0 & 3.3 & 4.0 & 4.3 & 4.0 & 4.0 & 4.0 & 4.0 \\ \cline{2-12} 
 & \multirow{4}{*}{\begin{tabular}[c]{@{}c@{}}Objective\\ Function\end{tabular}} & PG ($\delta_{\text{PG}}=100$) & 176.59 & 168.92 & 161.80 & 153.77 & 145.16 & 137.28 & 133.43 & 124.29 & 115.99 \\
 &  & PG ($\delta_{\text{PG}}=10$) & 176.58 & 168.91 & 161.79 & 153.77 & 145.16 & 137.28 & 133.44 & 124.29 & 115.99 \\
 &  & PG ($\delta_{\text{PG}}=1$) & 176.34 & 168.65 & 161.53 & 153.39 & 145.04 & 137.20 & 133.36 & 124.22 & 115.92 \\
 &  & MM & 176.59 & 168.92 & 161.80 & 153.77 & 145.16 & 137.28 & 133.44 & 124.29 & 116.00 \\ \hline
\multirow{12}{*}{$N=1000$} & \multirow{4}{*}{\begin{tabular}[c]{@{}c@{}}Running\\ Time (sec)\end{tabular}} & PG ($\delta_{\text{PG}}=100$) & 1200 & 1200 & 1200 & 1200 & 1200 & 1200 & 1200 & 1200 & 1200 \\
 &  & PG ($\delta_{\text{PG}}=10$) & 1200 & 1200 & 1200 & 1200 & 1200 & 1200 & 1200 & 1200 & 1200 \\
 &  & PG ($\delta_{\text{PG}}=1$) & 1200 & 1200 & 1200 & 1200 & 1200 & 1200 & 1200 & 1200 & 1200 \\
 &  & MM & 335.1 & 319.6 & 312.5 & 336.4 & 344.1 & 379.4 & 390.2 & 404.6 & 417.3 \\ \cline{2-12} 
 & \multirow{4}{*}{\begin{tabular}[c]{@{}c@{}}Num.\\ Iterations\end{tabular}} & PG ($\delta_{\text{PG}}=100$) & 15.0 & 14.0 & 13.3 & 14.7 & 14.3 & 14.3 & 15.0 & 14.3 & 14.3 \\
 &  & PG ($\delta_{\text{PG}}=10$) & 14.7 & 14.0 & 13.7 & 15.0 & 14.0 & 14.0 & 15.0 & 14.3 & 13.7 \\
 &  & PG ($\delta_{\text{PG}}=1$) & 14.3 & 14.0 & 13.7 & 14.7 & 14.0 & 13.7 & 15.0 & 14.3 & 14.0 \\
 &  & MM & 3.0 & 3.0 & 3.0 & 3.0 & 3.3 & 3.7 & 3.7 & 4.0 & 4.3 \\ \cline{2-12} 
 & \multirow{4}{*}{\begin{tabular}[c]{@{}c@{}}Objective\\ Function\end{tabular}} & PG ($\delta_{\text{PG}}=100$) & -177.75 & -39.25 & -199.40 & -195.41 & -126.87 & -175.56 & -377.55 & -165.09 & -205.60 \\
 &  & PG ($\delta_{\text{PG}}=10$) & -49.77 & 27.66 & -38.03 & -162.10 & 10.93 & 0.96 & -244.69 & -92.17 & -126.78 \\
 &  & PG ($\delta_{\text{PG}}=1$) & 99.93 & 105.72 & 40.19 & 38.91 & 60.40 & 0.60 & -96.88 & -20.30 & 23.39 \\
 &  & MM & 178.14 & 171.93 & 165.57 & 154.56 & 147.48 & 140.34 & 133.73 & 126.11 & 118.56\\
\bottomrule
\end{tabular}
}
\end{table}

The experimental results under the exponential demand model are presented in \Cref{tab:results-diff-theta-exponential-demand}, where we assume $\theta_i$ varies across $i$.
Although the exponential demand model requires an additional root-finding step rather than using a closed-form solution for $\boldsymbol{\lambda}^{(t)}$, the MM method still significantly outperforms the PG method in both running time and number of iterations. These results demonstrate that our main findings are robust to the choice of demand model.

\subsection{Simulation} \label{appendix:simulation}

We also conduct a simulation study to validate our pricing results in a dynamic stochastic matching environment. In the simulation, rider arrivals of each type follow independent Poisson processes, and their sojourn times are sampled from exponential distributions. All settings remain consistent with those in \Cref{subsec:chicago}.

We compare two price plans, by considering the LP model \eqref{equ:cn} with and without constraints \eqref{equ:cn-flow}, respectively. The first price plan serves as a baseline and is derived by solving~\eqref{equ:gn} with the constraint~\eqref{equ:cn-bound} relaxed in~\eqref{equ:cn}. This corresponds to assuming demands have infinite patience (i.e., $\theta_i=0$, $\forall i\in[N]$). As shown in \Cref{prop:theta_zero}, under this assumption, $c(\boldsymbol{\lambda})$ becomes linear, and maximizing $g(\boldsymbol{\lambda})$ becomes a convex optimization problem that can be solved directly using the first-order condition. The second price plan (i.e., our prices) considers the full LP program \eqref{equ:cn} with constraints \eqref{equ:cn-flow}. These prices are obtained from the MM algorithm (\Cref{alg:mm}). %
By comparing the two price plans, we can highlight the value of incorporating demands' finite patience into the LP formulation~\eqref{equ:cn}.

For the matching policy, we implement a dual-based approach, similar to the hindsight dual policy in Section 4.1 of \citet{ma2025potential}. For our pricing plan, after obtaining pricing solutions using the MM algorithm, we re-solve \eqref{equ:cn} to extract the optimal dual variables $\boldsymbol{\gamma}^*=(\gamma_i^*)_{i\in[N]}$ associated with constraints \eqref{equ:cn-flow}. Similarly, for the baseline pricing plan, after obtaining pricing solutions, we re-solve \eqref{equ:cn} without constraints \eqref{equ:cn-bound} to extract $\boldsymbol{\gamma}^*$. When a new rider of type $i\in[N]$ arrives, we attempt to match her with an available type-$j$ rider such that $c\ell_{(i,j)}-\gamma_i^*-\gamma_j^* \leq 0$ and this value is minimized among all available $j\in[N]$. Ties are broken in favor of the earliest-arrived rider. If no such match exists, the new rider is placed in the waiting queue.

\begin{table}[t] \centering \scriptsize
\caption{Average profit (\$/min) obtained over 150 simulation runs, when $\theta_1=\cdots=\theta_N=\theta$.} \label{tab:simulation-same-theta}
\resizebox{\textwidth}{!}{
\begin{tabular}{clllllllllllll}
\toprule
\multicolumn{2}{c}{$c$ (\$/mile)}                      & \multicolumn{4}{c}{$0.7$}     & \multicolumn{4}{c}{$0.9$}        & \multicolumn{4}{c}{$1.1$}              \\ \cmidrule(lr){3-6} \cmidrule(lr){7-10} \cmidrule(lr){11-14}  
\multicolumn{2}{c}{$\theta$ (min$^{-1}$)}              & $1/5$ & $1/3$ & $1$   & $2$   & $1/5$ & $1/3$ & $1$   & $2$      & $1/5$ & $1/3$    & $1$      & $2$      \\ \hline
\multirow{3}{*}{$N=100$} & LP w/o \eqref{equ:cn-bound} & 64.31 & 58.80 & 42.21 & 30.14 & 39.14 & 32.63 & 13.96 & 0.90 & 18.24 & 11.22 & -8.14 & -20.89 \\
 & LP w/ \eqref{equ:cn-bound} & 64.55 & 59.37 & 44.57 & 35.00 & 39.67 & 33.89 & 18.95 & 10.82 & 19.52 & 13.87 & 3.80 & 0.23 \\
 & Improvement (\%) & 0.4 & 1.0 & 5.6 & 16.1 & 1.3 & 3.9 & 35.7 & 1095.4 & 7.0 & 23.6 & $\infty$ & $\infty$ \\ \hline
\multirow{3}{*}{$N=200$} & LP w/o \eqref{equ:cn-bound} & 57.56 & 50.43 & 32.65 & 22.06 & 30.33 & 22.24 & 2.94 & -8.14 & 8.25 & -0.12 & -19.08 & -29.57 \\
 & LP w/ \eqref{equ:cn-bound} & 58.77 & 52.30 & 37.92 & 30.42 & 32.47 & 25.93 & 13.01 & 7.53 & 12.89 & 8.48 & 1.95 & -0.38 \\
 & Improvement (\%) & 2.1 & 3.7 & 16.1 & 37.9 & 7.0 & 16.6 & 341.8 & $\infty$ & 56.2 & $\infty$ & $\infty$ & $\infty$\\ \hline
 \multirow{3}{*}{$N=1000$} & LP w/o \eqref{equ:cn-bound} & 42.08 & 35.99 & 23.28 & 16.03 & 12.66 & 6.05 & -7.52 & -15.16 & -9.47 & -16.04 & -29.27 & -36.67 \\
 & LP w/ \eqref{equ:cn-bound} & 48.40 & 43.17 & 33.13 & 27.95 & 22.99 & 18.16 & 9.99 & 6.10 & 7.89 & 5.07 & 0.74 & -1.41 \\
 & Improvement (\%) & 15.0 & 19.9 & 42.3 & 74.4 & 81.5 & 200.3 & $\infty$ & $\infty$ & $\infty$ & $\infty$ & $\infty$ & $\infty$ \\
                         \bottomrule
\end{tabular}
}
\end{table}

\begin{table}[t] \centering \scriptsize
\caption{Average profit (\$/min) obtained over 150 simulation runs, when $\theta_i$ varies across $i$.} \label{tab:simulation-diff-theta}
\begin{tabular}{cllllllllll} \toprule
\multicolumn{2}{c}{$c$ (\$/mile)}                      & \multicolumn{3}{c}{$0.7$} & \multicolumn{3}{c}{$0.9$} & \multicolumn{3}{c}{$1.1$} \\ \cmidrule(lr){3-5} \cmidrule(lr){6-8} \cmidrule(lr){9-11}  
\multicolumn{2}{c}{$\theta$ (min$^{-1}$)} & $(1/5,1/3)$ & $(1/3,1)$ & $(1,2)$ & $(1/5,1/3)$ & $(1/3,1)$ & $(1,2)$ & $(1/5,1/3)$ & $(1/3,1)$ & $(1,2)$ \\ \hline
\multirow{3}{*}{$N=100$} & LP w/o \eqref{equ:cn-bound} & 61.62 & 49.81 & 35.78 & 35.98 & 22.45 & 6.99 & 14.78 & 0.51 & -14.94 \\
 & LP w/ \eqref{equ:cn-bound} & 62.04 & 51.19 & 39.40 & 36.82 & 25.45 & 14.48 & 16.60 & 7.72 & 1.67 \\
 & Improvement (\%) & 0.7 & 2.8 & 10.1 & 2.3 & 13.4 & 107.1 & 12.3 & 1419.0 & $\infty$ \\ \hline
\multirow{3}{*}{$N=200$} & LP w/o \eqref{equ:cn-bound} & 53.96 & 40.58 & 26.91 & 26.23 & 11.44 & -3.07 & 3.95 & -10.77 & -24.85 \\
 & LP w/ \eqref{equ:cn-bound} & 55.52 & 44.27 & 33.75 & 29.16 & 18.50 & 9.86 & 10.60 & 4.17 & 0.54 \\
 & Improvement (\%) & 2.9 & 9.1 & 25.4 & 11.2 & 61.8 & $\infty$ & 168.0 & $\infty$ & $\infty$\\ \hline
 \multirow{3}{*}{$N=1000$} & LP w/o \eqref{equ:cn-bound} & 38.90 & 28.75 & 19.36 & 9.17 & -1.66 & -11.67 & -12.86 & -23.64 & -33.25 \\
 & LP w/ \eqref{equ:cn-bound} & 45.57 & 37.28 & 30.30 & 20.46 & 13.37 & 7.88 & 6.53 & 2.44 & -0.53 \\
 & Improvement (\%) & 17.2 & 29.7 & 56.5 & 123.0 & $\infty$ & $\infty$ & $\infty$ & $\infty$ & $\infty$ \\
                         \bottomrule
\end{tabular}
\end{table}
\Cref{tab:simulation-same-theta} shows the simulation results when $\theta_1=\cdots=\theta_N=\theta$, and \Cref{tab:simulation-diff-theta} shows the results when $\theta_i$ differs across $i\in[N]$. Each result is averaged over 150 simulation runs, and the average profit per minute is reported. In both tables, ``LP w/o \eqref{equ:cn-bound}" refers to the baseline price plan, and ``LP w/ \eqref{equ:cn-bound}" denotes our price plan that accounts for finite rider patience. We also report the relative improvement in performance.

The results demonstrate that our price plan consistently outperforms the baseline. As $\theta$ increases---indicating lower rider patience---the performance gap between the two methods widens. Larger values of $N$ also result in a larger improvement under our method, demonstrating the benefit of modeling patience in more complex matching environments.

\end{ECSwitch}

\end{document}